\documentclass[12pt,a4paper]{amsart}
\usepackage{amssymb}

\calclayout

\allowdisplaybreaks

\newcommand{\+}{\nobreakdash-}
\renewcommand{\:}{\colon}
\renewcommand{\.}{\mskip.5\thinmuskip}

\renewcommand{\le}{\leqslant}
\renewcommand{\ge}{\geqslant}

\DeclareMathOperator{\Id}{Id}
\DeclareMathOperator{\id}{id}
\DeclareMathOperator{\Hom}{Hom}
\DeclareMathOperator{\Ext}{Ext}

\newcommand{\rarrow}{\longrightarrow}
\newcommand{\larrow}{\longleftarrow}
\newcommand{\mpsto}{\longmapsto}
\newcommand{\birarrow}{\rightrightarrows}
\newcommand{\eps}{\epsilon}
\newcommand{\ups}{\upsilon}
\newcommand{\kap}{\varkappa}
\renewcommand{\d}{\partial}
\newcommand{\ot}{\otimes}

\newcommand{\lrarrow}{\.\relbar\joinrel\relbar\joinrel\rightarrow\.}
\newcommand{\bu}{{\text{\smaller\smaller$\scriptstyle\bullet$}}}

\newcommand{\A}{{\mathcal A}}
\newcommand{\D}{{\mathcal D}}
\newcommand{\E}{{\mathcal E}}
\newcommand{\F}{{\mathcal F}}
\newcommand{\G}{{\mathcal G}}
\renewcommand{\H}{{\mathcal H}}
\newcommand{\I}{{\mathcal I}}
\newcommand{\J}{{\mathcal J}}
\renewcommand{\S}{{\mathcal S}}

\newcommand{\Z}{{\mathbb Z}}
\newcommand{\Q}{{\mathbb Q}}

\newcommand{\s}{{\mathfrak s}}
\renewcommand{\t}{{\mathfrak t}}

\newcommand{\gr}{{\mathrm{gr}}}

\newcommand{\Section}[1]{\bigskip\section{#1}\medskip}

\theoremstyle{plain}
\newtheorem{thm}{Theorem}[section]
\newtheorem{conj}[thm]{Conjecture}
\newtheorem{lem}[thm]{Lemma}
\newtheorem{prop}[thm]{Proposition}
\newtheorem{cor}[thm]{Corollary}
\theoremstyle{definition}
\newtheorem{rem}[thm]{Remark}
\newtheorem{ex}[thm]{Example}

\begin{document}

\title{Categorical Bockstein sequences}
\author{Leonid Positselski}

\address{Department of Mathematics, Faculty of Natural Sciences,
University of Haifa, Mount Carmel, Haifa 31905, Israel; and
\newline\indent  Laboratory of Algebraic Geometry, National Research
University Higher School of Economics, Moscow 119048; and
\newline\indent Sector of Algebra and Number Theory, Institute
for Information Transmission Problems, Moscow 127051, Russia; and
\newline\indent Mathematical Institute, Czech Academy of Sciences,
\v Zitn\'a~25, 115~67 Prague~1, Czech Republic}

\email{posic@mccme.ru}

\begin{abstract}
 We construct the reduction of an exact category with a twist functor
with respect to an element of its graded center in presence of
an exact-conservative forgetful functor annihilating this central
element.
 The construction uses matrix factorizations in a nontraditional way.
 We obtain the Bockstein long exact sequences for the Ext groups
in the exact categories produced by reduction.
 Our motivation comes from the theory of Artin--Tate motives and motivic
sheaves with finite coefficients, and our key techniques generalize
those of~\cite[Section~4]{Partin}.
\end{abstract}

\maketitle

\tableofcontents

\section*{Introduction}
\medskip

 The goal of this paper is to develop a general categorical framework
for the following problem.
 Let $G$ be a finite group.
 For any commutative ring~$k$, denote by $\F_k$ the category of
representations of $G$ in finitely generated free $k$\+modules.
 The category $\F_k$ has a natural exact category structure in which
a short sequence is exact if and only if it is exact as a sequence
of modules over $k[G]$, or equivalently, split exact as a sequence of
$k$\+modules.
 Let $m=l^r$ be a prime power.
 How does one recover the exact category of modular representations
$\F_{\Z/m}$ from the exact category $\F_{\Z_l}$ of representations
of $G$ over the $l$\+adic integers?

 Notice that the reduction functor $\rho\:\F_{\Z_l}\rarrow\F_{\Z/m}$
taking a free $\Z_l$\+module $M$ with an action of $G$ to the free
$\Z/m$\+module $\rho(M)=M/mM$ with the induced action of $G$ is
\emph{not} surjective on the isomorphism classes of objects.
 E.~g., for $m=l^2$ with an odd prime~$l$ and a cyclic group $G=\Z/l$,
the representation of $G$ in a free $\Z/l^2$\+module of rank~$1$
corresponding to a nontrivial character $\Z/l\rarrow (\Z/l^2)^*$
cannot be lifted to a representation of $G$ in a free $\Z_l$\+module
of rank~$1$.
 Similarly, for a prime number $m=l>3$ and a cyclic group $G=\Z/l$,
the representation of $G$ in a free $\Z/l$\+module of rank~$2$ given by
the matrix $\left(\begin{smallmatrix}1&1\\0&1\end{smallmatrix}\right)$
modulo~$l$ cannot be lifted to a representation of $G$ in a free
$\Z_l$\+module of rank~$2$.
 On the other hand, the regular representation of a finite group $G$
over the residue ring $\Z/m$ can, of course, be lifted to a regular
representation of $G$ over the ring~$\Z_l$.

 Neither is the functor~$\rho$ surjective on morphisms.
 Instead, for any two objects $M$ and $N\in\F_{\Z_l}$ there is a natural
\emph{Bockstein long exact sequence}
\begin{alignat*}{3}
 0&\lrarrow\Hom_{\F_{\Z_l}}(M,N)&&\lrarrow\Hom_{\F_{\Z_l}}(M,N)
 &&\lrarrow\Hom_{\F_{\Z/m}}(\rho(M),\rho(N)) \\
 &\lrarrow \Ext^1_{\F_{\Z_l}}(M,N)&&\lrarrow\Ext^1_{\F_{\Z_l}}(M,N)
 &&\lrarrow\Ext^1_{\F_{\Z/m}}(\rho(M),\rho(N)) \\
 &\lrarrow \Ext^2_{\F_{\Z_l}}(M,N)&&\lrarrow\Ext^2_{\F_{\Z_l}}(M,N)
 &&\lrarrow\Ext^2_{\F_{\Z/m}}(\rho(M),\rho(N))\lrarrow\dotsb
\end{alignat*}
 Moreover, given two prime powers $m'=l^s$ and $m''=l^t$ with
$m=m'm''$, there is a Bockstein long exact sequence
\begin{alignat*}{3}
 0&\rarrow\Hom_{\F_{\Z/m'}}(\rho'(M),\rho'(N))&&\rarrow
 \Hom_{\F_{\Z/m}}(M,N)&&\rarrow
 \Hom_{\F_{\Z/m''}}(\rho''(M),\rho''(N)) \\
 &\rarrow \Ext^1_{\F_{\Z/m'}}(\rho'(M),\rho'(N))&&\rarrow
 \Ext^1_{\F_{\Z/m}}(M,N)&&\rarrow
 \Ext^1_{\F_{\Z/m''}}(\rho''(M),\rho''(N)) \\
 &\rarrow\Ext^2_{\F_{\Z/m'}}(\rho'(M),\rho'(N))&&\rarrow
 \Ext^2_{\F_{\Z/m}}(M,N)&&\rarrow\dotsb
\end{alignat*}
for the reduction functors $\rho^{(i)}\:\F_{\Z/m}\rarrow\F_{\Z/m^{(i)}}$,
\,$i=1$,~$2$, and any two objects $M$, $N\in\F_{\Z/m}$.
 We would like to have such long exact sequences coming out from our
categorical formalism of reductions.

 More generally, let $G$ be a profinite group.
 Then it is more convenient to consider representations of $G$
in infinitely generated $k$\+modules, particularly when the ring
of coefficients~$k$ is itself infinite.
 So, let $\F_{\Z/m}^+$ denote the category of free $\Z/m$\+modules
endowed with a discrete action of~$G$.
 For the $l$\+adic coefficients, set $\F_{\Z/l}^+$ to be the category
of $l$\+divisible $l$\+primary torsion abelian groups (or, in
a different language, injective discrete $\Z_l$\+modules) with
a discrete action of~$G$.

 The reduction functor $\rho\:\F_{\Z_l}^+\rarrow\F_{\Z/m}^+$ acts by
assigning to an $l$\+divisible $l$\+primary abelian group $M$ with
a discrete action of $G$ the (co)free $\Z/m$\+module ${}_mM\subset M$
of all the elements of $M$ annihilated by~$m$, endowed with
the restriction of the action of~$G$.
 Our procedure of reduction of exact categories will produce
the exact category $\F_{\Z/m}^+$ with the reduction functor~$\rho$
starting from the exact category $\F_{\Z_l}^+$ endowed with
the natural tranformation of the identity endofunctor
$m\:\Id_{\F_{\Z_l}^+}\rarrow \Id_{\F_{\Z_l}^+}$
(acting by multiplication with~$m$ on every object
$M\in\F_{\Z_l}^+$) and some additional data.

 If one feels a bit uncomfortable about a reduction functor assigning
to an injective $l^\infty$\+torsion abelian group its subgroup of
elements annihilated by~$m=l^r$, one can restate the above description
of the category $\F_{\Z_l}^+$ in the language of free
$\Z_l$\+contramodules~\cite{Pweak,Prev,Pcta} rather than
cofree discrete $\Z_l$\+modules.
 The $\Z_l$\+module $\Z_l(G)$ of continuous $\Z_l$\+valued functions on
a profinite group $G$ is a \emph{$\Z_l$\+free coalgebra} in the sense
of~\cite[Subsections~1.6 and~3.1]{Pweak}.
 The category $\F_{\Z_l}^+$ can be equivalently defined as consisting of
free/projective $\Z_l$\+contramodules $P$ (i.~e., torsion-free
abelian groups for which the natural map $P\rarrow\varprojlim_n
P/l^nP$ is an isomorphism) endowed with a \emph{$\Z_l$\+free
$\Z_l(G)$\+comodule} structure~\cite[Subsections~3.1 and~3.3]{Pweak}
(cf.~\cite[Example in Subsection~3.1]{PosVish}).
 The reduction functor $\rho\:\F_{\Z_l}^+\rarrow\F_{\Z/m}^+$ then acts
by assigning the $\Z/m$\+free discrete $G$\+module $P/mP$ to
a $\Z_l$\+free $\Z_l(G)$\+comodule~$P$.
 The equivalence between the two definitions of the category
$\F_{\Z_l}^+$ is provided by the rules $P=\Psi_{\Z_l}(M)=
\Hom_\Z(\Q_l/\Z_l,M)$ and $M=\Phi_{\Z_l}(P)=\Q_l/\Z_l\ot_\Z P$
(see~\cite[Proposition~2.1]{Har} and~\cite[Subsection~1.5
and Proposition~3.3.2(b)]{Pweak}).

 In fact, the reduction construction presented in this paper is
applicable to a wider class of situations than the above discussion
may seem to suggest.
 Invented originally in the author's paper~\cite{Partin}, 
the first version of this reduction procedure was intended to
solve the following associated graded category problem, which
looks quite different from the above coefficient reduction questions
at the first glance.

 Let $l$~be a prime number and $G$ be a pro-$l$-group.
 We are interested in (say, finite-dimensional) discrete
$G$\+modules $M$ over the field $k=\Z/l$ endowed with a finite
decreasing filtration by $G$\+submodules
$\dotsb\supset F^{-1}M\supset F^0M\supset F^1M\supset\dotsb$
such that the action of $G$ is trivial on the quotient modules
$F^nM/F^{n+1}M$.
 The question is to define the structure induced on the associated
graded vector space $\bigoplus_n F^nM/F^{n+1}M$ by the $G$\+module
structure on~$M$.
 Let us first describe the answer algebraically, and then formulate
the problem in categorical terms.

 More generally, let $C$ be a coassociative coalgebra over
a field~$k$ and $0=F_{-1}C\subset F_0C\subset F_1C\subset F_2C
\subset\dotsb$ be a comultiplicative increasing filtration on~$C$.
 Set $F^{-i}C=F_iC$, and consider the category of finite-dimensional
left $C$\+comodules $M$ endowed with a decreasing filtration $F$
compatible with the filtration on~$C$.
 The above example with a pro-$l$-group $G$ is obtained by setting
$C=k(G)$ to be the group coalgebra (coalgebra of locally constant
$k$\+valued functions with the convolution comultiplication) of
the profinite group $G$ and $F_nC\subset C$ to be the components of
the coaugmentation filtration $F_nC=\ker(C\to (C/k)^{\ot n+1})$
(cf.~\cite[Section~2]{Partin}).

 In these terms, the induced structure on the graded vector space
$\gr_FM=\bigoplus_nF^nM/F^{n+1}M$ is simply that of a graded left
comodule over the graded coalgebra $\gr_FC=\bigoplus_n F^nC/F^{n+1}C$.
 The Ext spaces computed in the categories of filtered
$C$\+comodules and graded $\gr_FC$\+comodules are related
by a Bockstein long exact sequence (cf.\ the spectral sequence
in~\cite[proof of main theorem]{PosVish}).

 Now let $\F$ be the exact category of finite-dimensional left
$C$\+comodules $M$ endowed with finite decreasing filtrations $F$
compatible with the filtration $F$ on~$C$.
 The category $\F$ comes endowed with the \emph{twist functor}
$X\mpsto X(1)$ taking a filtered $C$\+comodule $(M,F)$ to the same
$C$\+comodule $M$ with the shifted filtration $F(1)^n M = F^{n-1}M$.
 One has $F^nM\subset F(1)^nM$ for all $M$ and~$n$; hence for every
filtered $C$\+comodule $X=(M,F)$ there is a natural morphism
$\s_X\:X\rarrow X(1)$ in the category~$\F$.
 The natural transformation $\s\:\Id_\F\rarrow(1)$ can be thought of
as an element of the \emph{graded center} of the category~$\F$ with
the twist functor $(1)\:\F\rarrow\F$.
 The reduction construction, applied to the exact category $\F$ with
the natural transformation~$\s$, should produce the exact (and, in
fact, in this case abelian) category $\G$ of finite-dimensional
graded modules over the graded coalgebra~$\gr_FC$.

 A solution to the latter ``categorical filtration reduction'' problem
was worked out in our previous paper~\cite{Partin}.
 The aim of the present work is to generalize the reduction construction
of~\cite[Section~4]{Partin} so as to make it also applicable to, e.~g.,
the ``categorical coefficient reduction'' problem described in
the beginning of this introduction.
 We also generalize the categorical Bockstein sequence construction
of~\cite[Section~4]{Partin}, formulating it in the abstract terms of
an exact category with a graded center element and two exact functors
to two other exact categories, satisfying appropriate conditions.
 This allows us to obtain the more complicated ``finite-finite-finite''
Bockstein long exact sequence for the Ext groups in our reduced
exact categories alongside with the simpler ``integral-integral-finite''
sequence.

 The construction of the reduced category $\G=\F/\s$ uses matrix
factorizations.
 Indeed, the diagrams $V(-1)\rarrow U\rarrow V\rarrow U(1)$ defining
objects of the intermediate category $\widetilde\H$ are nothing but
matrix factorizations of the natural transformation~$\s\:\Id\rarrow(1)$
on the category $\F$ (in the sense of, e.~g.,
\cite[Remark~2.7]{Psing}).
 However, the category-theoretic procedures that we apply to this
category of matrix factorizations are quite different from the ones
conventionally employed in the matrix factorization theory.
 Our aim is also quite different: while in the matrix factorization
theory as it is presently known one produces a \emph{triangulated}
category out of one's matrix factorizations, in this paper we
use matrix factorizations in order to produce an \emph{exact}
category (cf.~\cite[Remark~4.3]{Partin}).

 One of the differences is that our approach requires
an exact-conservative functor $\pi\:\F\rarrow\E$ annihilating all
the morphisms $\s_X\:X\rarrow X(1)$ (and the ones divisible by
these, but no other morphisms) to be given as an additional
piece of data.
 We start by considering the full subcategory $\H\subset\widetilde\H$
consisting of all the matrix factorizations that are transformed
into exact sequences in $\E$ by the functor~$\pi$
(cf.~\cite[Lemma~1.5]{PolVain}).
 This allows to construct the middle cocycles/coboundaries functor
$\Delta\:\H\rarrow\E$.

 We proceed by passing to the quotient category of the category $\H$
by the ideal $\I$ of all morphisms annihilated by the functor~$\Delta$.
 This serves to make the class $\S$ of all morphisms transformed to
isomorphisms by $\Delta$ a localizing class in the category
$\H/\I$, and we finally set $\G=(\H/\I)[\S^{-1}]$.
 Though this two-step procedure may remind one of (and was indeed
inspired by) the two-step construction of the derived category
starting from the category of complexes (with passing to the quotient
category by the ideal of morphisms homotopic to zero on the first
step), our present construction produces, to repeat, an exact
category and \emph{not} a triangulated one.

 Let us make a few more comments on the difference between the two
theories looking from a different angle.
 The conventional theory of matrix factorizations, from its inception
in the works of Eisenbud~\cite{Eis} and Buchweitz~\cite{Buch}, was
concerned with (local or global) \emph{singularities} of algebraic
varieties.
 When there were, in fact, no singularities, the theory would become
largely trivial.

 In the form the theory obtained in the work of Orlov~\cite{Or1,Or3}
(see also~\cite{PolVain}), the homotopy/derived category of
matrix factorizations was identified with the triangulated category
of singularities of the zero locus $X_0$ of a nonzero-dividing
global section $w\in L(X)$ of a line bundle $L$ on a smooth scheme~$X$.
 In its most advanced present form~\cite{Or3,Psing}, the theory
uses categories of matrix factorizations to describe relative
singularies (in one or another sense) of the Cartier divisor $X_0$
as compared to those of the whole (also singular) scheme~$X$.
 When the zero locus is, in fact, nonsingular (or if its singularities
are no worse than those of the ambient variety), the triangulated
categories of matrix factorizations vanish.

 To compare, consider an associative ring $R$ with a nonzero-dividing
central element $s\in R$, and let $\F$ be the exact category of all
left $R$\+modules $M$ without $s$\+torsion for which the natural map
$M\rarrow\varprojlim_n M/s^nM$ is an isomorphism (or, which is
equivalent for $R$\+modules without $s$\+torsion, one has
$\Ext^*_R(R[s^{-1}],M)=0$; cf.~\cite[Section~1 and Appendix~B]{Pweak}
and~\cite[Theorem~2.4]{Pcta}).
 The multiplication with~$s$ provides a natural transformation
$\s\:\Id_\F\rarrow\Id_\F$.
 Take $\E$ to be the abelian category of abelian groups (or $k$\+vector
spaces, if $R$ contains a field~$k$, etc.).
 One can use the functor assigning the abelian group/vector space $M/sM$
to an $s$\+complete $R$\+module without $s$\+torsion $M\in\F$
in the role of the background exact-conservative functor
$\pi\:\F\rarrow\E$.

 Applying our reduction procedure to this set of inputs, one obtains
the abelian category of left $R/(s)$\+modules $\G=\F/\s$ in the output.
 Notice that the $s$\+com\-plete\-ness condition is necessary for our
construction to work (or otherwise the functor~$\pi$ would not be
``exact-conservative'').
 On the other hand, we obtain the category of $R/(s)$\+modules
\emph{itself} in the result, rather than any category of singularities
of the quotient ring~$R/(s)$.
 Assuming the ring $R$ to be commutative and Noetherian (of finite Krull
dimension, etc.), and the ring $R/(s)$ to be regular, etc., does
\emph{not} make our reduction construction, or the exact/abelian
category produced by it, trivial (in any sense apparent to the author).
 Of course, this does not mean that the category of $R/(s)$\+modules
could not be obtained without our formalism.
 
 The importance of the reduction construction worked out in this paper
lies, in our view, in its wide applicability, including applicability
to complicated exact categories for which its outputs may be hard to
produce in any alternative or explicit way.
 In all the examples discussed above in this introduction, it was known
in advance what the reduced exact category $\G=\F/\s$ is supposed to be.
 The related Bockstein long exact sequences were not difficult to
obtain ``by hand'', either.
 This does not seem to be true for the examples that we are really
interested in, however.

 The latter mostly come from the theory of Artin--Tate motives and
motivic sheaves with finite coefficients~\cite{Partin,Pmotsh}.
 The construction of~\cite[Section~4]{Partin} was developed and
applied for producing the associated graded category $\G$ of the exact
category $\F$ of mixed Artin--Tate motives with finite coefficients
(and of the similar exact categories of filtered objects in a given
exact category with the successive quotients belonging to a given
additive subcategory endowed with the split exact structure).
 We do \emph{not} know of any other way to define this exact category
$\G$ (e.~g., as the category of graded modules or comodules over
anything, etc.).

 Our motivation for developing the construction presented below also
comes from the theory of Artin--Tate motivic sheaves with finite
coefficients.
 It was shown in the paper~\cite{Pmotsh} that the exact category
$\F_X^m$ of Artin--Tate motivic sheaves with the coefficients $\Z/m$
on a smooth algebraic variety $X$ (over a field of characteristic
prime to~$m$) has the Ext groups between the Tate objects
$\Ext^i_{\F_X^m}(\Z/m,\Z/m(j))$ agreeing with the motivic cohomology
groups of the variety $X$ (as described by the Beilinson--Lichtenbaum
conjectures) if and only if all the scheme points~$y$ of varieties
$Y$ \'etale over $X$ have the same property.
 Furthermore, for a field $K$ containing a primitive $m$\+root of
unity it was proven in~\cite{Partin} that one can express
this ``$K(\pi,1)$'' (Ext agreement) conjecture as the Koszul
property of a certain ``big graded algebra''.

 We would like to get rid of the root of unity assumption in the latter
theorem.
 This is straightforward when $m=l$ is a prime, because the cyclotomic
modules $\mu_m^{\ot j}$ over the Galois group $G_K$ are direct summands
of permutational modules in this case.
 So the plan is to reduce the $K(\pi,1)$\+conjecture for Artin--Tate
motives with $\Z/m$\+coefficients to the case of prime
coefficients~$\Z/l$.
 This is easily done (using the Koszul algebras language) when
the field $K$ contains a primitive $m$\+root of
unity~\cite[Subsections~7.3 and~9.5]{Partin}, but we
are interested in the opposite case.
 In the general situation, we want to use the Bockstein long exact
sequences relating the Ext groups in the categories $\F_K^m$ with
varying coefficients~$\Z/m$.

 The existence of such Bockstein exact sequences is itself a nontrivial
assertion, and we intend to prove it by constructing exact categories
$\G_K^m=\F_K^{\.l^\infty}/m$ in which the desired long exact sequences
are forced to hold first, and comparing the categories $\G_K^m$ with
the desired categories $\F_K^m$ later.
 So in this purported application we know in advance what the reduced
exact category $\G$ should be, but the Bockstein exact sequences
for the desired exact categories cannot be obtained directly, and
proving that the category $\G$ is what it is supposed to be is
a nontrivial task.
 In particular, it appears that the $K(\pi,1)$/Koszulity hypothesis
for prime coefficients $\Z/l$ may be \emph{needed} for obtaining
the Bockstein long exact sequences.

 At the very end of the paper, the intended main results are formulated
as two conjectures.

\medskip\noindent
\textbf{Acknowledgement.}
 The author's conception of the question of constructing
the quotient category of an exact category with a twist functor
by a natural transformation goes back to my years as a graduate student
at Harvard University in the second half of~'90s.
 My thinking was influenced by conversations with V.~Voevodsky\-
and A.~Beilinson at the time.
 The details below were worked out in Moscow in
February--March~2010 (as presented in~\cite[Section~4]{Partin})
and subsequently in September~2013 (in full generality).
 The paper was written while I~was vacationing in Prague
in March--April~2014, visiting Ben Gurion University of the Negev
in Be'er Sheva in June--September~2014, visiting the Technion
in Haifa in October 2014--March 2015, and working as a postdoc
at the University of Haifa in 2016--2018.
 I~am grateful to the anonymous referee for careful reading of
the manuscript and many insightful suggestions, which helped to
improve the exposition.
 In particular, the arguments in the paragraphs preceding
Lemma~\ref{ext-1-secondary-product} are largely due to the referee.
 The author was supported in part by RFBR grants in Moscow,
a fellowship from the Lady Davis Foundation at the Technion, and
the ISF grant~\#\,446/15 at the University of Haifa.

\setcounter{section}{-1}
\Section{Preliminaries}

\setcounter{subsection}{-1}
\subsection{Notation and terminology}  \label{conventions}
 Throughout this paper, most natural isomorphisms will be harmlessly
treated as identifications.
 We will also tacitly presume that a Grothendieck universe has been
chosen and can be enlarged if necessary, so words like ``a category'' or
``the category of modules'' are understood in the sense relative to
the chosen universe.
 This will allow us to perform various categorical constructions without
fear of problems arising from the foundations.

 By an \emph{exact category} we mean an exact category in Quillen's
sense, i.~e., an additive category endowed with a class of short exact
sequences satisfying the natural axioms (see, e.~g., \cite{Kel,Kel2},
\cite{Bueh}, or~\cite[Appendix~A]{Partin}).
 A sequence of objects and morphisms in an exact category is said to
be \emph{exact} if it is obtained by splicing short exact sequences.
 A functor between exact categories is called exact if it takes
short (or, equivalently, arbitrary) exact sequences in the source
category to short (resp., long) exact sequences in the target one.

 A \emph{twist functor} on a category $\F$ is an autoequivalence
denoted usually by $X\mpsto X(1)$.
 The inverse autoequivalence is denoted by $X\mpsto X(-1)$, and
the integral powers of the twist functor are denoted by
$X\mpsto X(n)$, \,$n\in\Z$.
 Twist functors on exact categories will be presumed to be exact
autoequivalences.

 Given two categories $\F$ and $\E$ endowed with twist functors,
a functor $\pi\:\F\rarrow\E$ is said to commute with the twists if
a functorial isomorphism $\pi(X(1))\simeq\pi(X)(1)$ is fixed for
all objects $X\in\F$.
 Speaking of a commutative diagram of functors $\F\rarrow\G\rarrow\E$
commuting with the twists, we will always presume that
the commutation isomorphisms form commutative diagrams of morphisms.

 A morphism of endofunctors $\t\:\Id\rarrow(n)$, \,$n\in\Z$ on
a category $\F$ with a twist functor $X\mpsto X(1)$ (i.~e., a morphism
$\t_X\:X\rarrow X(n)$ defined for every object $X\in\F$ and functorial
with respect to all the morphisms $X\rarrow Y$ in $\F$) is said to
\emph{commute with the twist} if for every object $X\in\F$ the equation
$\t_{X(1)}=\t_X(1)$ holds in the set $\Hom_\F(X(1),X(n+1))$.

 Notice that the endomorphisms of the identity functor on an additive
category $\F$ always form a commutative ring, which is called the
\emph{center} of the category~$\F$.
 It is the universal object among all the commutative rings~$k$ for
which $\F$ can be endowed with the structure of a $k$\+linear category.
 Similarly, given an additive category $\F$ with a twist functor
$X\mpsto X(1)$, morphisms of endofunctors $\Id\rarrow(n)$ commuting
with the twist form a commutative ring with a $\Z$\+grading, which
can be called the \emph{graded center} of an (additive) category
with a twist functor.

 We will say that a morphism $f\:X\rarrow Y$ in $\F$ is \emph{divisible
by} a natural transformation $\t\:\Id\rarrow(n)$ commuting with
the twist functor $X\mpsto X(1)$ on $\F$ if the morphism~$f$ factorizes
through the morphism $\t_X\:X\rarrow X(n)$, or equivalently, through
the morphism $\t_{Y(-n)}\:Y(-n)\rarrow Y$.
 Indeed, these conditions are equivalent, because for any morphism
$g\:X\rarrow Y(-n)$ one has $g(n)\t_X=\t_{Y(-n)}g$,
since $\t\:\Id\rarrow(n)$ is a natural transformation.
 Similarly, a morphism $f\:X\rarrow Y$ is \emph{annihilated by}
an element of the graded center $\t\:\Id\rarrow(n)$ of an additive
category $\F$ if the composition $X\rarrow Y\rarrow Y(n)$ vanishes,
or equivalently, the composition $X(-n)\rarrow X\rarrow Y$ vanishes
in~$\F$.

 Furthermore, suppose that two commuting autoequivalences
$X\mpsto X(1)$ and $X\mpsto X\{1\}$ are defined on a category~$\F$.
 Then one can consider morphisms of endofunctors $\t\:\Id_\F\rarrow
(n)\{m\}$ commuting with \emph{both} the twist functors $(1)$
and~$\{1\}$, i.~e., satisfying the equations
$\t_{X(1)}=\t_X(1)$ and $\t_{X\{1\}}=\t_X\{1\}$.
 For an additive category $\F$, such natural transformations form
a bigraded commutative ring, which can be called
the \emph{bigraded center} of~$\F$.

 More generally, let $\F$ and $\G$ be two categories with twist functors
$X\mpsto X(1)$ and $\rho\:\F\rarrow\G$ be a functor commuting with
the twists.
 Then a morphism of functors $\s\:\rho\rarrow\rho(n)$ acting between
the categories $\F$ and $\G$ is said to \emph{commute with the twists}
if for every object $X\in\F$ the equation $\s_{X(1)}=\s_X(1)$ holds
in the set $\Hom_\G(\rho(X)(1),\rho(X)(n+1))$.
 If this is the case and $f\:X\rarrow Y$ is a morphism in the category
$\F$, then the composition $\rho(X)\rarrow\rho(Y)\rarrow\rho(Y)(n)$
of the morphisms $\rho(f)$ and $\s_Y$, which is equal to the composition
$\rho(X)\rarrow\rho(X)(n)\rarrow\rho(Y)(n)$ of the morphisms~$\s_X$ and
$\rho(f)(n)$, is called the \emph{product} of the morphism~$\rho(f)$
with the natural transformation~$\s$ and denoted by $\s\rho(f)$.

 Given two objects $X$ and $Y$ in the category $\F$, a morphism
$g\:\rho(X)\rarrow\rho(Y)$ in the category $\G$ is said to be
\emph{divisible by} a natural transformation $\s\:\rho\rarrow\rho(n)$
commuting with the twists if it has the form $\s\rho(f)$ for a certain
morphism $f\:X\rarrow Y(-n)$ in the category $\F$, that is
the morphism~$g$ is equal to the composition $\rho(X)\rarrow\rho(Y)(-n)
\rarrow\rho(Y)$, or equivalently, to the composition
$\rho(X)\rarrow\rho(X)(n)\rarrow\rho(Y)$, where the morphisms
$\rho(X)\rarrow\rho(Y)(-n)$ and $\rho(X)(n)\rarrow\rho(Y)$ come
from morphisms in the category $\F$ via the functor~$\rho$.
 Similarly, given a morphism $f\:X\rarrow Y$ in the category $\F$,
the morphism $\rho(f)$ is said to be \emph{annihilated by} the natural
transformation~$\s$ if the product $\s\rho(f)$ vanishes, i.~e.,
the composition $\rho(X)\rarrow\rho(Y)\rarrow\rho(Y)(n)$ is equal
to zero, or equivalently, the composition $\rho(X)(-n)\rarrow\rho(X)
\rarrow\rho(Y)$ is equal to zero in the category~$\G$.

 An exact functor between exact categories $\pi\:\F\rarrow\E$ is called
\emph{exact-conservative} if it reflects admissible monomorphisms,
admissible epimorphisms, and exact sequences.
 In other words, a functor $\pi$ is said to be exact-conservative if
a morphism in $\F$ is an admissible monomorphism or admissible
epimorphism, or a complex in $\F$ is an exact sequence, if and only if
so is its image with respect to the functor~$\pi$ in
the exact category~$\E$.
 Notice that any exact-conservative functor between exact categories is
conservative in the conventional sense (i.~e., reflects isomorphisms).

\subsection{Exact surjectivity conditions}  \label{exact-surjectivity}
 Let $\eta\:\F\rarrow\G$ be an exact functor between two exact
categories.
 The following conditions on a functor~$\eta$ will play a key role
in the construction of the Bockstein long exact sequence in
Section~\ref{bockstein-sequence-section}:
\begin{itemize}
\item[(i$'$)] for any object $X\in\F$ and any admissible epimorphism
$T\rarrow\eta(X)$ in $\G$ there exists an admissible epimorphism
$Z\rarrow X$ in $\F$ and a morphism $\eta(Z)\rarrow T$ in $\G$ making
the triangle diagram $\eta(Z)\rarrow T\rarrow\eta(X)$ commutative;
\item[(i$''$)] for any object $Y\in\F$ and any admissible monomorphism
$\eta(Y)\rarrow T$ in $\G$ there exists an admissible monomorphism
$Y\rarrow Z$ in $\F$ and a morphism $T\rarrow\eta(Z)$ in $\G$ making
the triangle diagram $\eta(Y)\rarrow T\rarrow\eta(Z)$ commutative;
\item[(ii$'$)] for any objects $X$, $Y\in\F$ and any morphism
$\eta(X)\rarrow\eta(Y)$ in $\G$ there exists an admissible epimorphism
$X'\rarrow X$ and a morphism $X'\rarrow Y$ in $\F$ making the triangle
diagram $\eta(X')\rarrow\eta(X)\rarrow\eta(Y)$ commutative in~$\G$;
\item[(ii$''$)] for any objects $X$, $Y\in\F$ and any morphism
$\eta(X)\rarrow\eta(Y)$ in $\G$ there exists an admissible monomorphism
$Y\rarrow Y'$ and a morphism $X\rarrow Y'$ in $\F$ making the triangle
diagram $\eta(X)\rarrow\eta(Y)\rarrow\eta(Y')$ commutative in~$\G$.
\end{itemize}

 We will say that an exact functor~$\eta$ satisfies the condition~(i)
if both the dual conditions (i$'$) and (i$''$) hold for it.
 Similarly, we will say that $\eta$~satisfies the condition~(ii)
if both the dual conditions (ii$'$) and (ii$''$) hold for~$\eta$.
 For simple examples of exact functors satisfying
the conditions~(i\+ii), we refer the reader to
Subsection~\ref{bockstein-examples}.

 For the sake of completeness, consider also the following two easier
formulated conditions:
\begin{itemize}
\item[($*'$)] for any object $T\in\G$ there exists an object $U\in\F$
and an admissible epimorphism $\eta(U)\rarrow T$ in~$\G$;
\item[($*''$)] for any object $T\in\G$ there exists an object $V\in\F$
and an admissible monomorphism $T\rarrow\eta(V)$ in~$\G$.
\end{itemize}
 We will say that an exact functor~$\eta$ satisfies~($*$)
if it satisfies both~($*'$) and~($*''$).
 The following lemma will be useful in
Subsection~\ref{exact-surjectivity-compositions}, and also
relevant in Subsection~\ref{matrix-reduction-properties}.

\begin{lem}  \label{exact-surjectivity-implications}
 If the three conditions~($*'$), (i$\.'$), (ii$\.'$) hold for
an exact functor~$\eta$, then the morphism $\eta(Z)\rarrow T$
in~(i$\.'$) can be chosen to be an admissible epimorphism, too.
 Moreover, for a functor~$\eta$ reflecting admissible epimorphisms,
the condition~(i$\.'$) entirely follows from~($*'$) and~(ii$\.'$).
\end{lem}

\begin{proof}
 Indeed, let $T\rarrow\eta(X)$ be an admissible epimorphism in~$\G$.
 According to~($*'$), there exists an object $U\in\F$ together with
an admissible epimorphism $\eta(U)\rarrow T$ in~$\G$.
 The composition $\eta(U)\rarrow T\rarrow\eta(X)$ is a morphism
in $\G$ between two objects coming from~$\F$.
 According to~(ii$'$), there exists an admissible epimorphism
$U'\rarrow U$ and a morphism $U'\rarrow X$ in $\F$ making the
triangle diagram $\eta(U')\rarrow\eta(U)\rarrow\eta(X)$
commutative in~$\G$.
 Then the composition of two admissible epimorphisms $\eta(U')\rarrow
\eta(U)\rarrow T$ is an admissible epimorphism in $\G$, and
the triangle diagram $\eta(U')\rarrow T\rarrow\eta(X)$ is commutative.

 Now if the functor~$\eta$ reflects admissible epimorphisms, then
the morphism $U'\rarrow X$ is an admissible epimorphism in $\F$,
because its image in $\G$ is equal to the composition of admissible
epimorphisms $\eta(U')\rarrow T\rarrow\eta(X)$.
 In the general case, one applies~(i$')$ to find an admissible
epimorphism $Z\rarrow X$ in $\F$ whose image in $\G$ factorizes
through the admissible epimorphism $T\rarrow\eta(X)$.
 Then $U'\oplus Z\rarrow X$ is an admissible epimorphism in $\F$
whose image in $\G$ is the composition of the two admissible
epimorphisms $\eta(U')\oplus\eta(Z)\rarrow T\rarrow\eta(X)$.
\end{proof}

 Finally, the following pair of conditions on an exact functor
$\eta\:\F\rarrow\G$ will be needed in
Subsection~\ref{exact-surjectivity-compositions}:
\begin{itemize}
\item[($**'$)] For any object $X\in\F$ and any morphism
$\eta(X)\rarrow T$ in $\G$ there exists an admissible epimorphism
$X'\rarrow X$ in $\F$, a morphism $X'\rarrow S$ in $\F$, and
an admissible epimorphism $\eta(S)\rarrow T$ in $\G$ making
the square diagram $\eta(X')\rarrow\eta(X)\rarrow T$, \
$\eta(X')\rarrow\eta(S)\rarrow T$ commutative in~$\G$.
\item[($**''$)] For any object $Y\in\F$ and any morphism
$T\rarrow\eta(Y)$ in $\G$ there exists an admissible monomorphism
$Y\rarrow Y'$ in $\F$, a morphism $S\rarrow Y'$ in $\F$, and
an admissible monomorphism $T\rarrow\eta(S)$ in $\G$ making
the square diagram $T\rarrow\eta(Y)\rarrow\eta(Y')$, \
$T\rarrow\eta(S)\rarrow\eta(Y')$ commutative in~$\G$.
\end{itemize}
 We will say that an exact functor~$\eta$ satisfies~($**$) if
it satisfies both~($**'$) and~($**''$).

\subsection{Ext trivia}
 Here we formulate several elementary lemmas about the low-degree
Yoneda Ext groups in exact categories which will be useful in
Subsections~\ref{third-term-bockstein}\+-\ref{exactness-long-sequence}.
 The proofs are left to the reader.
 
\begin{lem}  \label{ext-1-equation}
 Let $z\in\Ext^1_\F(X,Y)$ and $w\in\Ext^1_\F(U,V)$ be two $\Ext^1$
classes in an exact category $\F$ represented by short exact
sequences $0\rarrow Y\rarrow Z\rarrow X\rarrow 0$ and
$0\rarrow V\rarrow W\rarrow U\rarrow 0$.
 Let $f\:Y\rarrow V$ and $g\:X\rarrow U$ be two morphisms in~$\F$.
 Then the equation $fz=wg$ holds in the group\/ $\Ext^1_\F(X,V)$
if and only if there exists a morphism $h\:Z\rarrow W$ in $\F$
extending the given collection of morphisms to a morphism of
short exact sequences (i.~e., a commutative diagram)
$(Y\to Z\to X)\rarrow(V\to W\to U)$ in~$\F$.  \qed
\end{lem}

\begin{lem}  \label{ext-2-square}
 Let $(K\to L\to M)\rarrow (X\to Y\to Z)\rarrow (U\to V\to W)$
be a $3\times3$ square commutative diagram in an exact category~$\F$,
all of whose three rows and three columns are short exact sequences
in~$\F$.
 Denote by $l\in\Ext^1_\F(M,K)$, \ $z\in\Ext^1_\F(W,M)$, \
$v\in\Ext^1_\F(W,U)$, and $x\in\Ext^1_\F(U,K)$ the\/ $\Ext^1$ classes
represented by the four short exact sequences along the perimeter.
 Then the equation $lz+xv=0$ holds in the group\/
$\Ext^2_\F(W,K)$.  \qed
\end{lem}

 Let $X\rarrow Y\rarrow Z$ be two morphisms with vanishing composition
in an exact category~$\F$.
 Assume that $X\rarrow Y$ is an admissible monomorphism with the cokernel
$C$ and $Y\rarrow Z$ is an admissible epimorphism with the kernel~$K$.
 Then the morphim $Y\rarrow Z$ decomposes as $Y\rarrow C\rarrow Z$ and
the morphism $X\rarrow Y$ decomposes as $X\rarrow K\rarrow Y$.

 If the category $\F$ is weakly idempotent complete (``semi-saturated'')
\cite{Neem}, \cite{Bueh}, \cite[Appendix~A]{Partin}, then it follows
that the morphism $C\rarrow Z$ is an admissible epimorphism and
the morphism $X\rarrow K$ is an admissible monomorphism.
 Applying the Snake lemma~\cite[Corollary~8.13]{Bueh} to the morphism of
short exact sequences $(X\to Y\to C)\rarrow(K\to Y\to Z)$, one concludes
that the kernel $H'$ of the morphism $C\rarrow Z$ is naturally
isomorphic to the cokernel $H''$ of the morphism $X\rarrow K$.
 The object $H'\simeq H''$ is called the \emph{cohomology object} of
the pair of morphisms $X\rarrow Y\rarrow Z$.

 When the category $\F$ is not weakly idempotent complete,
the cohomology object $H$ of the pair of morphisms $X\rarrow Y\rarrow Z$
is \emph{a priori} defined as an object of the weak idempotent
completion (``semi-saturation'') $\F'$ of the category $\F$
\cite[Remark~1.12]{Neem}, \cite[Remark~7.8]{Bueh},
\cite[Example~A.5(6)]{Partin}.
 One says that the cohomology object of the pair of morphisms $X\rarrow
Y\rarrow Z$ \emph{exists} in $\F$ if one has $H\in\F\subset\F'$.

 Notice that $\F'$ is naturally an exact category and $\F\subset\F'$ is
a full subcategory closed under extensions.
 Thus, when it can be shown that the object $H\in\F'$ is an extension of
two objects from $\F$, it follows that the cohomology object exists
in~$\F$.
 In particular, if $X'\rarrow X$ is an admissible monomorphism and
$Z\rarrow Z'$ is an admissible epimorphism in $\F$, and the cohomology
object of the pair of morphisms $X\rarrow Y\rarrow Z$ exists in $\F$,
so does the cohomology object of the pair of morphisms $X'\rarrow Y
\rarrow Z'$.

\begin{lem}  \label{ext-1-secondary-product}
 Let $0\rarrow K\rarrow L\rarrow M\rarrow 0$ and $0\rarrow U
\rarrow V\rarrow W\rarrow 0$ be two short exact sequences in
an exact category $\F$, and let $K\rarrow V$ and $L\rarrow W$
be two morphisms in $\F$ forming a commutative square diagram
$K\rarrow L\rarrow W$, $K\rarrow V\rarrow W$.
 Then\par
\textup{(a)} the cohomology object $H$ of the pair of mophisms with zero
composition $K\rarrow V\oplus L\rarrow W$ exists and the short sequence
$0\rarrow U\rarrow H\rarrow M\rarrow0$ of morphisms induced by
the morphisms $U\rarrow V$ and $L\rarrow M$ is exact in~$\F$; \par
\textup{(b)} the short exact sequence constructed in part~(a)
splits if and only if there exists a morphism $L\rarrow V$
in $\F$ making the triangle diagrams $K\rarrow L\rarrow V$ and
$L\rarrow V\rarrow W$ commutative; \par
\textup{(c)} when the morphism $K\rarrow V$ vanishes, so the morphism
$L\rarrow W$ factorizes through the admissible epimorphism $L\rarrow M$,
the extension class in\/ $\Ext^1_\F(M,U)$ represented by the short exact
sequence $(U\to H\to M)$ is equal to the product of the morphism
$M\rarrow W$ and the extension class in\/ $\Ext^1_\F(W,U)$ represented
by the short exact sequence $(U\to V\to W)$; \par
\textup{(d)} given a morphism of short exact sequences
$(K'\to L'\to M')\rarrow(K\to L\to M)$ in $\F$, and denoting
by $H'$ the cohomology object of the pair of morphisms
$K'\rarrow V\oplus L'\rarrow W$, the class in\/
$\Ext^1_\F(M',U)$ represented by the short exact sequence
$(U\to H'\to M')$ is equal to the product of the morphism
$M'\rarrow M$ and the class in\/ $\Ext^1_\F(M,U)$ represented by
the sequence $(U\to H\to M)$; \par
\textup{(e)} given two short exact sequences
$0\rarrow K\rarrow L\rarrow M\rarrow 0$ and $0\rarrow U
\rarrow V\rarrow W\rarrow 0$ as above and two morphisms of morphisms
(two commutative squares) $(K\to L)\birarrow(V\to W)$ in~$\F$,
the class in\/ $\Ext^1_\F(M,U)$ assigned to the sum of two morphisms
of morphisms by the construction of part~(a) is equal to the sum of
the classes assigned to the summands.  \qed
\end{lem}

\subsection{Exact surjectivity and the Ext groups}
 The proofs of the two parts of the next proposition can be found
in~\cite[Subsection~4.4]{Partin} (for a discussion of big graded
rings, see~\cite[Subsection~A.1]{Partin}).
 These results will play a key role in the arguments of
Subsections~\ref{higher-bockstein-construction}\+-%
\ref{exactness-long-sequence}.
 We denote by $\eta^n=\eta^n_{X,Y}\:\Ext^n_\F(X,Y)\rarrow
\Ext^n_\G(\eta(X),\eta(Y))$ the Ext group homomorphisms induced
by an exact functor $\eta\:\F\rarrow\G$.

\begin{prop} \label{exact-surjectivity-ext-prop}
 Let $\F$ and $\G$ be two exact categories and $\eta\:\F\rarrow\G$
be an exact functor satisfying the condition~(i$\.'$).  Then \par
\textup{(a)} for any objects $X$, $Y\in\F$ and $W\in\G$, and any
Ext classes $a\in\Ext^n_\F(X,Y)$ and\/ $b\in\Ext^m_\G(\eta(Y),W)$
such that\/ $b\eta^n(a)=0$ and\/ $m\ge1$, there exists an object
$Y'\in\F$, a morphism $f\:Y'\rarrow Y$ in $\F$, and a class $a'\in
\Ext^n_\F(X,Y')$ for which $a=fa'$ and\/ $b\eta(f)=0$; \par
\textup{(b)} for any object $W\in\G$, the right graded module\/
$(\Ext_\G^n(\eta(X),W))_{X\in\F;\,n\ge0}$ over the big graded ring\/
$(\Ext_\F^n(X,Y))_{Y,X\in\F;\,n\ge0}$ over the set of all objects of $\F$
is induced from the right module\/ $(\Hom_\G(\eta(Y),W))_{Y\in\F}$
over the big subring\/ $(\Hom_\F(X,Y))_{Y,X}\subset
(\Ext_\F^n(X,Y))_{Y,X;\,n}$. \qed
\end{prop}

 The following two corollaries allow one to prove that certain
exact functors between exact categories induce isomorphisms of
the Ext groups.
 They will be useful for us in
Subsection~\ref{reduction-independence}.

\begin{cor} \label{exact-surjectivity-ext-cor1}
 Let $\F$, $\G'$, $\G''$ be three exact categories, and let $\eta'\:
\F\rarrow\G'$ and $\iota\:\G'\rarrow\G''$ be two exact functors.
 Suppose that the composition of exact functors $\eta''=\iota\eta'\:
\F\rarrow\G''$ satisfies the condition~(i$\.'$) and the functor~$\iota$
induces isomorphisms $\iota\:\Hom_{\G'}(\eta'(X),W)\simeq
\Hom_{\G''}(\eta''(X),\iota(W))$ for all the objects $X\in\F$,
\,$W\in\G'$.  Then \par
\textup{(a)} the functor~$\iota$ induces isomorphisms
of Ext groups $\iota^n\:\Ext^n_{\G'}(\eta'(X),W)\simeq
\Ext^n_{\G''}(\eta''(X),\iota(W))$ for all the objects
$X\in\F$, \,$W\in\G'$ and all integers $n\ge0$; \par
\textup{(b)} if the functor~$\eta'$ satisfies the condition~($*'$),
then the functor~$\iota$ is fully faithful, its image $\iota(\G')$
is a full subcategory closed under extensions in $\G''$, the exact
category structure on $\G'$ coincides with the one induced from $\G''$
via~$\iota$, and the functor~$\iota$ induces isomorphisms of
the Ext groups $\iota^n\:\Ext^n_{\G'}(T,W)\simeq
\Ext^n_{\G''}(\iota(T),\iota(W))$ for all the objects $T$,
$W\in\G'$ and all\/~$n\ge0$.
\end{cor}

\begin{proof}
 It is straighforward to see that, in our assumptions on
the functor~$\iota$, a functor~$\eta'$ satisfies
the condition~(i$'$) whenever the functor~$\eta''$ does.
 Applying Proposition~\ref{exact-surjectivity-ext-prop}(b)
to both the functors~$\eta'$ and~$\eta''$ and comparing
the resulting descriptions of Ext groups in the categories
$\G'$ and $\G''$, one obtains the assertion of part~(a).
 
 All the assertions in the conclusion of part~(b) follow from
the last one.
 To deduce it from part~(a), it suffices to consider
a short exact sequence $0\rarrow S\rarrow\eta'(X)\rarrow T\rarrow0$
in $\G'$, the morphism between the related long exact sequences of
Ext groups induced by the functor~$\iota$, and argue
by induction in~$n$ using the 5\+lemma, proving the injectivity
first and then the surjectivity for every given~$n\ge0$.
\end{proof}

\begin{cor} \label{exact-surjectivity-ext-cor2}
 Let $\F$, $\G'$, $\G''$ be three exact categories, and let $\eta'\:
\F\rarrow\G'$ and $\iota\:\G'\rarrow\G''$ be two exact functors.
 Suppose that both functors~$\eta'$ and $\eta''=\iota\eta'\:
\F\rarrow\G''$ satisfy the condition~(i$\.'$), and the functor~$\iota$
induces isomorphisms $\iota\:\Hom_{\G'}(\eta'(X),\eta'(Y))\simeq
\Hom_{\G''}(\eta''(X),\eta''(Y))$ for all the objects $X$, $Y\in\F$.
 Then \par
\textup{(a)} the functor~$\iota$ induces isomorphisms of
Ext groups $\iota^n\:\Ext^n_{\G'}(\eta'(X),\eta'(Y))\simeq
\Ext^n_{\G''}(\eta''(X),\eta''(Y))$ for all the objects
$X$, $Y\in\F$ and all integers~$n\ge0$; \par
\textup{(b)} if the exact category $\G'$ coincides with its
minimal full subcategory containing all the objects in the image
of the functor~$\eta'$ and closed under extensions and direct
summands, or alternatively if the functor~$\eta'$ satisfies both
the conditions~($*'$) and~($*''$), then the conclusion of
Corollary~\ref{exact-surjectivity-ext-cor1}(b) also holds.
\end{cor}

\begin{proof}
 Part~(a) is provided by
Proposition~\ref{exact-surjectivity-ext-prop}(b)
as above, and part~(b) in its second set of assumptions follows
from part~(a) as above, by resolving both the objects $T$ and $W$
on the appropriate sides.
 Part~(b) in its first set of assumptions also follows by
the 5\+lemma (no induction in~$n$ needed in this situation).
\end{proof}

\subsection{Exact surjectivity in compositions}
\label{exact-surjectivity-compositions}
 Let $\F$, $\G$, $\H$ be three exact categories and
$\gamma\:\F\rarrow\G$, \,$\eta\:\G\rarrow\H$ be two exact functors.
 We are interested in proving that the functor~$\eta$ satisfies
the above conditions (i) and~(ii) provided that both
the functors~$\gamma$ and $\eta\gamma$ satisfy certain
(stronger) conditions.
 The next lemma will be important in Subsection~\ref{third-bockstein}.

\begin{lem}  \label{exact-surjectivity-compositions-lemma}
\textup{(a)} If the functor~$\eta\gamma$ reflects admissible
epimorphisms, admissible monomorphisms, or exact sequences, then
so does the functor~$\gamma$. 
 If the functor~$\eta\gamma$ satisfies the condition~($*'$),
then so does the functor~$\eta$. \par
\textup{(b)} If the functor~$\gamma$ satisfies the condition~($*'$)
and the functor~$\eta\gamma$ satisfies the condition~(i$\.'$),
then the functor~$\eta$ also satisfies the condition~(i$\.'$). \par
\textup{(c)} If the functor~$\gamma$ satisfies the condition~($*'$)
and the functor~$\eta\gamma$ satisfies the condition~($**'$), then
the functor~$\eta$ also satisfies the condition~($**'$). \par
\textup{(d)} If the functor~$\gamma$ satisfies the condition~($*'$),
and the functor~$\eta\gamma$ satisfies the conditions~(i$\.'$\+ii$\.'$),
($*'$\+-$\.**'$) and reflects admissible epimorphisms, then
the functor~$\eta$ satisfies the condition~(ii$\.'$).
\end{lem}

\begin{proof}
 Part~(a) is obvious.
 Part~(b): let $T\rarrow\eta(X)$ be an admissible epimorphism in $\H$
onto the image of an object $X\in\G$.
 Pick an admissible epimorphism $\gamma(U)\rarrow X$ in $\G$ from
the image of an object $U\in\F$, and denote by $W$ the fibered product
of $T$ and $\eta\gamma(U)$ over $\eta(X)$ in~$\H$.
 Then $W\rarrow\eta\gamma(U)$ is an admissible epimorphism in $\H$
onto the image of an object from $\F$, so there exists an admissible
epimorphism $Z\rarrow U$ in $\F$ and a morphism $\eta\gamma(Z)\rarrow W$
in $\H$ making the triangle $\eta\gamma(Z)\rarrow W\rarrow\eta\gamma(U)$
commutative in~$\H$.
 Now the composition $\gamma(Z)\rarrow\gamma(U)\rarrow X$ is
an admissible epimorphism in $\G$ whose image in $\H$ decomposes
as $\eta\gamma(Z)\rarrow W\rarrow T\rarrow\eta(X)$ and therefore
factorizes through the original admissible epimorphism
$T\rarrow\eta(X)$.
 Notice also that the morphism $\eta\gamma(Z)\rarrow T$ is
an admissible epimorphism in $\H$ whenever the morphism
$\eta\gamma(Z)\rarrow W$ is.

 Part~(c): let $\eta(X)\rarrow T$ be a morphism in $\H$ from the image
of an object $X\in\G$.
 Pick an admissible epimorphism $\gamma(U)\rarrow X$ in $\G$ and
consider the composition $\eta\gamma(U)\rarrow\eta(X)\rarrow T$ in~$\H$.
 Then there exists an admissible epimorphism $U'\rarrow U$ in $\F$,
a morphism $U'\rarrow S$ in $\F$, and an admissible epimorphism
$\eta\gamma(S)\rarrow T$ in $\H$ making the square diagram
$\eta\gamma(U')\rarrow\eta\gamma(U)\rarrow T$, \
$\eta\gamma(U')\rarrow\eta\gamma(S)\rarrow T$ commutative in~$\H$.
 Now the composition $\gamma(U')\rarrow\gamma(U)\rarrow X$ is
an admissible epimorphism in $\G$ whose image under~$\eta$
makes a commutative square diagram with the image of the morphism
$\gamma(U')\rarrow\gamma(S)$, the original morphism $\eta(X)\rarrow T$,
and the admissible epimorphism $\eta\gamma(S)\rarrow T$ in~$\H$.

 Part~(d): let $\eta(X)\rarrow\eta(Y)$ be a morphism in $\H$ between
two objects coming from~$\G$.
  Pick an admissible epimorphism $\gamma(U)\rarrow X$ in $\G$ and
consider the composition $\eta\gamma(U)\rarrow\eta(X)\rarrow\eta(Y)$
in~$\H$.
 According to the condition~($**'$), there exists an admissible
epimorphism $U'\rarrow U$ in $\F$, a morphism $U'\rarrow S$ in $\F$,
and an admissible epimorphism $\eta\gamma(S)\rarrow\eta(Y)$ in $\H$
such that the composition $\eta\gamma(U')\rarrow\eta\gamma(U)
\rarrow\eta(X)\rarrow\eta(Y)$ is equal to the composition
$\eta\gamma(U')\rarrow\eta\gamma(S)\rarrow\eta(Y)$ in~$\H$.
 According to part~(b), there exists an admissible epimorphism
$Z\rarrow Y$ in $\G$ and a morphism $\eta(Z)\rarrow\eta\gamma(S)$
in $\H$ making the triangle $\eta(Z)\rarrow\eta\gamma(S)
\rarrow\eta(Y)$ commutative in~$\H$.
 Moreover, in view of the remark after the above proof of part~(b)
and the first assertion of Lemma~\ref{exact-surjectivity-implications}
applied to the functor~$\eta\gamma$, one can choose the morphism
$\eta(Z)\rarrow\eta\gamma(S)$ to be an admissible epimorphism, too.

 Pick an admissible epimorphism $\gamma(V)\rarrow Z$ in $\G$
and consider the composition $\eta\gamma(V)\rarrow\eta(Z)\rarrow
\eta\gamma(S)$ in~$\H$.
  According to the condition~(ii$'$) for the functor~$\eta\gamma$,
there exists an admissible epimorphism $V'\rarrow V$ and
a morphism $V'\rarrow S$ in $\F$ such that the composition
$\eta\gamma(V')\rarrow\eta\gamma(V)\rarrow\eta\gamma(S)$ in $\H$
comes from the morphism $V'\rarrow S$ via the functor~$\eta\gamma$.
 The morphism $\eta\gamma(V')\rarrow\eta\gamma(S)$ being
a composition of three admissible epimorphisms in $\H$,
we can conclude that the morphism $V'\rarrow S$ is an admissible
epimorphism in~$\F$.
 Now let $W$ be the fibered product of the objects $V'$ and $U'$
over the object $S$ in the category~$\F$.
 Then the composition $\gamma(W)\rarrow\gamma(U')\rarrow\gamma(U)
\rarrow X$ is an admissible epimorphism in $\G$ whose image
under~$\eta$ composed with the original morphism
$\eta(X)\rarrow\eta(Y)$ in $\H$ is equal to the image of
the composition of morphisms $\gamma(W)\rarrow\gamma(V')
\rarrow\gamma(V)\rarrow Z\rarrow Y$ in~$\G$.
\end{proof}

\Section{The Bockstein Sequence}  \label{bockstein-sequence-section}

\setcounter{subsection}{-1}
\subsection{Toy version: two-category sequence}
\label{bockstein-toy}
 Let $\F$ and $\F_\s$ be two exact categories endowed with twist
functors (exact autoequivalences) $X\mpsto X(1)$.
 Suppose that we are given an exact functor $\eta_\s\:\F\rarrow\F_\s$
commuting with the twists.
 Assume that the functor~$\eta_\s$ satisfies
the conditions~(i\+ii) of Subsection~\ref{exact-surjectivity}.

 Suppose also that we are given a natural transformation
$\s\:\Id\rarrow(1)$ on the category $\F$ commuting with
the twist functor $(1)\:\F\rarrow\F$ as explained in
Subsection~\ref{conventions}.
 Assume the following further conditions to be satisfied:

\begin{itemize}
\item[(\i\i\i)] for any object $X\in\F$, the morphism $\s_X\:X
\rarrow X(1)$ is monic and epic; in other words, no nonzero
morphism in the category $\F$ is annihilated by the natural
transformation~$\s$;
\item[(\i\i\i\i)] a morphism in the category $\F$ is annihilated
by the functor~$\eta_\s$ if and only if it is divisible by
the natural transformation~$\s$.
\end{itemize}

 In this section we will construct, in the assumption of
conditions~(i\+ii) and~(\i\i\i\+\i\i\i\i), the following Bockstein
long exact sequence of Ext groups
\begin{alignat*}{3}
 0&\lrarrow\Hom_\F(X,Y(-1))&&\lrarrow\Hom_\F(X,Y)&&\lrarrow
 \Hom_{\F_\s}(\eta_\s(X),\eta_\s(Y)) \\
 &\lrarrow \Ext^1_\F(X,Y(-1))&&\lrarrow\Ext^1_\F(X,Y)&&\lrarrow
 \Ext^1_{\F_\s}(\eta_\s(X),\eta_\s(Y)) \\
 &\lrarrow\Ext^2_\F(X,Y(-1))&&\lrarrow\Ext^2_\F(X,Y)&&\lrarrow
 \Ext^2_{\F_\s}(\eta_\s(X),\eta_\s(Y))\lrarrow\dotsb
\end{alignat*}
for any two objects $X$, $Y\in\F$.
 The differentials in this long exact sequence have the following
properties:
\begin{enumerate}
\renewcommand{\theenumi}{\alph{enumi}}
\item the maps $\eta_\s=\eta_\s^n\:\Ext^n_\F(X,Y)\rarrow
\Ext^n_{\F_\s}(\eta_\s(X),\eta_\s(Y))$ are induced by the exact
functor~$\eta_\s\:\F\rarrow\F_\s$;
\item the maps $\s=\s_n\:\Ext^n_\F(X,Y(-1))\rarrow
\Ext^n_\F(X,Y)$ are provided by the composition with
the natural transformation $\s\:\Id_\F\rarrow(1)$;
\item the maps $\d=\d^n\:\Ext^n_{\F_\s}(\eta_\s(X),\eta_\s(Y))
\rarrow\Ext^{n+1}_\F(X,Y(-1))$ satisfy
the equation $$\d^{i+n+j}(\eta_\s^i(a)z\eta_\s^j(b))=(-1)^i
a(-1)\d^n(z)b$$ for any objects
$U$, $X$, $Y$, $V\in\F$ and any Ext classes $b\in\Ext^j_\F(U,X)$, \
$z\in\Ext^n_{\F_\s}(\eta_\s(X),\eta_\s(Y))$, and $a\in\Ext^i_\F(Y,V)$.
 Here $(-1)^i$ in the right-hand side denotes a plus or minus sign
depending on the parity of the integer~$i$, while
$a(-1)\in\Ext^i_\F(Y(-1),V(-1))$ is the inverse twist of
the Ext class~$a$.
\end{enumerate}

 This long exact sequence is a more immediate generalization of
the long exact sequence of~\cite[Section~4]{Partin} than
the more elaborated constructions of Subsections~\ref{bockstein-posing}
and~\ref{bockstein-generalization} below.
 It is also essentially the particular case of the long exact sequence
of the next Subsection~\ref{bockstein-posing} corresponding to
the situation with $\F_\t=\F_{\s\t}=\F$ and the identity functor
$\eta_\t=\Id_\F$ (or, if one wishes, the particular case of
the even more general long exact sequence of
Subsection~\ref{bockstein-generalization} corresponding to
the situation with $\F_\t=\F=\F_{\s\t}$ and the identity functors
$\eta_\t=\Id_\F=\eta_{\s\t}$).

 Notice that in presence of the above condition~(\i\i\i)
the condition~(iv) of Subsection~\ref{bockstein-posing} becomes
equivalent to its apparently stronger form~(\i\i\i\i).
 One can see this, e.~g., by comparing the initial three-term
fragments of the Bockstein sequences in this subsection and in
Subsection~\ref{bockstein-posing}
(cf.\ Lemmas~\ref{bockstein-s-zero}(d)
and~\ref{bockstein-r-zero}(c) below).

\subsection{Posing the problem: three-category sequence}
\label{bockstein-posing}
 Let $\F_\t$, \,$\F_\s$, and $\F_{\s\t}$ be three exact categories
endowed with twist functors (exact autoequivalences) $X\mpsto X(1)$.
 Suppose that we are given exact functors
$\eta_\t\:\F_{\s\t}\rarrow\F_\t$ and $\eta_\s\:\F_{\s\t}\rarrow\F_\s$,
both commuting with the twists.
 Assume that both the functors $\eta_\t$ and~$\eta_\s$ satisfy
the conditions~(i\+ii) of Subsection~\ref{exact-surjectivity}.

 Furthermore, suppose that we are given a natural transformation
$\s\:\Id\rarrow(1)$ on the category $\F_{\s\t}$ commuting with the twist
functor $(1)\:\F_{\s\t}\rarrow\F_{\s\t}$ as explained in
Subsection~\ref{conventions}.
 We assume the following further conditions to be satisfied:
\begin{enumerate}
\renewcommand{\theenumi}{\roman{enumi}}
\setcounter{enumi}{2}
\item a morphism $X\rarrow Y$ in the category $\F_{\s\t}$ is annihilated
by the functor~$\eta_\t$ if and only if it is annihilated by
the natural transformation~$\s$ in~$\F_{\s\t}$;
\item a morphism $X\rarrow Y$ in the category $\F_{\s\t}$ is annihilated
by the functor~$\eta_\s$ if and only if there exists an admissible
epimorphism $X'\rarrow X$ such that the composition $X'\rarrow X
\rarrow Y$ is divisible by~$\s$ in $\F_{\s\t}$, or equivalently, if and
only if there exists an admissible monomorphism $Y\rarrow Y'$ such that
the composition $X\rarrow Y\rarrow Y'$ is divisible by~$\s$
in~$\F_{\s\t}$.
\end{enumerate}

 We will see below in Subsection~\ref{first-term-bockstein} that
the two dual formulations of the condition~(iv) are equivalent modulo
our previous assumptions (specifically, the argument is based on
the condition~(ii) for the functor~$\eta_\t$ and the condition~(iii)).

 Our goal in this section is to construct, in the assumption of
the conditions~(i\+iv), the following Bockstein long exact sequence
for the Ext groups
\begin{alignat*}{3}
 0&\lrarrow\Hom_{\F_\t}(\eta_\t(X),\eta_\t(Y)(-1))&&\lrarrow
 \Hom_{\F_{\s\!\.\t}}(X,Y)&&\lrarrow
 \Hom_{\F_\s}(\eta_\s(X),\eta_\s(Y)) \\
 &\lrarrow \Ext^1_{\F_\t}(\eta_\t(X),\eta_\t(Y)(-1))&&\lrarrow
 \Ext^1_{\F_{\s\!\.\t}}(X,Y)&&\lrarrow
 \Ext^1_{\F_{\s}}(\eta_\s(X),\eta_\s(Y)) \\
 &\lrarrow\Ext^2_{\F_\t}(\eta_\t(X),\eta_\t(Y)(-1))&&\lrarrow
 \Ext^2_{\F_{\s\!\.\t}}(X,Y)&&\lrarrow\dotsb
\end{alignat*}
for any two objects $X$, $Y\in\F_{\s\t}$.
 The differentials in this long exact sequence have the following
properties:
\begin{enumerate}
\renewcommand{\theenumi}{\alph{enumi}}
\item the maps $\eta_\s=\eta_\s^n\:\Ext^n_{\F_{\s\!\t}}(X,Y)\rarrow
\Ext^n_{\F_\s}(\eta_\s(X),\eta_\s(Y))$ are induced by the exact
functor~$\eta_\s\:\F_{\s\t}\rarrow\F_\s$;
\item the maps $\s=\s_n\:\Ext^n_{\F_\t}(\eta_\t(X),\eta_\t(Y)(-1))
\rarrow\Ext^n_{\F_{\s\!\t}}(X,Y)$ satisfy the equations
$\s_0(\id_{\eta_\t(E)})=\s_E\in\Hom_{\F_{\s\!\t}}(E,E(1))$ and
$$
 \s_{i+n+j}(\eta_\t^i(a)(-1)z\eta_\t^j(b))=a\s_n(z)b
$$ for any
objects $E$, $U$, $X$, $Y$, $V\in\F_{\s\t}$ and any Ext classes
$b\in\Ext^j_{\F_{\s\!\t}}(U,X)$, \
$z\in\Ext^n_{\F_\t}(\eta_\t(X),\eta_\t(Y)(-1))$, and
$a\in\Ext^i_{\F_{\s\!\t}}(Y,V)$;
\item the maps $\d=\d^n\:\Ext^n_{\F_\s}(\eta_\s(X),\eta_\s(Y))
\rarrow\Ext^{n+1}_{\F_\t}(\eta_\t(X),\eta_\t(Y)(-1))$ satisfy
the equation $$\d^{i+n+j}(\eta_\s^i(a)z\eta_\s^j(b))=(-1)^i
\eta_\t^i(a)(-1)\d^n(z)\eta_\t^j(b)$$ for any objects
$U$, $X$, $Y$, $V\in\F_{\s\t}$ and any Ext classes
$b\in\Ext^j_{\F_{\s\!\t}}(U,X)$, \
$z\in\Ext^n_{\F_\s}(\eta_\s(X),\eta_\s(Y))$, and
$a\in\Ext^i_{\F_{\s\!\t}}(Y,V)$.
 Here $(-1)^i$ in the right-hand side denotes a plus or minus sign
depending on the parity of~$i$, while $\eta_\t^i(a)(-1)\in
\Ext^i_{\F_\t}(\eta_\t(Y)(-1),\eta_\t(V)(-1))$ is the inverse twist of
the Ext class~$\eta_\t^i(a)$.
\end{enumerate}

\subsection{Examples}  \label{bockstein-examples}
 The ``real'' examples of exact functors $\eta_\s\:\F\rarrow\F_\s$ and
natural transformations $\s\:\Id_\F\rarrow(1)$ satisfying together
the conditions~(i\+\i\i\i\i), as well as those of pairs of exact
functors $\eta_\s$, $\eta_\t\:\F_{\s\t}\rarrow\F_\s$, $\F_\t$ and
natural transformations $\s\:\Id_{\F_{\s\!\t}}\rarrow(1)$ satisfying
the conditions~(i\+iv), will appear in connection with
the reduction construction of
Section~\ref{matrix-factor-construct-section}, where the results
of the present section will be applied (see
Subsections~\ref{first-bockstein} and~\ref{third-bockstein};
cf.\ Subsection~\ref{second-bockstein}, where the even more
general setting of Subsection~\ref{bockstein-generalization}
below will appear).
 Nevertheless, let us present two very simple explicit
examples here.

\begin{ex}  \label{finite-group-finite-finite-finite-example}
 Let $G$ be a finite group, $l$~be a prime number, and $s$, $t\ge2$
be any two powers of~$l$.
 For any prime power $m\ge2$, let $\F_{\Z/m}=\F_{\Z/m}^G$ be
the exact category of finitely generated free $\Z/m\.$\+modules
with an action of~$G$.

 Set $\F_\s=\F_{\Z/s}$, \,$\F_\t=\F_{\Z/t}$, and $\F_{\s\t}=\F_{\Z/st}$.
 Let the functors $\eta_\s\:\F_{\s\t}\rarrow\F_\t$ and
$\eta_\t\:\F_{\s\t}\rarrow\F_\t$ take a $\Z/st\.$\+free $G$\+module
$M$ to the $\Z/s\.$\+free and $\Z/t\.$\+free $G$\+modules
$M/sM=tM$ and $M/tM=sM$, respectively.
  Set all the twists $(1)$ to be the identity functors, and
the natural transformation $\s\:\Id_{\F_{\s\!\t}}\rarrow
\Id_{\F_{\s\!\t}}$ to act on all the $\Z/st\.$\+free $G$\+modules
by the operator of multiplication with~$s$.

 Then we claim that all the conditions~(i\+iv) of
Subsections~\ref{exact-surjectivity} and~\ref{bockstein-posing}
are satisfied for the exact functors~$\eta_\s$, $\eta_\t$ and
the center element~$\s$.
 The assertion is also true for the exact categories of arbitrary
(infinitely generated) free $\Z/m\.$\+modules with an action of~$G$.
 Indeed, the functor~$\eta_\s$ is exact-conservative, since so is
the forgetful functor $\F^G_{\Z/st}\rarrow\F^{\{e\}}_{\Z/st}$ and
the reduction functor $\F^{\{e\}}_{\Z/st}\rarrow\F^{\{e\}}_{\Z/s}$
acting between the categories of free modules over $\Z/st$ and $\Z/s$
without any group action.

 To check the conditions ($*'$) and~($*''$) for the functor~$\eta_\s$,
one simply notices that any $\Z/m$\+free $G$\+module can be presented
as the image of a surjective homomorphism from, and embedded into,
a (co)free $G$\+module over $\Z/m$ (i.~e., a direct sum of copies
of $\Z/m[G]=\Z/m(G)$).
 Furthermore, (co)free $G$\+modules over $\Z/s$ can be obtained as
the reductions of similar $G$\+modules over~$\Z/st$.
 Using the projectivity/injectivity properties of (co)free
$G$\+modules, one can also check the conditions (ii$'$) and~(ii$''$),
as well as the other conditions of Subsection~\ref{exact-surjectivity}
for the functor~$\eta_\s$.

 The condition~(iii) is obvious: a morphism $f\:M\rarrow N$ in
the category $\F_{\Z/st}$ is annihilated by the functor~$\eta_\t$
if and only if its image is contained in $tN$ (or its kernel
contains~$sM$), which equivalently means that $sf=0$.
 Finally, to check the condition~(iv) one notices that any morphism
$f\:M\rarrow N$ in $\F_{\Z/st}$ with the image contained in $sN$
is divisible by~$s$ in the group $\Hom_{\F_{\Z/st}}(M,N)$ whenever at
least one of the $G$\+modules $M$ and $N$ is (co)free over~$\Z/st$.
\end{ex}

\begin{ex}  \label{finite-group-integral-integral-finite-example}
 Let $G$ be a finite group and $m=l^r$ be a prime power.
 Keeping the notation $\F_{\Z/m}=\F^G_{\Z/m}$ for the exact category of
finitely generated free $\Z/m\.$\+modules with an action of $G$,
set also $\F_{\Z_l}=\F^G_{\Z_l}$ to be the category of finitely
generated free $\Z_l$\+modules with a $G$\+action.
 Set $\F=\F_{\Z_l}$ and $\F_\s=\F_{\Z/m}$.

 Let the functor $\eta_\s\:\F\rarrow\F_\s$ take a $\Z_l$\+free
$G$\+module $M$ to the $\Z/m$\+free $G$\+module $M/mM$.
 Set all the twists $(1)$ to be the identity functors, and
the map $\s\:M\rarrow M$ to be the multiplication with~$m$ for
every module $M\in\F=\F_{\Z_l}$.
 Then it is claimed that all the conditions~(i\+\i\i\i\i) of
Subsections~\ref{exact-surjectivity} and~\ref{bockstein-toy} are
satisfied for the exact functor~$\eta_\s$ and the natural
transformation~$\s$.

 The functor~$\eta_\s$ is exact-conservative, since so are
the forgetful functor $\F^G_{\Z_l}\rarrow\F^{\{e\}}_{\Z_l}$ and
the reduction functor $\F^{\{e\}}_{\Z_l}\rarrow\F^{\{e\}}_{\Z/m}$.
 The ``exact surjectivity'' conditions of
Subsection~\ref{exact-surjectivity} can be proven in the same way
as in Example~\ref{finite-group-finite-finite-finite-example}.
 To compensate for an apparent loss of symmetry between
the injectivity and projectivity properties of the (co)free
$G$\+modules $\Z_l[G]=\Z_l(G)$ with $l$\+adic coefficients,
one can use the alternative interpretation of $\F_{\Z_l}$ as
the category of $l$\+divisible $l^\infty$\+torsion abelian
groups of finite rank endowed with an action of~$G$.

 The condition~(\i\i\i) is obvious: no nonzero morphism in the category
$\F_{\Z_l}^G$ is annihilated by the multiplication with~$m$ (since
the same is true in the category $\F_{\Z_l}^{\{e\}}$).
 The condition~(\i\i\i\i) holds, since any morphism
$f\:M\rarrow N$ in the category $\F_{\Z_l}^G$ that is annihilated by
the functor~$\eta_\s$ (i.~e., has the image contained in~$mN$)
is divisible by~$m$ in the group $\Hom_{\F_{\Z_l}^G}(M,N)$.
\end{ex}

 Applying the categorical Bockstein long exact sequence construction
of this section to these two examples, one obtains
the ``finite-finite-finite'' and ``integral-integral-finite''
Bockstein sequences for the Ext groups in the categories of finite group
representations written down in the beginning of the introduction.

\smallskip
 Examples of explicit sets of data that can be straightforwardly
shown to satisfy ``a~half of'' the conditions~(i\+\i\i\i\i) or~(i\+iv)
are more numerous, and one can easily work out some of them on the basis
of arguments similar to the above.

\subsection{Further generalization: four-category sequence}
\label{bockstein-generalization}
 Let $\F$, $\F_\t$, $\F_\s$, and $\F_{\s\t}$ be four exact categories
endowed with twist functors $X\rarrow X(1)$.
 Suppose we are given exact functors $\eta_\t\:\F\rarrow\F_\t$, \
$\eta_\s\:\F\rarrow\F_\s$, and $\eta_{\s\t}\:\F\rarrow\F_{\s\t}$,
all of them commuting with the twists.
 Assume that all the three functors $\eta_\t$, $\eta_\s$, $\eta_{\s\t}$
satisfy the conditions~(i\+ii) of Subsection~\ref{exact-surjectivity}.

 Suppose further that we are given a natural transformation
$\s\:\eta_{\s\t}\rarrow\eta_{\s\t}(1)$ of functors $\F\rarrow\F_{\s\t}$
which commutes with the twist functors $(1)\:\F\rarrow\F$ and
$(1)\:\F_{\s\t}\rarrow\F_{\s\t}$ in the sense of
Subsection~\ref{conventions}.
 Assume that the following further conditions are satisfied
by this set of data:

\begin{enumerate}
\renewcommand{\theenumi}{\Roman{enumi}}
\setcounter{enumi}{2}
\item a morphism $f\:X\rarrow Y$ in the category $\F$ is annihilated
by the functor~$\eta_\t$ if and only if the morphism $\eta_{\s\t}(f)$
is annihilated by the natural transformation~$\s$;
\item a morphism $f\:X\rarrow Y$ in the category $\F$ is annihilated
by the functor~$\eta_\s$ if and only if there exists an admissible
epimorphism $X'\rarrow X$ in $\F$ such that the composition
$\eta_{\s\t}(X')\rarrow\eta_{\s\t}(X)\rarrow\eta_{\s\t}(Y)$ is divisible
by the natural transformation~$\s$, or equivalently, if and only if
there exists an admissible monomorphism $Y\rarrow Y'$ in $\F$ such that
the composition $\eta_{\s\t}(X)\rarrow\eta_{\s\t}(Y)\rarrow
\eta_{\s\t}(Y')$ is divisible by~$\s$.
\end{enumerate}

 We will see below in Subsection~\ref{first-term-bockstein} that
the two dual formulations of the condition~(IV) are equivalent modulo
the previous assumptions (the argument is based on
the condition~(ii) for the functor~$\eta_\t$ and the condition~(III)).

 Assuming the conditions~(i\+ii) and~(III\+-IV), we will construct
a Bockstein long exact sequence of the form {\hfuzz=3pt
\begin{alignat*}{3}
 0&\rarrow\Hom_{\F_\t}(\eta_\t(X),\eta_\t(Y)(-1))&&\rarrow
 \Hom_{\F_{\s\!\.\t}}(\eta_{\s\t}(X),\eta_{\s\t}(Y))&&\rarrow
 \Hom_{\F_\s}(\eta_\s(X),\eta_\s(Y)) \\
 &\rarrow \Ext^1_{\F_\t}(\eta_\t(X),\eta_\t(Y)(-1))&&\rarrow
 \Ext^1_{\F_{\s\!\.\t}}(\eta_{\s\t}(X),\eta_{\s\t}(Y))&&\rarrow
 \Ext^1_{\F_{\s}}(\eta_\s(X),\eta_\s(Y)) \\
 &\rarrow\Ext^2_{\F_\t}(\eta_\t(X),\eta_\t(Y)(-1))&&\rarrow
 \Ext^2_{\F_{\s\!\.\t}}(\eta_{\s\t}(X),\eta_{\s\t}(Y))&&\rarrow\dotsb
\end{alignat*}
for any} two objects $X$, $Y\in\F$.
 The differentials in this long exact sequence have the following
properties:
\begin{enumerate}
\renewcommand{\theenumi}{\alph{enumi}}
\item the maps $r=r_\t^n\:
\Ext^n_{\F_{\s\!\t}}(\eta_{\s\t}(X),\eta_{\s\t}(Y))\rarrow
\Ext^n_{\F_\s}(\eta_\s(X),\eta_\s(Y))$ satisfty the equations
$$
 r_\t^n(\eta_{\s\t}^n(a))=\eta_\s^n(a)
 \quad\text{and}\quad
 r_\t^{i+j}(xy)=r_\t^i(x)r_\t^j(y)
$$
for any objects $X$, $Y$, $U$, $V$, $W\in\F$ and any Ext classes
$a\in\Ext^n_\F(X,Y)$, \
$y\in\Ext^j_{\F_{\s\!\t}}(\eta_{\s\t}(U),\eta_{\s\t}(V))$, and 
$x\in\Ext^i_{\F_{\s\!\t}}(\eta_{\s\t}(V),\eta_{\s\t}(W))$;
\item the maps $\s=\s_n\:\Ext^n_{\F_\t}(\eta_\t(X),\eta_\t(Y)(-1))
\rarrow\Ext^n_{\F_{\s\!\t}}(\eta_{\s\t}(X),\eta_{\s\t}(Y))$ satisfy
the equations $\s_0(\id_{\eta_\t(E)})=\s_E\in
\Hom_{\F_{\s\!\t}}(\eta_{\s\t}(E),\eta_{\s\t}(E)(1))$ and
$$
\s_{i+n+j}(\eta_\t^i(a)(-1)z\eta_\t^j(b))=
\eta_{\s\t}^i(a)\s_n(z)\eta_{\s\t}^j(b)
$$
for any objects $E$, $U$, $X$, $Y$, $V\in\F$ and any Ext classes
$b\in\Ext^j_\F(U,X)$, \
$z\in\Ext^n_{\F_\t}(\eta_\t(X),\eta_\t(Y)(-1))$, and
$a\in\Ext^i_\F(Y,V)$;
\item the maps $\d=\d^n\:\Ext^n_{\F_\s}(\eta_\s(X),\eta_\s(Y))
\rarrow\Ext^{n+1}_{\F_\t}(\eta_\t(X),\eta_\t(Y)(-1))$ satisfy
the equation $$\d^{i+n+j}(\eta_\s^i(a)z\eta_\s^j(b))=(-1)^i
\eta_\t^i(a)(-1)\d^n(z)\eta_\t^j(b)$$ for any objects
$U$, $X$, $Y$, $V\in\F$ and any Ext classes
$b\in\Ext^j_\F(U,X)$, \
$z\in\Ext^n_{\F_\s}(\eta_\s(X),\eta_\s(Y))$, and
$a\in\Ext^i_\F(Y,V)$.
 Here $(-1)^i$ in the right-hand side denotes a plus or minus sign
depending on the parity of~$i$, while $\eta_\t^i(a)(-1)\in
\Ext^i_{\F_\t}(\eta_\t(Y)(-1),\eta_\t(V)(-1))$ is the inverse twist of
the Ext class~$\eta_\t^i(a)$.
\end{enumerate}

 The Bockstein long exact sequence described in
Subsection~\ref{bockstein-posing} is a particular case of
the long exact sequence of the present subsection corresponding to
the situation when the functor $\eta_{\s\t}\:\F\rarrow\F_{\s\t}$
is an equivalence of exact categories.

\subsection{The first term}  \label{first-term-bockstein}
 Now we proceed to construct the Bockstein long exact sequence
promised in Subsection~\ref{bockstein-generalization}.
 We start with constructing the map~$\s_0$, proving that it is
injective, and describing its image.

 Let $X$ and $Y$ be two objects of the category $\F$, and
$p:\eta_\t(X)\rarrow\eta_\t(Y)(-1)$ be a morphism in
the category~$\F_\t$.
 According to the condition~(ii) for the functor~$\eta_\t$, there
exist an admissible epimorphism $X'\rarrow X$, an admissible
monomorphism $Y(-1)\rarrow Y'(-1)$, and morphisms $X'\rarrow Y(-1)$
and $X\rarrow Y'(-1)$ in the category $\F$ whose images
under the functor~$\eta_\t$ together with the morphism~$p$ form
a commutative diagram of two triangles with a common edge in
the category~$\F_\t$.

 The square diagram of morphisms $X'\rarrow X\rarrow Y'(-1)$, \
$X'\rarrow Y(-1)\rarrow Y'(-1)$ in the category $\F$ becomes commutative
after applying the functor~$\eta_\t$, hence it follows from
the condition~(III) that it is commutative modulo the ideal of morphisms
whose images under the functor~$\eta_{\s\t}$ are annihilated by
the natural transformation~$\s$.
 Multiplying by~$\s$ the images of both morphisms $X\rarrow Y'(-1)$ and
$X'\rarrow Y(-1)$ under the functor~$\eta_{\s\t}$ (and keeping in mind
that $\s\:\eta_{\s\t}\rarrow\eta_{\s\t}(1)$ is a natural transformation
commuting with the twists), we therefore obtain a commutative square
$\eta_{\s\t}(X')\rarrow \eta_{\s\t}(X)\rarrow \eta_{\s\t}(Y')$, \
$\eta_{\s\t}(X')\rarrow \eta_{\s\t}(Y)\rarrow \eta_{\s\t}(Y')$
in the category~$\F_{\s\t}$.

 Since the morphism $\eta_{\s\t}(X')\rarrow\eta_{\s\t}(X)$ is
an admissible epimorphism and the morphism $\eta_{\s\t}(Y)\rarrow
\eta_{\s\t}(Y')$ is an admissible monomorphism, it follows
that there exists a unique morphism $f\:\eta_{\s\t}(X)\rarrow
\eta_{\s\t}(Y)$ complementing the latter square to a commutative diagram
of two triangles with a common edge in the category~$\F_{\s\t}$.
 By the definition, we set $\s_0(p)=f$.
 As the morphisms $X'\rarrow X$ and $Y\rarrow Y'$ can be chosen
independently and the choice of either one of them is sufficient to
determine the morphism~$f$, it does not depend on these choices.

\begin{lem}  \label{bockstein-s-zero}
 Assuming the condition~(ii) for the functor~$\eta_\t$
and the condition~(III), the map
$\s_0\:\Hom_{\F_\t}(\eta_\t(X),\eta_\t(Y)(-1))\rarrow
\Hom_{\F_{\s\!\t}}(\eta_{\s\t}(X),\eta_{\s\t}(Y))$
has the following properties: \par
\textup{(a)} the equation $\s_0(\id_{\eta_\t(E)})=\s_E\in
\Hom_{\F_{\s\!\t}}(\eta_{\s\t}(E),\eta_{\s\t}(E)(1))$ describing
the image of the identity endomorphism in the category $\F_\t$
under the map~$\s_0$ in terms of the natural transformation
$\s\:\eta_{\s\t}\rarrow\eta_{\s\t}(1)$ of functors $\F\rarrow\F_{\s\t}$
holds in the category~$\F_{\s\t}$ for every object $E$ of
the category~$\F$; \par
\textup{(b)} the equation $\s_0(\eta_\t(g)(-1)p\eta_\t(h))=
\eta_{\s\t}(g)\s_0(p)\eta_{\s\t}(h)$ holds in the category $\F_{\s\t}$
for any two morphisms $h\:U\rarrow X$, \ $g\:Y\rarrow V$ in
the category $\F$ and any morphism $p\:\eta_\t(X)\rarrow\eta_\t(Y)(-1)$
in the category~$\F_\t$; \par
\textup{(c)} the map~$\s_0$ is injective for any objects
$X$, $Y\in\F$; \par
\textup{(d)} a morphism $\eta_{\s\t}(X)\rarrow\eta_{\s\t}(Y)$ belongs
to the image of the map~$\s_0$ if and only if there exists
an admissible epimorphism $X'\rarrow X$ in the category $\F$ such that
the composition $\eta_{\s\t}(X')\rarrow \eta_{\s\t}(X)\rarrow
\eta_{\s\t}(Y)$ is divisible by the natural transformation~$\s$, and
if and only if there exists an admissible monomorphism $Y\rarrow Y'$
such that the composition $\eta_{\s\t}(X)\rarrow\eta_{\s\t}(Y)
\rarrow\eta_{\s\t}(Y')$ is divisible by~$\s$.
\end{lem}

\begin{proof}
 Part~(a) is immediate from the construction (it suffices to take
$X'=X=E$ and $Y'=Y=E(1)$, with the identity morphisms $X'\rarrow X$ and
$Y(-1)\rarrow Y'(-1)$ and the identity morphisms $X'\rarrow Y(-1)$ and
$X\rarrow Y'(-1)$ in the category~$\F$).
 In part~(b), one checks the equations $\s_0(\eta_\t(g)(-1)p)=
\eta_{\s\t}(g)\s_0(p)$ and $\s_0(p\eta_\t(h))=\s_0(p)\eta_{\s\t}(h)$
separately, using the construction of the morphism $\s_0(p)$ in terms
of an admissible epimorphism $X'\rarrow X$ in the former case and
in terms of an admissible monomorphism $Y\rarrow Y'$ in
the latter one.
 Part~(c) holds, since the morphisms $X'\rarrow Y(-1)$ and
$X\rarrow Y'(-1)$ are annihilated by the functor~$\eta_\t$ whenever
their images under the functor~$\eta_{\s\t}$ are annihilated by
the multiplication with the natural transformation~$\s$, according to
the condition~(III).

 The assertions ``only if'' in part~(d) are obvious from
the construction of the map~$\s_0$.
 To prove the ``if'', suppose that we are given a morphism
$X'\rarrow Y(-1)$ in the category $\F$ whose image under
the functor~$\eta_{\s\t}$, multiplied with~$\s$, is equal to
the composition $\eta_{\s\t}(X')\rarrow \eta_{\s\t}(X)\rarrow
\eta_{\s\t}(Y)$.
 Denote by $K$ the kernel of the morphism $X'\rarrow X$ and
consider the composition $K\rarrow X'\rarrow Y(-1)$.
 The image of this composition under the functor~$\eta_{\s\t}$
is annihilated by the multiplication with~$\s$, and consequently,
according to~(III), the morphism $K\rarrow Y(-1)$ is annihilated
by the functor~$\eta_\t$.
 The short sequence $0\rarrow\eta_\t(K)\rarrow\eta_\t(X')
\rarrow\eta_\t(X)\rarrow0$ being exact in $\F_\t$, one obtains
the desired morphism $\eta_\t(X)\rarrow\eta_\t(Y)(-1)$.
\end{proof}

 In particular, it follows from part~(d) of the Lemma that
the two formulations of the condition~(IV) in
Subsection~\ref{bockstein-generalization} are equivalent
to each other, as are the two formulations of
the condition~(iv) in Subsection~\ref{bockstein-posing}.

\subsection{The second term}  \label{second-term-bockstein}
 Let us construct the map $r_\t^0$ and verify exactness of
our sequence at its second nontrivial term.

 Let $X$ and $Y$ be two objects of the category $\F$, and
$f\:\eta_{\s\t}(X)\rarrow\eta_{\s\t}(Y)$ be a morphism
in the category~$\F_{\s\t}$.
 According to the condition~(ii) for the functor~$\eta_{\s\t}$,
there exist an admissible epimorphism $X'\rarrow X$, an admissible
monomorphism $Y\rarrow Y'$, and morphisms $X'\rarrow Y$ and
$X\rarrow Y'$ in the category $\F$ whose images
under the functor~$\eta_{\s\t}$ together with the morphism~$f$
form a commutative diagram of two triangles with a common edge
in the category~$\F_{\s\t}$.

 Let $K\rarrow X'$ be the kernel of the admissible epimorphism
$X'\rarrow X$ and $Y'\rarrow C$ be the cokernel of the admissible
monomorphism $Y\rarrow Y'$.
 Then the compositions of morphisms $K\rarrow X'\rarrow Y$ and
$X\rarrow Y'\rarrow C$ in the category $\F$ are annihilated by
the functor~$\eta_{\s\t}$, so it follows from
(either formulation of) the condition~(IV) that they are also
annihilated by the functor~$\eta_\s$.
 The same applies to the difference of the two compositions
$X'\rarrow X\rarrow Y'$ and $X'\rarrow Y\rarrow Y'$ in
the category~$\F$.
 Hence the image of our square of morphisms in the category $\F$
with respect to the functor~$\eta_\s$ is commutative in
the category~$\F_\s$, and there is a unique morphism
$q\:\eta_\s(X)\rarrow\eta_\s(Y)$ complementing this square to
a commutative diagram of two triangles with a common edge
in~$\F_\s$.
 By the definition, we set $r_\t^0(f)=q$.

\begin{lem}  \label{bockstein-r-zero}
 Assuming the condition~(ii) for the functors~$\eta_\t$, $\eta_{\s\t}$
and the conditions~(III\+-IV), the map
$r_\t^0\:\Hom_{\F_{\s\t}}(\eta_{\s\t}(X),\eta_{\s\t}(Y))\rarrow
\Hom_{\F_\s}(\eta_\s(X),\eta_\s(Y))$ has the following properties: \par
\textup{(a)} the equation $r_\t^0(\eta_{\s\t}(e))=\eta_\s(e)$
holds in the category $\F_\s$ for any morphism $e\:X\rarrow Y$
in the category~$\F$; \par
\textup{(b)} the equation $r_\t^0(gh)=r_\t^0(g)r_\t^0(h)$ holds
in the category $\F_\s$ for any two morphisms $g\:\eta_{\s\t}(V)
\rarrow \eta_{\s\t}(W)$ and $h\:\eta_{\s\t}(U)\rarrow\eta_{\s\t}(V)$
in the category~$\F_{\s\t}$; \par
\textup{(c)} for any two objects $X$, $Y$ in the category $\F$,
the image of the injection
$\s_0\:\Hom_{\F_\t}(\eta_\t(X),\eta_\t(Y)(-1))\rarrow
\Hom_{\F_{\s\!\t}}(\eta_{\s\t}(X),\eta_{\s\t}(Y))$
coincides with the kernel of the map
$r_\t^0\:\Hom_{\F_{\s\!\t}}(\eta_{\s\t}(X),\eta_{\s\t}(Y))
\rarrow\Hom_{\F_\s}(\eta_\s(X),\eta_\s(Y))$.
\end{lem}

\begin{proof}
 Part~(a) is obvious from the construction.
 In part~(b), choose admissible epimorphisms $U'\rarrow U$ and
$V'\rarrow V$ and morphisms $U'\rarrow V$ and $V'\rarrow W$ in
the category $\F$ such that the images of these morphisms under
the functor~$\eta_{\s\t}$ form two commutative triangle diagrams
with the morphisms $h$ and~$g$ in the category~$\F_{\s\t}$.
 Let $U''$ be the fibered product of the objects $U'$ and $V'$
over the object $V$ in the category~$\F$.
 Then the morphism $U''\rarrow U'$ is an admissible epimorphism,
hence so is the composition $U''\rarrow U'\rarrow U$.
 The images of the composition $U''\rarrow V'\rarrow W$ and
the morphism $U''\rarrow U$ under the functor~$\eta_{\s\t}$ form
a commutative triangle diagram with the morphism
$gh\:\eta_{\s\t}(U)\rarrow\eta_{\s\t}(W)$ in the category~$\F_{\s\t}$.
 Applying the functor~$\eta_\t$ to the whole commutative diagram in
the category~$\F$, we conclude that the composition of the morphism
$r_\t^0(h)$ complementing the pair of morphisms $\eta_\t(U')\rarrow
\eta_\t(U)$, \ $\eta_\t(U')\rarrow\eta_\t(V)$ to a commutative triangle
and the morphism $r_\t^0(g)$ complementing the pair of morphisms
$\eta_\t(V')\rarrow\eta_\t(V)$, \ $\eta_\t(V')\rarrow\eta_\t(W)$ to
a commutative triangle complements the pair of morphisms
$\eta_\t(U'')\rarrow\eta_\t(U)$, \ $\eta_\t(U'')\rarrow\eta_\t(W)$ to
a commutative triangle diagram in the category~$\F_\t$.
% In part~(b), one first considers the situation when one of
%the morphisms $g$ or~$h$ comes from a morphism in the category $\F$
%via the functor~$\eta_{\s\t}$, proving the equation
%$r_\t^0(\eta_{\s\t}(j)h)=\eta_\s(j)r_\t^0(h)$ using the construction
%of the morphism $r_\t^0(h)$ in terms of an admissible epimorphism
%$U'\rarrow U$ and the equation and $r_\t^0(g\eta_{\s\t}(k))=
%r_\t^0(g)\eta_\s(k)$ using the construction of the morphism $r_\t^0(g)$
%in terms of an admissible monomorphism $W\rarrow W'$.
% Then one deduces and applies a common generalization of our two
%constructions of the map~$r_\t^0$, in which, to obtain the morphism
%$r_\t^0(f)\in\Hom_{\F_\s}(\eta_\s(X),\eta_\s(Y))$ for a given morphism
%$f\in\Hom_{\F_{\s\!\t}}(\eta_{\s\t}(X),\eta_{\s\t}(Y))$, one picks
%an admissible epimorphism $X'\rarrow X$ and an admissible monomorphism
%$Y\rarrow Y'$ in the category $\F$ for which the composition
%$\eta_{\s\t}(X')\rarrow\eta_{\s\t}(X)\rarrow\eta_{\s\t}(Y)\rarrow
%\eta_{\s\t}(Y')$ comes from a morphism $X'\rarrow Y'$
%in the category~$\F$.
% The morphism $r_\t^0(f)$ is the unique morphism $\eta_\s(X)\rarrow
%\eta_\s(Y)$ complementing to a commutative diagram the images of
%the morphisms $X'\rarrow X$, \ $Y\rarrow Y'$, and $X'\rarrow Y'$
%under the functor~$\eta_\s$.

 To prove part~(c), consider a morphism $f\:\eta_{\s\t}(X)\rarrow
\eta_{\s\t}(Y)$ in the category~$\F_{\s\t}$.
 Let $X'\rarrow X$ and $X'\rarrow Y$ be an admissible epimorphism
and a morphism in the category $\F$ whose images under
the functor~$\eta_{\s\t}$ form a commutative diagram with
the morphism~$f$.
 The equation $r_\t^0(f)=0$ means that the morphism $X'\rarrow Y$
is annihilated by the functor~$\eta_\s$.
 According to the condition~(IV), this is equivalent to the existence
of an admissible epimorphism $X''\rarrow X'$ in the category~$\F$
for which the composition $\eta_{\s\t}(X'')\rarrow\eta_{\s\t}(X')\rarrow
\eta_{\s\t}(Y)$ is divisible by~$\s$.
 If this is the case, then it follows by the way of
Lemma~\ref{bockstein-s-zero}(d) that the morphism~$f$ belongs to
the image of the map~$\s_0$ (as the composition $X''\rarrow X'
\rarrow X$ is also an admissible epimorphism in~$\F$).
 Conversely, by Lemma~\ref{bockstein-s-zero}(b) the composition
$\eta_{\s\t}(X')\rarrow\eta_{\s\t}(X)\rarrow\eta_{\s\t}(Y)$ belongs
to the image of the map~$\s_0$ whenever the morphism~$f$ does.
 Then it remains to apply Lemma~\ref{bockstein-s-zero}(d) in order
to deduce the existence of an admissible epimorphism $X''\rarrow X'$
with the desired property.
\end{proof}

 The result of Lemma~\ref{bockstein-r-zero}(b) says that for any
commutative diagram in the category $\F_{\s\t}$ with the objects in
the vertices coming from the category $\F$ via
the functor~$\eta_{\s\t}$, one can apply the maps~$r_\t^0$ to every
arrow in the diagram, obtaining a commutative diagram in
the category~$\F_\s$.
 By part~(a) of the Lemma, when some of the arrows in
the original diagram come from arrows in the category $\F$,
the procedure is compatible with the action of
the functors~$\eta_{\s\t}$ and $\eta_\s$ on such arrows.

\subsection{The third term}  \label{third-term-bockstein}
 Now we construct the map~$\d^0$ and check exactness of
the sequence at its third term.
 The construction and arguments largely follow those
in~\cite[Subsections~4.5\+-4.6]{Partin}.

 Let $X$ and $Y$ be two objects of the category $\F$, and
$q\:\eta_\s(X)\rarrow\eta_\s(Y)$ be a morphism
in the category $\F_\s$.
 According to the condition~(ii) for the functor~$\eta_\s$, there
exist an admissible epimorphism $X'\rarrow X$, an admissible
monomorphism $Y\rarrow Y'$, and morphisms $X'\rarrow Y$ and
$X\rarrow Y'$ in the category $\F$ whose images under
the functor~$\eta_\s$ together with the morphism~$q$ form
a commutative diagram of two triangles with a common edge in
the category~$\F_\s$.

 The square diagram of morphisms $X'\rarrow X\rarrow Y'$, \ $X'\rarrow
Y\rarrow Y'$ becomes commutative after applying the functor~$\eta_\s$,
and consequently, according to Lemma~\ref{bockstein-r-zero}(a,c), its
image under the functor~$\eta_{\s\t}$ is commutative in the category
$\F_{\s\t}$ up to a morphism coming from a morphism in $\F_\t$ via
the map~$\s_0$.

 Let $K\rarrow X'$ be the kernel of the morphism $X'\rarrow X$ and
$Y'\rarrow C$ be the cokernel of the morphism $Y\rarrow Y'$
in~$\F$.
 Then the compositions $K\rarrow X'\rarrow Y$ and $X\rarrow Y'\rarrow C$
are annihilated by the functor~$\eta_\s$, and therefore their images
under the functor~$\eta_{\s\t}$ come from morphisms
$\eta_\t(K)\rarrow\eta_\t(Y)(-1)$ and $\eta_\t(X)\rarrow
\eta_\t(C)(-1)$ in the category~$\F_\t$.
 The difference of the two compositions in the square diagram of
morphisms in $\F_{\s\t}$ comes from a morphism
$\eta_\t(X')\rarrow\eta_\t(Y')(-1)$ in~$\F_\t$.

 Together with the images of the short exact sequences $0\rarrow K
\rarrow X'\rarrow X\rarrow\nobreak 0$ and $0\rarrow Y(-1)\rarrow
Y'(-1)\rarrow C(-1)\rarrow 0$ with respect to the functor~$\eta_\t$,
these three morphisms form a diagram of two squares, one of which
is commutative and the other one anticommutative (as one can
check using Lemma~\ref{bockstein-s-zero}(b\+c)).
 Such a diagram defines an element of the group
$\Ext^1_{\F_\t}(\eta_\t(X),\eta_\t(Y)(-1))$ in any one of
the two dual ways differing by the minus sign
(cf.~Lemma~\ref{ext-1-equation}).

 Namely, the desired element can be obtained either
as the composition of the $\Ext^1$ class of
the sequence $0\rarrow\eta_\t(K)\rarrow\eta_\t(X')\rarrow\eta_\t(X)
\rarrow0$ with the morphism $\eta_\t(K)\rarrow\eta_\t(Y)(-1)$, or
as the composition of the morphism $\eta_\t(X)\rarrow\eta_\t(C)(-1)$
with the $\Ext^1$ class of the sequence $0\rarrow\eta_\t(Y)(-1)
\rarrow\eta_\t(Y')(-1)\rarrow\eta_\t(C)(-1)\rarrow0$ in
the exact category~$\F_\t$.
 By the definition, we set this element to be the value $\d^0(q)$
of the map $\d^0\:\Hom_{\F_\s}(\eta_\s(X),\eta_\s(Y))\rarrow
\Ext^1_{\F_\t}(\eta_\t(X),\eta_\t(Y)(-1))$ at the morphism
$q\:\eta_\s(X)\rarrow\eta_\s(Y)$.
 As the pairs of morphisms $X'\rarrow X$, $X'\rarrow Y$ and
$X\rarrow Y'$,  $Y\rarrow Y'$ can be chosen independently of one
another and the choice of either one of these two pairs is sufficient
to determine the extension class $\d^0(q)$, it does not depend on
these choices.

\begin{lem} \label{bockstein-d-zero}
 Assuming the condition~(ii) for the functors~$\eta_\t$, $\eta_{\s\t}$,
$\eta_\s$ and the conditions~(III\+-IV), the map
$\d^0\:\Hom_{\F_\s}(\eta_\s(X),\eta_\s(Y))\rarrow
\Ext^1_{\F_\t}(\eta_\t(X),\eta_\t(Y)(-1))$
has the following properties: \par
\textup{(a)} the equation $\d^0(\eta_\s(g)q\eta_\s(h))=\eta_\t(g)(-1)
\d^0(q)\eta_\t(h)$ holds in the category $\F_\t$ for any two morphisms
$h\:U\rarrow X$, \ $g\:Y\rarrow V$ in the category $\F$ and any
morphism $q\:\eta_\s(X)\rarrow\eta_\s(Y)$ in the category~$\F_\s$; \par
\textup{(b)} for any two objects $X$, $Y$ in the category $\F$,
the kernel of the map $\d^0\:\Hom_{\F_\s}(\eta_\s(X),\eta_\s(Y))\rarrow
\Ext^1_{\F_\t}(\eta_\t(X),\eta_\t(Y)(-1))$ coincides with the image of
the map $r_\t^0\:\Hom_{\F_{\s\!\t}}(\eta_{\s\t}(X),\eta_{\s\t}(Y))
\rarrow\Hom_{\F_\s}(\eta_\s(X),\eta_\s(Y))$.
\end{lem}

\begin{proof}
 To prove part~(a), one checks the equations $\d^0(\eta_\s(g)q)=
\eta_\t(g)(-1)\d^0(q)$ and $\d^0(q\eta_\s(h))=\d^0(q)\eta_\t(h)$
separately, using the construction of the element $\d^0(q)$ (as
the product of a morphism and an $\Ext^1$ class in~$\F_\t$) in terms of
an admissible epimorphism $X'\rarrow X$ in the former case and
in terms of an admissible monomorphism $Y\rarrow Y'$ in the latter one,
together with the result of Lemma~\ref{bockstein-s-zero}(b).

 To prove part~(b), consider a morphism $q\:\eta_\s(X)\rarrow\eta_\s(Y)$
in the category~$\F_\s$, and let $X'\rarrow X$ and $X'\rarrow Y$ be
an admissible epimorphism and a morphism in the category $\F$
whose images under the functor~$\eta_\s$ form a commutative
diagram together with the morphism~$q$.
 Let $K\rarrow X'$ be the kernel of the morphism $X'\rarrow X$.
 According to the above, the image of the composition of morphisms
$K\rarrow X'\rarrow Y$ in the category $\F$ under
the functor~$\eta_{\s\t}$ comes from a morphism $\eta_\t(K)\rarrow
\eta_\t(Y)(-1)$ in the category $\F_\t$ via the map~$\s_0$.

 The class $\d^0(q)\in\Ext^1_{\F_\t}(\eta_\t(X),\eta_\t(Y)(-1))$
is induced from the $\Ext^1$ class of the short exact sequence
$0\rarrow\eta_\t(K)\rarrow\eta_\t(X')\rarrow\eta_\t(X)\rarrow0$
using the morphism $\eta_\t(K)\rarrow\eta_\t(Y)(-1)$.
 Hence one has $\d^0(q)=0$ if and only if the latter morphism factorizes
through the admissible monomorphism $\eta_\t(K)\rarrow\eta_\t(X')$.

 Subtracting the image of the related morphism $\eta_\t(X')\rarrow
\eta_\t(Y)(-1)$ under the map~$\s_0$ from the image of the morphism
$X'\rarrow Y$ under the functor~$\eta_{\s\t}$, we obtain a morphism
$f'\:\eta_{\s\t}(X')\rarrow\eta_{\s\t}(Y)$ in the category~$\F_{\s\t}$
whose composition with the image of the admissible monomorphism
$K\rarrow X'$ under the functor~$\eta_{\s\t}$ vanishes (as one can
compute using Lemma~\ref{bockstein-s-zero}(b)).
 Hence the desired morphism $f\:\eta_{\s\t}(X)\rarrow\eta_{\s\t}(Y)$
in the category~$\F_{\s\t}$.
 Since $r^0_\t\circ\s_0=0$, one has $r^0_\t(f') = \eta_\s(X'\to Y)$
in the category $\F_\s$, and it follows that $r^0_\t(f) = q$. 

 Conversely, if there exists a morphism $f\:\eta_{\s\t}(X)\rarrow
\eta_{\s\t}(Y)$ such that $q=r_\t^0(f)$, then one can choose
a pair of morphisms $X'\rarrow X$, $X'\rarrow Y$ in the category $\F$
in such a way that the triangle diagram $\eta_{\s\t}(X')\rarrow
\eta_{\s\t}(X)\rarrow\eta_{\s\t}(Y)$ is commutative in the category
$\F_{\s\t}$ and the triangle diagram $\eta_\s(X')\rarrow\eta_\s(X)
\rarrow\eta_\s(Y)$ is commutative in the category~$\F_\s$.
 Then the image of the composition of morphisms $K\rarrow X'\rarrow Y$
in the category $\F$ under the functor~$\eta_{\s\t}$ vanishes in
the category $\F_{\s\t}$, hence the morphism $\eta_\t(K)\rarrow
\eta_\t(Y)(-1)$ vanishes in the category $\F_\t$ by
Lemma~\ref{bockstein-s-zero}(c), and therefore the extension class
$\d^0(q)$ vanishes as well.
\end{proof}

\subsection{Construction of higher differentials}
\label{higher-bockstein-construction}
 The constructions of the maps~$\s_n$, \,$r_\t^n$, and~$\d^n$
for $n\ge1$ are based on the result of
Proposition~\ref{exact-surjectivity-ext-prop}(b).
 We continue to follow~\cite[Subsection~4.5]{Partin}.

\begin{lem}  \label{s-n-bockstein}
 Assuming the conditions~(i\+-IV), there exists a unique way to extend
the maps $\s_0\:\Hom_{\F_\t}(\eta_\t(X),\eta_\t(Y)(-1))\rarrow
\Hom_{\F_{\s\!\t}}(\eta_{\s\t}(X),\eta_{\s\t}(Y))$ of
Subsection~\ref{first-term-bockstein} to maps
$\s_n\:\Ext^n_{\F_\t}(\eta_\t(X),\eta_\t(Y)(-1))\rarrow
\Ext^n_{\F_{\s\!\t}}(\eta_{\s\t}(X),\eta_{\s\t}(Y))$ defined for
all objects $X$, $Y\in\F$ and all integers $n\ge0$ and satisfying
the equations~(b) of Subsection~\ref{bockstein-generalization}.
\end{lem}

\begin{proof}
 Consider the two equations $\s_{i+n}(\eta_\t^i(a)(-1)z)=
\eta_{\s\t}^i(a)\s_n(z)$ and $\s_{n+j}(z\eta_\t^j(b))=
\s_n(z)\eta_{\s\t}^j(b)$ separately.
 In view of Proposition~\ref{exact-surjectivity-ext-prop}(b) and
the dual result, based on the conditions~(i$'$) and~(i$''$) for
the functor~$\eta_\t$, it follows from
Lemma~\ref{bockstein-s-zero}(b) that there exists a unique collection
of maps $\s_n''\:\Ext^n_{\F_\t}(\eta_\t(X),\eta_\t(Y)(-1))\rarrow
\Ext^n_{\F_{\s\!\t}}(\eta_{\s\t}(X),\eta_{\s\t}(Y))$ extending
the maps~$\s_0$ and satisfying the former system of equations,
and also a unique collection of maps
$\s_n'\:\Ext^n_{\F_\t}(\eta_\t(X),\eta_\t(Y)(-1))\rarrow
\Ext^n_{\F_{\s\!\t}}(\eta_{\s\t}(X),\eta_{\s\t}(Y))$ extending
the maps~$\s_0$ and satisfying the latter system of equations.
 It remains to show that $\s'_n=\s''_n$.

 Let us first prove that $\s'_1=\s''_1$.
 Suppose that we are given two short exact sequences
$0\rarrow V\rarrow P\rarrow X\rarrow 0$ and
$0\rarrow Y\rarrow Q\rarrow U\rarrow0$ in the category $\F$ representing
the $\Ext^1$ classes $b\in\Ext^1_\F(X,V)$ and $a\in\Ext^1_\F(U,Y)$.
 Suppose further that we are given two morphisms $w\:\eta_\t(V)
\rarrow\eta_\t(Y)(-1)$ and $z\:\eta_\t(X)\rarrow\eta_\t(U)(-1)$
in the category $\F_\t$ such that the equation $\eta_\t^1(a)(-1)z=
w\eta_\t^1(b)$ holds in the group
$\Ext^1_{\F_\t}(\eta_\t(X),\eta_\t(Y)(-1))$.
 Then the morphisms~$w$ and~$z$ can be extended to a morphism of
short exact sequences (that is a diagram of two commutative squares)
$(\eta_\t(V)\to\eta_\t(P)\to\eta_\t(X))\rarrow(\eta_\t(Y)(-1)\to
\eta_\t(Q)(-1)\to\eta_\t(U)(-1))$ in the category~$\F_\t$
(see Lemma~\ref{ext-1-equation}).

 Applying the maps~$\s_0$ to the morphisms $w\:\eta_\t(V)\rarrow
\eta_\t(Y)(-1)$, \ $\eta_\t(P)\rarrow\eta_\t(Q)(-1)$, and $z\:\eta_\t(X)
\rarrow\eta_\t(U)(-1)$ in the category $\F_\t$, we obtain, in view of
Lemma~\ref{bockstein-s-zero}(b), a morphism of short exact sequences
$(\eta_{\s\t}(V)\to\eta_{\s\t}(P)\to\eta_{\s\t}(X))\rarrow
(\eta_{\s\t}(Y)\to\eta_{\s\t}(Q)\to\eta_{\s\t}(U))$
in the category~$\F_{\s\t}$.
 The commutativity of this diagram of two squares proves
the desired equation $\eta_{\s\t}^1(a)\s_0(z)=\s_0(w)\eta_{\s\t}^1(b)$
in the group $\Ext^1_{\F_{\s\!\t}}(\eta_{\s\t}(X),\eta_{\s\t}(Y))$.

 Finally, we argue by induction in $n\ge1$.
 Let us assume that $\s'_i=\s''_i$ and $\s'_j=\s''_j$ for some
$i$, $j\ge1$ and deduce the equation $\s'_{i+j}=\s''_{i+j}$.
 Let $X$ and $Y$ be two objects in the category $\F$ and
$z\in\Ext^{i+j}_{\F_\t}(\eta_\t(X),\eta_\t(Y)(-1))$ be an Ext class in
the category~$\F_\t$.
 By the assertion dual to
Proposition~\ref{exact-surjectivity-ext-prop}(b) applied to
the functor~$\eta_\t$, there exists a morphism $p\:\eta_\t(X)\rarrow
\eta_\t(Y')(-1)$ in the category $\F_\t$ and an Ext class
$a\in\Ext^{i+j}_\F(Y',Y)$ in the category $\F$ such that
$z=\eta_\t^{i+j}(a)(-1)p$.
 By the definition of Yoneda Ext, there exists an object $U\in\F$ and
two Ext classes $c\in\Ext^j_\F(Y',U)$, \ $b\in\Ext^i_\F(U,Y)$
such that $a=bc$.

 Consider the Ext class $w=\eta_\t^j(c)(-1)p\in
\Ext^j_{\F_\t}(\eta_\t(X),\eta_\t(U)(-1))$.
 By Proposition~\ref{exact-surjectivity-ext-prop}(b) for
the functor~$\eta_\t$, there exists a morphism $q\:\eta_\t(X')\rarrow
\eta_\t(U)(-1)$ in the category $\F_\t$ and an Ext class
$d\in\Ext^j_\F(X,X')$ in the category $\F$ such that $w=q\eta_\t^j(d)$.
 Now we have
\begin{multline*}
 \s''_{i+j}(z)=\s''_{i+j}(\eta_\t^{i+j}(a)(-1)p)=
\s''_{i+j}(\eta_\t^i(b)(-1)\eta_\t^j(c)(-1)p) \\ =
\eta_{\s\t}^i(b)\s''_j(\eta_\t^j(c)(-1)p)=\eta_{\s\t}^i(b)\s''_j(w)=
\eta_{\s\t}^i(b)\s'_j(w) = \eta_{\s\t}^i(b)\s'_j(q\eta_\t^j(d)) \\ =
\eta_{\s\t}^i(b)\s_0(q)\eta_{\s\t}^j(d)=\s''_i(\eta_\t^i(b)(-1)q)
\eta_{\s\t}^j(d) =\s'_i(\eta_\t^i(b)(-1)q)\eta_{\s\t}^j(d) \\ =
\s'_{i+j}(\eta_\t^i(b)(-1)q\eta_\s^j(d))=\s'_{i+j}(\eta_\t^i(b)(-1)
\eta_\t^j(c)(-1)p)=\s'_{i+j}(z).
\end{multline*}
\end{proof}

\begin{lem}  \label{r-n-bockstein}
 Assuming the conditions~(i\+-IV), there exists a unique way to extend
the maps $r_\t^0\:\Hom_{\F_{\s\!\t}}(\eta_{\s\t}(X),\eta_{\s\t}(Y))
\rarrow\Hom_{\F_\s}(\eta_\s(X),\eta_\s(Y))$ of
Subsection~\ref{second-term-bockstein} to maps
$r_\t^n\:\Ext^n_{\F_{\s\!\t}}(\eta_{\s\t}(X),\eta_{\s\t}(Y))\rarrow
\Ext^n_{\F_\s}(\eta_\s(X),\eta_\s(Y))$ defined for all objects
$X$, $Y\in\F$ and all integers $n\ge0$ and satisfying the equations~(a)
of Subsection~\ref{bockstein-generalization}.
\end{lem}

\begin{proof}
 Consider first two weaker systems of equations
$r_\t^{i+n}(\eta_{\s\t}^i(a)z)=\eta_\s^i(a)r_\t^n(z)$ and
$r_\t^{n+j}(z\eta_{\s\t}^j(b))=r_\t^n(z)\eta_\s^j(b)$ for any
objects $U$, $X$, $Y$, $V\in\F$ and any Ext classes
$a\in\Ext^i_\F(Y,V)$, \ $b\in\Ext^j_\F(U,X)$, and
$z\in\Ext^n_{\F_{\s\!\t}}(\eta_{\s\t}(X),\eta_{\s\t}(Y))$.
 By the way of Proposition~\ref{exact-surjectivity-ext-prop}(b) and
the dual result, based on the conditions~(i$'$) and~(i$''$) for
the functor~$\eta_{\s\t}$, it follows from
Lemma~\ref{bockstein-r-zero}(a\+b) that there exists a unique
collection of maps
${}''\!\.r_\t^n\:\Ext^n_{\F_{\s\!\t}}(\eta_{\s\t}(X),\eta_{\s\t}(Y))
\rarrow\Ext^n_{\F_\s}(\eta_\s(X),\eta_\s(Y))$ extending
the maps~$r_\t^0$ and satisfying the former system of equations,
and also a unique collection of maps 
${}'\!\.r_\t^n\:\Ext^n_{\F_{\s\!\t}}(\eta_{\s\t}(X),\eta_{\s\t}(Y))
\rarrow\Ext^n_{\F_\s}(\eta_\s(X),\eta_\s(Y))$ extending
the maps~$r_\t^0$ and satisfying the latter system of equations.

 Due to a computation similar to the one that we have seen in the final
part of the previous proof, in order to show that ${}'\!\.r_\t^n=
{}''\!\.r_\t^n$ for all $n\ge1$, it suffices to check that
${}'\!\.r_\t^1={}''\!\.r_\t^1$.
 As in the previous proof, we have two short exact sequences
$0\rarrow V \rarrow P\rarrow X\rarrow 0$ and
$0\rarrow Y\rarrow Q\rarrow U\rarrow0$
in the category $\F$ representing the $\Ext^1$ classes 
$b\in\Ext^1_\F(X,V)$ and $a\in\Ext^1_\F(U,Y)$.
 We also have two morphisms $w\:\eta_{\s\t}(V)\rarrow\eta_{\s\t}(Y)$ and
$z\:\eta_{\s\t}(X)\rarrow\eta_{\s\t}(U)$ in the category $\F_{\s\t}$
for which the equation $\eta_{\s\t}^1(a)z=w\eta_{\s\t}^1(b)$ holds in 
$\Ext^1_{\F_{\s\!\t}}(\eta_{\s\t}(X),\eta_{\s\t}(Y))$.
 Then there is a morphism of short exact sequences
$(\eta_{\s\t}(V)\to\eta_{\s\t}(P)\to\eta_{\s\t}(X))\rarrow
(\eta_{\s\t}(Y)\to\eta_{\s\t}(Q)\to\eta_{\s\t}(U))$
in the category~$\F_{\s\t}$.

 Applying the maps~$r_\t^0$ to this commutative diagram in
the category~$\F_{\s\t}$ with the objects in the vertices coming
from the category $\F$ via the functor~$\eta_{\s\t}$, we obtain,
in view of Lemma~\ref{bockstein-r-zero}(a\+b), a morphism of
short exact sequences $(\eta_\s(V)\to\eta_\s(P)\to\eta_\s(X))\rarrow
(\eta_\s(Y)\to\eta_\s(Q)\to\eta_\s(U))$ in the category~$\F_\s$.
 Commutativity of the latter diagram in the category $\F_\s$ proves
the desired equation $\eta_\s^1(a)r_\t^0(z)=r_\t^0(w)\eta_\s^1(b)$
in the group $\Ext^1_{\F_\s}(\eta_\s(X),\eta_\s(Y))$.
 Thus ${}'\!\.r_\t^1={}''\!\.r_\t^1$, and therefore
${}'\!\.r_\t^n={}''\!\.r_\t^n$ for all $n\ge1$.

 Now, again by Proposition~\ref{exact-surjectivity-ext-prop}(b)
and its dual assertion for the functor~$\eta_{\s\t}$, for any three
objects $U$, $V$, $W$ in the category $\F$ and any two Ext classes
$y\in\Ext^j_{\F_{\s\!\t}}(\eta_{\s\t}(U),\eta_{\s\t}(V))$ and 
$x\in\Ext^i_{\F_{\s\!\t}}(\eta_{\s\t}(V),\eta_{\s\t}(W))$ 
in the category~$\F_{\s\t}$, one can find two morphisms
$g\:\eta_{\s\t}(U')\rarrow\eta_{\s\t}(V)$ and
$f\:\eta_{\s\t}(V)\rarrow\eta_{\s\t}(W')$ in $\F_{\s\t}$ and
two Ext classes $b\in\Ext^j_\F(U,U')$ and 
$a\in\Ext^i_\F(W',W)$ in $\F$ such that $y=g\eta_{\s\t}^j(b)$
and $x=\eta_{\s\t}^i(a)f$.
 Finally, we have $r_\t^{i+j}(xy) =
r_\t^{i+j}\big(\eta_{\s\t}^i(a)fg\eta_{\s\t}^j(b)\big) =
\eta_\s^i(a)r_\t^0(fg)\eta_\s^j(b) = \eta_\s^i(a)r_\t^0(f)
r_\t^0(g)\eta_\s^j(b) = r_\t^i(x)r_\t^j(y)$
in $\Ext_{\F_\s}^{i+j}(\eta_\s(U),\eta_\s(W))$.
\end{proof}

\begin{lem}  \label{d-n-bockstein}
 Assuming the conditions~(i\+-IV), there exists a unique way to extend
the maps $\d^0\:\Hom_{\F_\s}(\eta_\s(X),\eta_\s(Y))\rarrow
\Ext^1_{\F_\t}(\eta_\t(X),\eta_\t(Y)(-1))$
of Subsection~\ref{third-term-bockstein}
to maps $\d^n\:\Ext^n_{\F_\s}(\eta_\s(X),\eta_\s(Y))\rarrow
\Ext^{n+1}_{\F_\t}(\eta_\t(X),\eta_\t(Y)(-1))$ defined for all objects
$X$, $Y\in\F$ and all integers $n\ge0$ and satisfying
the equations~(c) of Subsection~\ref{bockstein-generalization}.
\end{lem}

\begin{proof}
 As in the proof of Lemma~\ref{s-n-bockstein}, we consider the two
equations $\d^{i+n}(\eta_\s^i(a)z)=(-1)^i\eta_\t^i(a)(-1)\d^n(z)$ and
$\d^{n+j}(z\eta_\s^j(b))=\d^n(z)\eta_\t^j(b)$ separately.
 In view of Proposition~\ref{exact-surjectivity-ext-prop}(b) and
the dual result, based on the conditions~(i$'$) and~(i$''$) for
the functor~$\eta_\s$, it follows from Lemma~\ref{bockstein-d-zero}(a)
that there exists a unique collection of maps
${}''\d^n\:\Ext^n_{\F_\s}(\eta_\s(X),\eta_\s(Y))\rarrow
\Ext^{n+1}_{\F_\t}(\eta_\t(X),\eta_\t(Y)(-1))$ extending the maps~$\d^0$
and satisfying the former system of equations, and also a unique
collection of maps 
${}'\d^n\:\Ext^n_{\F_\s}(\eta_\s(X),\eta_\s(Y))\rarrow
\Ext^{n+1}_{\F_\t}(\eta_\t(X),\eta_\t(Y)(-1))$ extending the maps~$\d^0$
and satisfying the latter system of equations.

 In order to show that ${}'\d^n={}''\d^n$ for all $n\ge1$, it
suffices to check that ${}'\d^1={}''\d^1$.
 As in the proofs of the preceding lemmas in this subsection, we have
two short exact sequences $0\rarrow V \rarrow P\rarrow X\rarrow 0$
and $0\rarrow Y\rarrow Q\rarrow U\rarrow0$
in the category $\F$ representing the $\Ext^1$ classes 
$b\in\Ext^1_\F(X,V)$ and $a\in\Ext^1_\F(U,Y)$.
 We have two morphisms $w\:\eta_\s(V)\rarrow\eta_\s(Y)$ and
$z\:\eta_\s(X)\rarrow\eta_\s(U)$ in the category $\F_\s$ for which
the equation $\eta_\s^1(a)z=w\eta_\s^1(b)$ holds in 
$\Ext^1_{\F_\s}(\eta_\s(X),\eta_\s(Y))$.
 Then there is a morphism of short exact sequences
$(\eta_\s(V)\to\eta_\s(P)\to\eta_\s(X))\rarrow
(\eta_\s(Y)\to\eta_\s(Q)\to\eta_\s(U))$ in the category~$\F_\s$. 

 According to the condition~(ii) for the functor~$\eta_\s$,
there exists an admissible epimorphism $X'\rarrow X$ and
a morphism $X'\rarrow U$ in the category $\F$ whose images
under the functor~$\eta_\s$ form a commutative diagram with
the morphism $\eta_\s(X)\rarrow\eta_\s(U)$ in the category~$\F_\s$.
 Denote by $P'''$ the fibered product of the objects $P$ and $X'$
over $X$ in the category~$\F$.
 Choose an admissible epimorphism $P''\rarrow P'''$ and
a morphism $P''\rarrow Q$ in the category $\F$ whose images
under~$\eta_\s$ form a commutative diagram with
the composition of morphisms $\eta_\s(P''')\rarrow\eta_\s(P)
\rarrow\eta_\s(Q)$ in~$\F_\s$.

 Consider the difference of the compositions of morphisms
$P''\rarrow P'''\rarrow X'\rarrow U$ and $P''\rarrow Q\rarrow U$
in the category~$\F$.
 It is annihilated by the functor~$\eta_\s$, and consequently, its
image under the functor~$\eta_{\s\t}$ comes from a morphism
$\eta_\t(P'')\rarrow\eta_\t(U)(-1)$ in the category $\F_\t$
via the map~$\s_0$.
 Denote by $T$ the fibered product of the objects $\eta_\t(P'')$
and $\eta_\t(Q)(-1)$ over $\eta_\t(U)(-1)$ in the category~$\F_\t$.
 The morphism $T\rarrow\eta_\t(P'')$ is an admissible epimorphism
in~$\F_\t$; hence, according to the condition~(i) for
the functor~$\eta_\t$, there exists an admissible epimorphism
$P'\rarrow P''$ in the category $\F$ and a morphism $\eta_\t(P')
\rarrow T$ in the category $\F_\t$ making the triangle diagram
$\eta_\t(P')\rarrow T\rarrow\eta_\t(P'')$ commutative in~$\F_\t$.

 Applying the map~$\s_0$ to the composition of morphisms
$\eta_\t(P')\rarrow T\rarrow\eta_\t(Q)(-1)$ in the category $\F_\t$,
we obtain a morphism $f\:\eta_{\s\t}(P')\rarrow\eta_{\s\t}(Q)$
entering into a commutative square of morphisms
$\eta_{\s\t}(P')\rarrow\eta_{\s\t}(Q)\rarrow\eta_{\s\t}(U)$ and
$\eta_{\s\t}(P')\rarrow\eta_{\s\t}(P'')\rarrow\eta_{\s\t}(U)$
in the category~$\F_{\s\t}$.
 Here the morphisms $\eta_{\s\t}(Q)\rarrow\eta_{\s\t}(U)$ and
$\eta_{\s\t}(P')\rarrow\eta_{\s\t}(P'')\rarrow\eta_{\s\t}(U)$ come
from morphisms in the category $\F$ via the functor~$\eta_{\s\t}$,
while at the same time the morphisms $\eta_{\s\t}(P'')\rarrow
\eta_{\s\t}(U)$ and $f\:\eta_{\s\t}(P')\rarrow\eta_{\s\t}(Q)$ come
from morphisms in the category $\F_\t$ via the maps~$\s_0$, and are,
consequently, annihilated by the maps~$r_\t^0$.

 Define the morphism $P'\rarrow P$ as the composition $P'\rarrow P''
\rarrow P'''\rarrow P$ and the morphism $P'\rarrow X'$ as
the composition $P'\rarrow P''\rarrow P'''\rarrow X'$ of morphisms
in the category~$\F$.
 Let $\eta_{\s\t}(P')\rarrow\eta_{\s\t}(X')$, \ $\eta_{\s\t}(X')\rarrow
\eta_{\s\t}(U)$, and $\eta_{\s\t}(Q)\rarrow\eta_{\s\t}(U)$ be
the images of the morphisms $P'\rarrow X'$, \ $X'\rarrow U$, and
$Q\rarrow U$ under the functor~$\eta_{\s\t}$.
 Furthermore, set the new morphism $\eta_{\s\t}(P')\rarrow
\eta_{\s\t}(Q)$ to be the sum of the image of the composition
$P'\rarrow P''\rarrow Q$ under the functor~$\eta_{\s\t}$
and the morphism~$f$.
 Then the square diagram formed by the morphisms
$\eta_{\s\t}(P')\rarrow\eta_{\s\t}(X')\rarrow\eta_{\s\t}(U)$
and $\eta_{\s\t}(P')\rarrow\eta_{\s\t}(Q)\rarrow\eta_{\s\t}(U)$
is commutative in the category $\F_{\s\t}$, while the triangle
diagram $\eta_\s(P')\rarrow\eta_\s(P)\rarrow\eta_\s(Q)$ formed by
the morphism $\eta_\s(P)\rarrow\eta_\s(Q)$, the image of the morphism
$P'\rarrow P$ under the functor~$\eta_\s$, and the image of
the morphism $\eta_{\s\t}(P')\rarrow\eta_{\s\t}(Q)$ under
the map~$r_\t^0$ is commutative in the category~$\F_\s$.

 Let $V'\rarrow P'$ be the kernel of the admissible epimorphism
$P'\rarrow X'$ in the category~$\F$.
 Then there is an admissible epimorphism of short exact sequences
$(V'\to P'\to X')\rarrow (V\to P\to X)$ in the category $\F$
and a morphism of short exact sequences
$(\eta_{\s\t}(V')\to\eta_{\s\t}(P')\to\eta_{\s\t}(X'))\rarrow
(\eta_{\s\t}(Y)\to\eta_{\s\t}(Q)\to\eta_{\s\t}(U))$
in the category $\F_{\s\t}$ whose images under the functor~$\eta_\s$
and the maps~$r_\t^0$ form a commutative triangle
with the morphism of short exact sequences
$(\eta_\s(V)\to\eta_\s(P)\to\eta_\s(X))\rarrow
(\eta_\s(Y)\to\eta_\s(Q)\to\eta_\s(U))$ in the category~$\F_\s$.
 Let $0\rarrow K\rarrow L\rarrow M\rarrow 0$ be the kernel of
the admissible epimorphism $(V'\to P'\to X')\rarrow (V\to P\to X)$
(in the exact category) of short exact sequences in~$\F$.
 Then the composition of morphisms of short exact sequences
$(\eta_{\s\t}(K)\to\eta_{\s\t}(L)\to\eta_{\s\t}(M))\rarrow
(\eta_{\s\t}(V')\to\eta_{\s\t}(P')\to\eta_{\s\t}(X'))\rarrow
(\eta_{\s\t}(Y)\to\eta_{\s\t}(Q)\to\eta_{\s\t}(U))$
is annihilated by the maps~$r_\t^0$, so,
by Lemmas~\ref{bockstein-r-zero}(a\+c) and
Lemma~\ref{bockstein-s-zero}(b\+c),
it comes from a (uniquely defined) morphism of short exact
sequences $(\eta_\t(K)\to\eta_\t(L)\to\eta_\t(M))\rarrow
(\eta_\t(Y)(-1)\to\eta_\t(Q)(-1)\to\eta_\t(U)(-1))$
in the category $\F_\t$ via the maps~$\s_0$.

 Consider the extension of short exact sequences $0\rarrow\eta_\t(V)
\rarrow\eta_\t(P)\rarrow\eta_\t(X)\rarrow0$ and $0\rarrow\eta_\t(K)
\rarrow\eta_\t(L)\rarrow\eta_\t(M)\rarrow0$ with the middle term
$0\rarrow\eta_\t(V')\rarrow\eta_\t(P')\rarrow\eta_\t(X')\rarrow0$
in (the exact category of short exact sequences in) the category
$\F_\t$, and induce from it an extension of the exact sequences
$0\rarrow\eta_\t(V)\rarrow\eta_\t(P)\rarrow\eta_\t(X)\rarrow0$
and $0\rarrow\eta_\t(Y)(-1)\rarrow\eta_\t(Q)(-1)\rarrow\eta_\t(U)(-1)
\rarrow0$ using the above-constructed morphism of short exact
sequences in~$\F_\t$.
 We have obtained a commutative $3\times3$ square formed by short
exact sequences in the exact category~$\F_\t$.
 For any such square, the two $\Ext^2$ classes between the objects
at the opposite vertices obtained by composing the $\Ext^1$ classes
along the perimeter differ by the minus sign (see
Lemma~\ref{ext-2-square}).
 This proves the desired equation
$-\eta_\t^1(a)(-1)\d^0(z) = \d^0(w)\eta_\t^1(b)$
in the group $\Ext^2_{\F_\t}(\eta_\t(X),\eta_\t(Y)(-1))$
(cf.~\cite[Subsection~4.5]{Partin}).
\end{proof}

\subsection{Exactness of the long sequence}
\label{exactness-long-sequence}
 Here we partly follow~\cite[Subsection~4.6]{Partin}.
 The argument is based on
Proposition~\ref{exact-surjectivity-ext-prop}(a).
 We start with the following lemma, which is in some way similar
to Lemma~\ref{bockstein-s-zero}(b).

\begin{lem}  \label{ext-1-secondary-product-s-compatibility}
 Let $0\rarrow K\rarrow L\rarrow M\rarrow 0$ and $0\rarrow U
\rarrow V\rarrow W\rarrow 0$ be two short exact sequences in
the category~$\F$, and let $f\:\eta_\t(K)\rarrow\eta_\t(V)(-1)$
and $g\:\eta_\t(L)\rarrow\eta_\t(W)(-1)$ be two morphisms in
the category $\F_\t$ forming a commutative square diagram with
the images of the morphisms $K\rarrow L$ and $V(-1)\rarrow W(-1)$
under the functor~$\eta_\t$.
 Let $w\in\Ext^1_{\F_\t}(\eta_\t(M),\eta_\t(U)(-1))$ denote
the related extension class provided by the construction of
Lemma~\ref{ext-1-secondary-product}(a), and let
$z\in\Ext^1_{\F_{\s\t}}(\eta_{\s\t}(M),\eta_{\s\t}(U))$
be the similar extension class produced by the same construction
applied to the images of the short exact sequences
$0\rarrow K\rarrow L\rarrow M\rarrow 0$ and $0\rarrow U
\rarrow V\rarrow W\rarrow 0$ under the functor~$\eta_{\s\t}$
and the morphisms $\s_0(f)\:\eta_{\s\t}(K)\rarrow\eta_{\s\t}(V)$
and $\s_0(g)\:\eta_{\s\t}(L)\rarrow\eta_{\s\t}(W)$
in the category~$\F_{\s\t}$.
 Then one has $z=\s_1(w)$.
\end{lem}

\begin{proof}
 Choose an admissible epimorphism $L'\rarrow L$ and a morphism
$L'\rarrow W(-1)$ in the category $\F$ making the triangle
diagram $\eta_\t(L')\rarrow\eta_\t(L)\rarrow\eta_\t(W)(-1)$
commutative in the category~$\F_\t$.
 Denote by $L''$ the fibered product of the objects $L'$ and
$V(-1)$ over $W(-1)$ in~$\F$.
 Let $K''$ denote the kernel of the composition of admissible
epimorphisms $L''\rarrow L'\rarrow L\rarrow M$; then there is
a natural morphism of short exact sequences $(K''\to L''\to M)
\rarrow (K\to L\to M)$ in the category~$\F$.
 By Lemma~\ref{ext-1-secondary-product}(d), both the classes
$w\in\Ext^1_{\F_\t}(\eta_\t(M),\eta_\t(U)(-1))$
and $z\in\Ext^1_{\F_{\s\t}}(\eta_{\s\t}(M),\eta_{\s\t}(U))$
can be obtained by applying the construction of
Lemma~\ref{ext-1-secondary-product}(a) to the images of
the short exact sequences $0\rarrow K''\rarrow L''\rarrow M\rarrow 0$
and $0\rarrow U\rarrow V\rarrow W\rarrow0$ under the functors
$\eta_\t$ and~$\eta_{\s\t}$ together with the composition of
morphisms of morphisms $(\eta_\t(K'')\to\eta_\t(L''))\rarrow
(\eta_\t(K)\to\eta_\t(L))\rarrow(\eta_\t(V)(-1)\to\eta_\t(W)(-1))$
and its image under the maps~$\s_0$.

 On the other hand, the images of the composition $K''\rarrow L''
\rarrow V(-1)$ and the morphism $L''\rarrow W(-1)$ under
the functor~$\eta_\t$ provide another commutative square
$(\eta_\t(K'')\to\eta_\t(L''))\rarrow(\eta_\t(V)(-1)\to\eta_\t(W)(-1))$
in the category~$\F_\t$.
 By Lemma~\ref{ext-1-secondary-product}(b), the $\Ext^1$ classes
assigned to this morphism of morphisms and its image under
the maps~$\s_0$ by the construction of
Lemma~\ref{ext-1-secondary-product}(a) vanish.
 Subtracting one morphism of morphisms from the other one and
applying Lemma~\ref{ext-1-secondary-product}(e), we reduce
the original problem to the case of a morphism of morphisms~$(f,g)$
with a vanishing second component~$g=0$.
 In this case, in view of the observation dual to
Lemma~\ref{ext-1-secondary-product}(c), the construction of
Lemma~\ref{ext-1-secondary-product}(a) reduces to that of
the product of a morphism with an $\Ext^1$ class; so the desired
assertion holds by the equations~(b) from
Subsection~\ref{bockstein-generalization} (or rather, even
by the definition of the map~$\s_1$).
\end{proof}

\begin{lem}  \label{bockstein-initial}
 The initial {segment \hfuzz=3pt
\begin{gather*}
 0\rarrow\Hom_{\F_\t}(\eta_\t(X),\eta_\t(Y)(-1))\rarrow
 \Hom_{\F_{\s\!\.\t}}(\eta_{\s\t}(X),\eta_{\s\t}(Y))\rarrow
 \Hom_{\F_\s}(\eta_\s(X),\eta_\s(Y)) \\
 \rarrow \Ext^1_{\F_\t}(\eta_\t(X),\eta_\t(Y)(-1))\rarrow
 \Ext^1_{\F_{\s\!\.\t}}(\eta_{\s\t}(X),\eta_{\s\t}(Y))\rarrow
 \Ext^1_{\F_{\s}}(\eta_\s(X),\eta_\s(Y))
\end{gather*}
of} the long sequence that we have constructed is exact for
any two objects $X$ and $Y$ in the category~$\F$.
\end{lem}

\begin{proof}
 Exactness at the first three nontrivial terms has been proven already
in Subsections~\ref{first-term-bockstein}\+-\ref{third-term-bockstein}.
 Let us prove exactness at the term
$\Ext^1_{\F_\t}(\eta_\t(X),\eta_\t(Y)(-1))$.

 First of all, we have to check that the composition of two maps
going through this term vanishes, that is $\s_1\d^0=0$.
 Let $q\:\eta_\s(X)\rarrow\eta_\s(Y)$ be a morphism in
the category~$\F_\s$.
 Choose an admissible epimorphism $X'\rarrow X$ with the kernel $K$
and a morphism $X'\rarrow Y$ in the category $\F$ making the triangle
diagram $\eta_\s(X')\rarrow\eta_\s(X)\rarrow\eta_\s(Y)$ commutative in
the category~$\F_\s$.
 Let $b\in\Ext^1_\F(X,K)$ denote the extension class of the short exact
sequence $0\rarrow K\rarrow X'\rarrow X\rarrow 0$ in the category~$\F$,
and let $p\:\eta_\t(K)\rarrow\eta_\t(Y)(-1)$ be a morphism in
the category $\F_\t$ such that the morphism $\s_0(p)\:\eta_{\s\t}(K)
\rarrow\eta_{\s\t}(Y)$ in the category $\F_{\s\t}$ is equal to
the image of the composition $K\rarrow X'\rarrow Y$ under
the functor~$\eta_{\s\t}$.
 By the definition, we have $\d^0(q)=p\eta_\t^1(b)
\in\Ext^1_{\F_\t}(\eta_\t(X),\eta_\t(Y)(-1))$.

 According to the equations~(b) from
Subsection~\ref{bockstein-generalization}, it follows that
$\s_1\d^0(q)=\s_1(p\eta_\t^1(b))=\s_0(p)\eta_{\s\t}^1(b)$.
 It remains to observe that the morphism $\s_0(p)\:\eta_{\s\t}(K)
\rarrow\eta_{\s\t}(Y)$ by construction factorizes through the admissible
monomorphism $\eta_{\s\t}(K)\rarrow\eta_{\s\t}(X')$ in the short exact
sequence $0\rarrow\eta_{\s\t}(K)\rarrow\eta_{\s\t}(X')\rarrow\eta_{\s\t}(X)
\rarrow0$, which represents the extension class $\eta_{\s\t}^1(b)\in
\Ext^1_{\F_{\s\t}}(\eta_{\s\t}(X),\eta_{\s\t}(K))$
in the category~$\F_{\s\t}$.
 Hence $\s_0(p)\eta_{\s\t}^1(b)=0$ in
$\Ext^1_{\F_{\s\t}}(\eta_{\s\t}(X),\eta_{\s\t}(Y))$.

 Let us show that the kernel of the map
$\s_1\:\Ext^1_{\F_\t}(\eta_\t(X),\eta_\t(Y)(-1))\rarrow
\Ext^1_{\F_{\s\t}}(\eta_{\s\t}(X),\eta_{\s\t}(Y))$ is contained in
the image of the map $\d^0\:\Hom_{\F_\s}(\eta_\s(X),\eta_\s(Y))
\allowbreak\rarrow\Ext^1_{\F_\t}(\eta_\t(X),\eta_\t(Y)(-1))$.
 According to Proposition~\ref{exact-surjectivity-ext-prop}(b) and
the condition~(i) for the functor~$\eta_\t$, any element~$z$ in
the group $\Ext^1_{\F_\t}(\eta_\t(X),\eta_\t(Y)(-1))$ is equal to
the product $p\eta_\t^1(b)$ of the image $\eta_\t^1(b)$ of an $\Ext^1$
class~$b$ represented by a short exact sequence $0\rarrow Y'\rarrow Z'
\rarrow X\rarrow 0$ in the category $\F$ under the functor~$\eta_\t$
and a morphism $p\:\eta_\t(Y')\rarrow\eta_\t(Y)(-1)$ in
the category~$\F_\t$. {\hbadness=1125\par}

 By the definition, the element $\s_1(z)\in
\Ext^1_{\F_{\s\!\t}}(\eta_{\s\t}(X),\eta_{\s\t}(Y))$
is constructed as the composition $\s_0(p)\eta_{\s\t}^1(b)$ of
the $\Ext^1$ class $\eta_{\s\t}^1(b)\in
\Ext^1_{\F_{\s\!\t}}(\eta_{\s\t}(X),\eta_{\s\t}(Y'))$ and the morphism
$\s_0(p)\:\eta_{\s\t}(Y')\rarrow\eta_{\s\t}(Y)$ in
the category~$\F_{\s\t}$.
 The equation $\s_1(z)=\s_0(p)\eta_{\s\t}^1(b)=0$ means that
the morphism $\s_0(p)$ factorizes through the morphism
$\eta_{\s\t}(Y')\rarrow\eta_{\s\t}(Z')$ in the category~$\F_{\s\t}$,
i.~e., there exists a morphism $\eta_{\s\t}(Z')\rarrow\eta_{\s\t}(Y)$
in $\F_{\s\t}$ making the triangle $\eta_{\s\t}(Y')\rarrow
\eta_{\s\t}(Z')\rarrow\eta_{\s\t}(Y)$ commutative.

 Applying the maps~$r_\t^0$ to the whole diagram in the category
$\F_{\s\t}$ with the objects in the vertices coming from
the category $\F$ via the functor~$\eta_{\s\t}$, we see that
the morphism $r_\t^0\s_0(p)\:\eta_\s(Y')\rarrow\eta_\s(Y)$ vanishes,
so the morphism $\eta_\s(Z')\rarrow\eta_\s(Y)$ factorizes through
the admissible epimorphism $\eta_\s(Z')\rarrow\eta_\s(X)$
and there exists a morphism $q\:\eta_\s(X)\rarrow\eta_\s(Y)$
in the category $\F_\s$ making the triangle $\eta_\s(Z')\rarrow
\eta_\s(X)\rarrow\eta_\s(Y)$ commutative.
 Let us check that $\d^0(q)=z$.

 Pick an admissible epimorphism $Z''\rarrow Z'$ and
a morphism $Z''\rarrow Y$ in the category $\F$ making
the triangle diagram $\eta_{\s\t}(Z'')\rarrow\eta_{\s\t}(Z')
\rarrow\eta_{\s\t}(Y)$ commutative in the category~$\F_{\s\t}$.
 Let $Y''$ denote the kernel of the composition of admissible
epimorphisms $Z''\rarrow Z'\rarrow X$; then there is
a natural morphism of short exact sequences $(Y''\to Z''\to X)
\rarrow (Y'\to Z'\to X)$ in the category~$\F$.
 Set $b'\in\Ext^1_\F(X,Y'')$ to be the $\Ext^1$ class represented
by the short exact sequence $0\rarrow Y''\rarrow Z''\rarrow X
\rarrow0$ in~$\F$, and denote by $p'\:\eta_\t(Y'')\rarrow
\eta_\t(Y)(-1)$ the composition of the image of the morphism
$Y''\rarrow Y'$ under the functor~$\eta_\t$ with the morphism~$p$.
 Then the triangle diagram $\eta_\s(Z'')\rarrow\eta_\s(X)\rarrow
\eta_\s(Y)$ is commutative in $\F_\s$ and the morphism
$\s_0(p')\:\eta_{\s\t}(Y'')\rarrow\eta_{\s\t}(Y)$ is equal to
the image of the composition $Y''\rarrow Z''\rarrow Y$ under
the functor~$\eta_{\s\t}$.
 By the definition, one has $\d^0(q)=p'\eta_\t^1(b')=
p\eta_\t^1(b)=z$.

 It remains to prove exactness of our sequence at the term
$\Ext^1_{\F_{\s\!\t}}(\eta_{\s\t}(X),\eta_{\s\t}(Y))$.
 Once again, firstly we have to check that the composition
$r_\t^1\s_1$ vanishes.
 Let $z\in\Ext_{\F_\t}^1(\eta_\t(X),\eta_\t(Y)(-1))$ be an extension
class, represented as above in the form $z=p\eta_\t^1(b)$.
 By the equations~(a) and~(b) of
Subsection~\ref{bockstein-generalization}, we have
$\s_1(z)=\s_0(p)\eta_{\s\t}^1(b)$ and $r_\t^1(\s_0(p)\eta_{\s\t}^1(b))=
r_\t^0(\s_0(p))\eta_\s^1(b)=0$, since $r_\t^0\s_0=0$ by
Lemma~\ref{bockstein-r-zero}(c).

 Let us show that the kernel of the map~$r_\t^1$ is contained in
the image of the map~$\s_1$.
 According to Proposition~\ref{exact-surjectivity-ext-prop}(b) and
the condition~(i) for the functor~$\eta_{\s\t}$, any element
$z\in\Ext^1_{\F_{\s\!\t}}(\eta_{\s\t}(X),\eta_{\s\t}(Y))$
can be presented as a product of the form $z=\eta_{\s\t}^1(a)f$,
where $a$~is an $\Ext^1$ class represented by a short exact
sequence $0\rarrow Y\overset u\rarrow Z'\overset v\rarrow X'\rarrow0$
in the category $\F$ and $f$~is a morphism $\eta_{\s\t}(X)\rarrow
\eta_{\s\t}(X')$ in the category~$\F_{\s\t}$. 
 The equation $r_\t^1(z)=\eta_\s^1(a)r_\t^0(f)=0$ in
$\Ext^1_{\F_\s}(\eta_\s(X),\eta_\s(Y))$ means that there exists
a morphism $\eta_\s(X)\rarrow\eta_\s(Z')$ in the category $\F_\s$
forming a commutative triangle with the image of the morphism
$Z'\rarrow X'$ under the functor~$\eta_\s$ and the image of
the morphism $f\:\eta_{\s\t}(X)\rarrow\eta_{\s\t}(X')$
under the map~$r_\t^0$.

 Applying the condition~(ii) for the functors~$\eta_\s$ and
$\eta_{\s\t}$, one can find an admissible epimorphism $e\:X''\rarrow X$
and morphisms $b\:X''\rarrow X'$, \ $c\:X''\rarrow Z'$ in
the category~$\F$ making the triangle diagram $\eta_{\s\t}(X'')\rarrow
\eta_{\s\t}(X) \rarrow\eta_{\s\t}(X')$ commutative in the category
$\F_{\s\t}$ and the triangle diagram $\eta_\s(X'')\rarrow\eta_\s(X)
\rarrow\eta_\s(Z')$ commutative in the category~$\F_\s$.
 Let $k\:K\rarrow X''$ denote the kernel of the morphism
$e\:X''\rarrow X$ in the category~$\F$; then the composition
$K\overset k\rarrow X''\overset b\rarrow X'$ is annihilated by
the functor~$\eta_{\s\t}$ and the composition $K\overset k\rarrow X''
\overset c\rarrow Z'$ is annihilated by the functor~$\eta_\s$.
 Furthermore, it follows that the difference of the composition
$X''\overset c\rarrow Z'\overset v\rarrow X'$ and the morphism
$b\:X''\rarrow X'$ in the category $\F$ is also annihilated by
the functor~$\eta_\s$, and the image of this difference under
the functor~$\eta_{\s\t}$ forms a commutative square with the image of
the composition $K\overset k\rarrow X''\overset c\rarrow Z'$ and
the images of the morphisms $k\:K\rarrow X''$ and $v\:Z'\rarrow X'$
under the same functor.

 By Lemma~\ref{bockstein-r-zero}(a,c), the images of our morphisms
$ck\:K\rarrow Z'$ and $vc-b\:X''\rarrow X'$ under
the functor~$\eta_{\s\t}$ are consequently equal to the images of
(uniquely defined) morphisms $\eta_\t(K)\rarrow\eta_\t(Z')(-1)$ and
$\eta_\t(X'')\rarrow\eta_\t(X')(-1)$ under the map~$\s_0$.
 Using Lemma~\ref{bockstein-s-zero}(b\+c), one checks that
the square diagram $\eta_\t(K)\rarrow\eta_\t(Z')(-1)\rarrow
\eta_\t(X')(-1)$, \ $\eta_\t(K)\rarrow\eta_\t(X'')\rarrow
\eta_\t(X')(-1)$, where the morphisms $\eta_\t(Z')\rarrow\eta_\t(X')$
and $\eta_\t(K)\rarrow\eta_\t(X'')$ come from the given morphisms
in the category~$\F$ via the functor~$\eta_\t$, is commutative
in the category~$\F_\t$.

 Applying the construction of Lemma~\ref{ext-1-secondary-product}(a)
to the images of the short exact sequences $0\rarrow K\rarrow X''
\rarrow X\rarrow 0$ and $0\rarrow Y\rarrow Z'\rarrow X'\rarrow 0$
under the functor~$\eta_\t$ and the above commutative square of
morphisms in the category $\F_\t$ produces the desired class
$w\in\Ext^1_{\F_\t}(\eta_\t(X),\eta_\t(Y)(-1))$, up to a minus sign.
 In fact, we will now see that the result of applying the same
construction to the images of the same short exact sequences under
the functor~$\eta_{\s\t}$ and the commutative square $\eta_{\s\t}(K)
\rarrow\eta_{\s\t}(Z')\rarrow\eta_{\s\t}(X')$, \ $\eta_{\s\t}(K)\rarrow
\eta_{\s\t}(X'')\rarrow\eta_{\s\t}(X')$ in the category~$\F_{\s\t}$ is
equal to minus the original class
$z\in\Ext^1_{\F_{\s\t}}(\eta_{\s\t}(X),\eta_{\s\t}(Y))$; so it remains to
use Lemma~\ref{ext-1-secondary-product-s-compatibility}.

 Our morphism of morphisms (commutative square) $(\eta_{\s\t}(ck),
\eta_{\s\t}(vc-b))\:(\eta_{\s\t}(K)\to\eta_{\s\t}(X''))\rarrow
(\eta_{\s\t}(Z')\to\eta_{\s\t}(X'))$ in the category $\F_{\s\t}$ can be
obtained by subtracting the morphism of morphisms
$(0,\eta_{\s\t}(b))\:(\eta_{\s\t}(K)\to\eta_{\s\t}(X''))\rarrow
(\eta_{\s\t}(Z')\to\eta_{\s\t}(X'))$ from the morphism of morphisms
$(\eta_{\s\t}(ck),\eta_{\s\t}(vc))\:(\eta_{\s\t}(K)\to\eta_{\s\t}(X''))
\rarrow(\eta_{\s\t}(Z')\to\eta_{\s\t}(X'))$.
 By Lemma~\ref{ext-1-secondary-product}(e), the resulting class in
$\Ext^1_{\F_{\s\t}}(\eta_{\s\t}(X),\eta_{\s\t}(Y))$ is equal to
the difference of the two related classes.
 By Lemma~\ref{ext-1-secondary-product}(b), the extension class
corresponding to the morphism of morphisms
$(\eta_{\s\t}(ck),\eta_{\s\t}(vc))$ vanishes.
 Finally, by part~(c) of the same Lemma, the extension class
corresponding to the morphism of morphisms
$(0,\eta_{\s\t}(b))$ is equal to $\eta_{\s\t}^1(a)f=z$.
\end{proof}

\begin{cor}
 The whole long sequence of\/ $\Ext$ groups from
Subsection~\ref{bockstein-generalization}, as constructed in
Subsection~\ref{higher-bockstein-construction}, is exact
for any two objects $X$, $Y$ of the category~$\F$.
\end{cor}

\begin{proof}
 The assertion follows formally from the construction and
Lemma~\ref{bockstein-initial} in view of
Proposition~\ref{exact-surjectivity-ext-prop}(a) applied
to the functors~$\eta_\s$, $\eta_\t$, and~$\eta_{\s\t}$.

 E.~g., let us prove exactness at the terms
$\Ext^n_{\F_\s}(\eta_\s(X),\eta_\s(Y))$ for all $n\ge1$.
 Let $z$~be an element of our Ext group in the category~$\F_\s$. 
 By Proposition~\ref{exact-surjectivity-ext-prop}(b) applied to
the functor~$\eta_\s$, there exists an Ext class $b\in\Ext^n_\F(X,X')$
in the category $\F$ and a morphism $q\:\eta_\s(X')\rarrow\eta_\s(Y)$
in the category $\F_\s$ such that the element~$z$ is equal to
the product $q\eta^n_\s(b)$ in the group
$\Ext^n_{\F_\s}(\eta_\s(X),\eta_\s(Y))$.
 By the definition, one has $\d^n(z)=\d^0(q)\eta_\t^n(b)$
in $\Ext^{n+1}_{\F_\t}(\eta_\t(X),\eta_\t(Y)(-1))$.

 Now assume that $\d^0(q)\eta_\t^n(b)=0$.
 By Proposition~\ref{exact-surjectivity-ext-prop}(a) applied
to the functor~$\eta_\t$, there exists a morphism $f\:X''\rarrow X'$
and an Ext class $b'\in\Ext^n_\F(X,X'')$ in the category $\F$
such that $b=fb'$ in $\Ext^n_\F(X,X')$ and $\d^0(q)\eta_\t(f)=0$
in $\Ext^1_{\F_\t}(\eta_\t(X''),\eta_\t(Y)(-1))$.
 By Lemma~\ref{d-n-bockstein} (or, actually, even
Lemma~\ref{bockstein-d-zero}(a)) one has $\d^0(q\eta_\s(f))=
\d^0(q)\eta_\t(f)=0$, and by Lemma~\ref{bockstein-initial} (or,
actually, Lemma~\ref{bockstein-d-zero}(b)), there exists
a morphism $g\:\eta_{\s\t}(X'')\rarrow\eta_{\s\t}(Y)$ in
the category $\F_{\s\t}$ such that $q\eta_\s(f)=r_\t^0(g)$
in the group $\Hom_{\F_\s}(\eta_\s(X''),\eta_\s(Y))$.
 Finally, we have $z=q\eta_\s^n(b)=q\eta_\s(f)\eta_\s^n(b')=
r_\t^0(g)\eta_\s^n(b')=r_\t^n(g\eta_{\s\t}^n(b'))$ in
$\Ext^n_{\F_\s}(\eta_\s(X),\eta_\s(Y))$.
\end{proof}

\Section{The Matrix Factorization Construction}
\label{matrix-factor-construct-section}

\subsection{Posing the problem} \label{reduction-posing}
 Let $\F$ be an exact category endowed with a twist functor
(exact autoequivalence) $X\mpsto X(1)$.
 Let $\s$ be a morphism of endofunctors $\Id\rarrow(1)$ on
the category $\F$ commuting with the twist functor
$(1)\:\F\rarrow\F$ (see Subsection~\ref{conventions} for
the precise definitions and discussion).

 Let $\E$ be another exact category endowed with a twist functor
$(1)\:\E\rarrow\E$.
 Suppose that we are given an exact functor $\pi\:\F\rarrow\E$
commuting with the twists on $\F$ and $\E$, and that the following
conditions are satisfied:
\begin{enumerate}
\renewcommand{\theenumi}{\roman{enumi}}
\setcounter{enumi}{4}
\item the functor~$\pi$ is exact-conservative;
\item the functor~$\pi$ takes all the morphisms
$\s_X\:X\rarrow X(1)$ in the category $\F$ to zero
morphisms in the category~$\E$;
\item any morphism in the category $\F$ annihilated by
the functor~$\pi$ is divisible by the natural transformation~$\s$;
\item for any object $X\in\F$, the morphism $\s_X\:X\rarrow X(1)$ 
is monic and epic; in other words, no nonzero morphism
in the category $\F$ is annihilated by the natural transformation~$\s$.
\end{enumerate} 

 The conditions~(vi) and~(vii) taken together can be restated
by saying that a morphism in the category $\F$ is annihilated
by the functor~$\pi$ if and only if it is divisible by
the natural transformation~$\s$.
 The conditions~(vii) and~(viii) taken together can be reformulated
by saying that any morphism in the category $\F$ annihilated
by the functor~$\pi$ is uniquely divisible by the natural
transformation~$\s$.
 Taken together, these are simply a restament of
the conditions~(\i\i\i\+\i\i\i\i) of Subsection~\ref{bockstein-toy}
(notice, however, that we also presume the exact-conservativity of
our functor~$\pi$ here, and most importantly, do \emph{not} assume
the conditions~(i\+ii)).

 Our goal in this section is to construct an exact category $\G$
which we will call the \emph{reduction of the exact category $\F$
by the natural transformation~$\s$ taken on the background of
the functor~$\pi$}.
 The category $\G$ comes endowed with exact-conservative functors
$\gamma\:\F\rarrow\G$ and $\eps\:\G\rarrow\E$ whose
composition~$\eps\gamma$ is identified with~$\pi$.
 The functor~$\gamma$ annihilates all the morphisms~$\s_X$, while
the functor~$\eps$ reflects zero morphisms (i.~e., it is faithful).
 The category $\G$ is also endowed with a twist functor
$(1)\:\G\rarrow\G$, and the functors~$\gamma$ and~$\eps$ commute
with the twists.

 The Ext groups computed in the categories $\F$ and $\G$ are
related by the following Bockstein long exact sequence
\begin{alignat*}{3}
 0&\lrarrow\Hom_\F(X,Y(-1))&&\lrarrow\Hom_\F(X,Y)&&\lrarrow
 \Hom_\G(\gamma(X),\gamma(Y)) \\
 &\lrarrow \Ext^1_\F(X,Y(-1))&&\lrarrow\Ext^1_\F(X,Y)&&\lrarrow
 \Ext^1_\G(\gamma(X),\gamma(Y)) \\
 &\lrarrow \Ext^2_\F(X,Y(-1))&&\lrarrow\Ext^2_\F(X,Y)&&\lrarrow
 \Ext^2_\G(\gamma(X),\gamma(Y)) \lrarrow\dotsb
\end{alignat*}
for any two objects $X$, $Y\in\F$ (cf.~\cite[Subsection~4.1]{Partin}).

 Here the map $\s_n\:\Ext^n_\F(X,Y(-1))\rarrow\Ext^n_\F(X,Y)$
is provided by the composition with the morphism $\s_{Y(-1)}\:
Y(-1)\rarrow Y$ (or, equivalently, the twist by~$(1)$ and
the composition with the morphism $\s_X\:X\rarrow X(1)$)
in the category~$\F$.
 The map $\gamma^n\:\Ext^n_\F(X,Y)\rarrow\Ext^n_\G(\gamma(X),\gamma(Y))$
is induced by the exact functor $\gamma\:\F\rarrow\G$.
 Finally, the boundary map $\d^n\:\Ext^n_\G(\gamma(X),\gamma(Y))
\rarrow\Ext^{n+1}_\F(X,Y(-1))$ is defined by the construction
of Subsections~\ref{third-term-bockstein}\+-%
\ref{higher-bockstein-construction} and satisfies the equation
$$
 \d^{i+n+j}(\gamma^i(a)z\gamma^j(b))=(-1)^ia(-1)\d^n(z)b
$$
for any objects $U$, $X$, $Y$, $V\in\F$ and any Ext classes
$b\in\Ext^j_\F(U,X)$, \ $z\in\Ext^n_\G(\gamma(X),\gamma(Y))$, and
$a\in\Ext^i_\F(Y,V)$.

\subsection{Examples} \label{reduction-examples}
 The following examples illustrate the ways of choosing 
a background exact-conservative functor~$\pi$ satisfying
the conditions~(vi\+vii) for a given natural transformation
$\s\:\Id\rarrow(1)$ on an exact category~$\F$.

\begin{ex}  \label{pro-finite-group-reduction-example}
 (a)~Let $G$ be a finite group and $m=l^r$ be a prime power.
 As in Example~\ref{finite-group-integral-integral-finite-example},
let $\F=\F_{\Z_l}^G$ be the exact category of finitely generated free
$\Z_l$\+modules with an action of~$G$.
 Set the twist $(1)\:\F\rarrow\F$ to be the identity functor
and the natural transformation $\s\:\Id_\F\rarrow\Id_\F$ to act
by the multiplication with~$m$.

 Let $\E$ be the category of finitely generated free $\Z/m$\+modules
with the split exact category structure, and $\pi\:\F\rarrow\E$
be the functor taking a finitely generated free $\Z_l$\+module $M$
with an action of $G$ to the $\Z/m$\+module $M/mM$.
 Then $\pi$~is an exact-conservative functor satisfying
the conditions~(vi\+vii) for the center element~$\s$ in~$\F$.

\medskip
 (b)~More generally, let $G$ be a profinite group.
 As in the introduction, let $\F^+=\F_{\Z_l}^G{}^+$ be the exact
category of $l$\+divisible $l^\infty$\+torsion ($l$\+primary)
abelian groups with a discrete action of~$G$.
 Set the twist~$(1)$ on $\F^+$ to be the identity functor and
the natural transformation~$\s$ to act by the multiplication with~$m$.

 Let $\E^+$ be the category of free $\Z/m$\+modules with the split
exact category structure, and $\pi\:\F^+\rarrow\E^+$ be the functor
taking an $l$\+divisible $l^\infty$\+torsion abelian group $M$ with
a discrete action of $G$ to the $\Z/m$\+module ${}_mM$ of all
the elements annihilated by~$m$ in~$M$.
 Then $\pi$~is an exact-conservative functor satisfying
the conditions~(vi\+vii) for the center element~$\s$ in~$\F^+$
(which satisfies the condition~(viii)).
 (See Subsection~\ref{reduction-rep-categories} below for further
discussion of this example.)
\end{ex}

\begin{ex} \label{associative-ring-reduction-example}
 (a)~The following example was also mentioned in the introduction
already.
 Let $R$ be an associative ring and $s\in R$ be a nonzero-dividing
central element.
 One can assume that $R$ is $s$\+complete, i.~e., the natural map
$R\rarrow\varprojlim_n R/s^nR$ is an isomorphism.
 Let $\F$ be the exact category of all $s$\+complete left $R$\+modules
without $s$\+torsion, i.~e., left $R$\+modules $M$ such that the map
$s\:M\rarrow M$ is injective and the natural map $M\rarrow\varprojlim_n
M/s^nM$ is an isomorphism.
 For $R$\+modules satisfying the former condition, the latter one is
equivalent to the condition that $M$ is an $s$\+contramodule, which
means that the $R$\+modules $\Hom_R(R[s^{-1}],M)$ and
$\Ext_R^1(R[s^{-1}],M)$ vanish~\cite[Theorem~2.4]{Pcta}.
 The exact category structure on $\F$ is inherited from the abelian
category of left $R$\+modules, in which $\F$ is a full subcategory closed
under extensions and kernels~\cite[Theorem~1.2(a)]{Pcta}.

 Let the twist $(1)$ on $\F$ be the identity functor and the natural
transformation~$\s$ act by the multiplication with~$s$.
 Let $\E$ be the abelian category of abelian groups and $\pi\:\F\rarrow\E$
be the functor assigning the abelian group $M/sM$ to an $R$\+module
$M\in\nobreak\F$.
 The functor~$\pi$ is exact and exact-conservative due to the conditions
of $s$\+torsion-freeness and $s$\+completeness imposed on
the $R$\+modules~$M$ (as one can show using the facts that the class of
$s$\+contramodule $R$\+modules is closed under the kernels and cokernels
in the category of $R$\+modules~\cite[Theorem~1.2(a)]{Pcta}) and
$P/sP=0$ implies $P=0$ for an $s$\+contramodule~$P$).
 The functor~$\pi$ also satisfies the conditions~(vi\+vii) for
the natural transformation~$\s$ (which satisfies the condition~(viii)).

\medskip
 (b)~Let $S$ be a filtered associative ring, $\dotsb\subset F^2S\subset
F^1S\subset F^0S\subset F^{-1}S\subset F^{-2}S\subset\dotsb\subset S$,
such that $1\in F^0S$, \ $F^iS\cdot F^jS\subset F^{i+j}S$ for all $i$,
$j\in\Z$, and $S=\bigcup_{n\to-\infty}F^nS=
\varprojlim_{n\to+\infty}S/F^nS$.
 Let $\F$ be the category of filtered left $S$\+modules, $\dotsb\subset
F^1N\subset F^0N\subset F^{-1}N\subset\dotsb\subset N$, such that
$F^iS\cdot F^jN\subset F^{i+j}N$ and $N=\bigcup_{n\to-\infty}F^nN=
\varprojlim_{n\to+\infty}N/F^nN$.
 The category $\F$ has a natural exact category structure in which
a short sequence with zero composition is exact if and only if
the induced short sequence of the associated graded groups/modules
$\gr_FN=\bigoplus_nF^nN/F^{n+1}N$ is.

 Let the twist functor $(1)\:\F\rarrow\F$ take a filtered $S$\+module
$(N,F)$ to the same $S$\+module $N$ with the shifted filtration
$F(1)^nN=F^{n-1}N$.
 The natural transformation $\s\:(N,F)\rarrow(N,F(1))$ acts by
the identity maps $N\rarrow N$, which are filtered $S$\+module morphisms
because $F^nN\subset F(1)^nN\subset N$.
 Let $\E$ be the abelian category of $\Z$\+graded abelian groups, and
let $\pi\:\F\rarrow\E$ be the exact functor taking a filtered $S$\+module
$(N,F)$ to the graded abelian group~$\gr_FN$.
 Then $\pi$~is an exact-conservative functor satisfying
the conditions~(vi\+vii) for the natural transformation
$\s\:\Id_\F\rarrow(1)$ (which satisfies the condition~(viii)).

\medskip
 Example~\ref{associative-ring-reduction-example}(b) is ``almost''
a particular case of
Example~\ref{associative-ring-reduction-example}(a):
assuming an additional $\Z$\+grading on the ring and all the modules
and abelian groups in~(a), one can describe~(b) as a particular case
of~(a) by passing to the Rees ring $R=\bigoplus_{n\in\Z}F^nS$ and
the Rees modules $M=\bigoplus_{n\in\Z}F^nN$.
 The element $s\in F^{-1}S\subset R$ of grading~$-1$ correponds to
the unit element $1\in F^0S\subset F^{-1}S$ viewed as an element of
$F^{-1}S$.
\end{ex}

\begin{ex}  \label{filtered-exact-associated-graded-reduction-example}
 (a)~Let $C$ be a coassociative coalgebra over a field~$k$ endowed
with a comultiplicative increasing filtration $0=F_{-1}C\subset F_0C
\subset F_1C\subset F_2C\subset\dotsb$.
 Set $F^{-n}C=F_nC$, and let $\F$ be the exact category of
finite-dimensional left $C$\+comodules $M$ endowed with finite
decreasing filtrations $\dotsb\supset F^{n-1}M\supset F^nM\supset
F^{n+1}M\supset\dotsb$, \,$F^{-n}M=M$ and $F^nM=0$ for $n\gg0$,
compatible with the filtration $F$ on~$C$.
 
 Let the twist functor $(1)\:\F\rarrow\F$ take a filtered
$C$\+comodule $(M,F)$ to the same $C$\+comodule $M$ with
the shifted filtration $F(1)^nM=F^{n-1}M$, and let
$\s_{(M,F)}\:(M,F)\rarrow(M,F(1))$ be the identity map $M\rarrow M$
viewed as a morphism of filtered comodules.

 Let $\E$ be the category of finite-dimensional graded $k$\+vector
spaces (with the split exact category structure), and let
$\pi\:\F\rarrow\E$ be the functor taking a filtered $C$\+comodule
$(M,F)$ to its associated graded vector space $\bigoplus_n
F^nM/F^{n+1}M$.
 Then $\pi$~is an exact-conservative functor satisfying
the conditions~(vi\+vii) for the natural transformation
$\s\:\Id\rarrow(1)$ on~$\F$.

\medskip
 In particular, when the coalgebra $C$ is endowed with a coaugmentation
(coalgebra morphism) $k\rarrow C$ and $F$ is the filtration by
the kernels of iterated comultiplication maps $F_nC =
\ker(C\to (C/k)^{\ot n+1})$ \cite[Section~2]{Partin}, this example is
naturally generalized to the next example~(b).
 Specifically, consider the abelian category $\A$ 
of left $C$\+comodules and the exact functor $\phi\:\E_0\rarrow\A$
of fully faithful embedding of the split exact category $\E_0$ of
finite-dimensional $k$\+vector spaces endowed with trivial
$C$\+comodule structures.
 Then the following filtered exact category construction yields
the above exact category $\F$ of filtered $C$\+comodules.

\medskip
 (b)~Let $\A$ and $\E_0$ be exact categories and $\phi\:\E_0\rarrow\A$
be an exact functor.
 Set $\F$ to be the category whose objects are the triples $(M,Q,q)$,
where $M=(M,F)$ is a finitely filtered object of the exact
category~$\A$, \ $Q=(Q_i)$ is a finitely supported object of
the Cartesian product $\prod_{i\in\Z}\E_0$, and $q$~is an isomorphism
$\gr_FM\simeq\phi(Q)$ of graded objects in~$\A$
\cite[Section~3 and Examples~A.5(4\+5)]{Partin}.

 The category $\F$ is endowed with an exact category structure in
which a short sequence with zero composition is exact if the related
short sequence of the objects $Q_i$ is exact in $\E_0$ for each
$i\in\Z$.
 The twist functor $(1)\:\F\rarrow\F$ shifts the indices of both
the filtration $F$ on $M$ and the sequence of objects~$Q_i$, and
the natural transformation $\s\:\Id_\F\rarrow(1)$ acts by identity
on the underlying objects $M\in\A$ of filtered objects $(M,F)$
and by zero on the graded objects $Q\in\prod_{i\in\Z}\E_0$.

 Let $\E$ be the full exact subcategory of finitely supported
objects in $\prod_{i\in\Z}\E_0$ and $\pi\:\F\rarrow\E$ be
the functor taking a triple $(M,Q,q)$ to the graded object~$Q$.
 Then $\pi$~is an exact-conservative functor satisfying
the condition~(vi) for the graded center element~$\s$
in~$\F$ (cf.~\cite[Section~4]{Partin}).
 When the functor $\phi\:\E_0\rarrow\A$ is fully faithful,
the conditions~(vii\+viii) are also
satisfied~\cite[Example~4.1]{Partin}.
\end{ex}

 The categorical reduction construction of this section applied
to the exact categories, natural transformations, and
background exact-conservative functors in
Example~\ref{pro-finite-group-reduction-example}(a\+b), produces
the exact categories $\G=\F_{\Z/m}^G$ and $\G^+=\F_{\Z/m}^{G\,+}$
of finitely and infinitely generated free $\Z/m$\+modules with
a (discrete) action of~$G$ (cf.\
Proposition~\ref{unfiltered-reduction-prop} below).
 The same construction in
Example~\ref{associative-ring-reduction-example}(a) produces the abelian
category $\G$ of left $R/(s)$\+modules, in
Example~\ref{associative-ring-reduction-example}(b) it leads to
the abelian category $\G$ of graded modules over the graded ring
$\gr_FS$, and in
Example~\ref{filtered-exact-associated-graded-reduction-example}(a) it
produces the abelian category $\G$ of finite-dimensional graded
comodules over the graded coalgebra $\gr_FC=\bigoplus_n F_nC/F_{n-1}C$. 

 We are not aware of any alternative way to define or produce
the exact category $\G$ that one obtains as the output of
the categorical reduction construction applied in the case of
Example~\ref{filtered-exact-associated-graded-reduction-example}(b)
in general (even when, e.~g., $\E_0$ is a split exact category and
$\A$ is an abelian category).

\begin{ex} \label{filtered-permutational-coefficient-reduction-example}
 The following example appears in connection with coefficient
reduction and Bockstein sequences in exact categories of
Artin--Tate motives with finite coefficients.
 Let $G$ be a profinite group, $k$~be a complete Noetherian local
ring, and $c\:G\rarrow k^*$ be a continuous group homomorphism;
set $k(n)$ to denote the free $k$\+module $k$ in which $G$ acts
by the character~$c^n$.

 For any integer~$n$, denote by $\E_n^+=\E_{k,\,n}^{G\,+}$ the category
of all injective discrete $k$\+modules $E$ with a discrete action
of $G$ for which the $G$\+module $E(-n)=k(-n)\ot_k E$ is permutational
(i.~e., isomorphic to a direct sum of $G$\+modules induced from
trivial representations of open subgroups $H\subset G$ in
injective discrete $k$\+modules).
 Endow the additive category $\E_n^+$ with the split exact
category structure.

 Let $\F^+_k=\F_k^{G\,+}$ denote the exact category of finitely
filtered discrete $G$\+modules $(M,F)$ over~$k$ whose associated
graded modules $\gr_F^nM=F^nM/F^{n+1}M$ belong to~$\E_n^+$
(cf.\ the construction of exact category $\F$ in
Example~\ref{filtered-exact-associated-graded-reduction-example}(b)).
 Set the twist $(1)\:\F^+_k\rarrow\F^+_k$ to be the identity
functor and the natural transformation $\s\:\Id_{\F_k^+}\rarrow
\Id_{\F_k^+}$ to act by the multiplication with a fixed
nonzero-dividing, noninvertible element $s\in k$.

 The quotient ring $k/s=k/(s)$ is also a complete Noetherian local
ring, and the submodule ${}_sM$ of all elements annihilated by~$s$
in any injective discrete $k$\+module $M$ is an injective discrete
module over~$k/s$.
 Let ${}'\!\.\E^+_{k/s}$ denote the full subcategory of finitely
supported graded objects in $\prod_{n\in\Z}\E_{k/s,\,n}^{G\,+}$, and
let ${}''\!\.\E^+_{k/s}$ denote the category of injective discrete
modules over~$k/(s)$.
 Notice that the objects of ${}'\!\.\E^+_{k/s}$ are collections of
$c^n$\+twisted permutational $G$\+modules over~$k/s$, while
the objects of ${}''\!\.\E^+_{k/s}$ are simply $k/s$\+modules without
any action of~$G$.

 Endow both the additive categories ${}'\!\.\E^+_{k/s}$ and 
${}''\!\.\E^+_{k/s}$ with split exact category structures, and
denote by $\E^+_{k/s}$ their Cartesian product ${}'\!\.\E^+_{k/s}
\times{}''\!\.\E^+_{k/s}$.
 Let the functor $\pi\:\F_k^+\rarrow\E_{k/s}^+$ take a finitely
filtered discrete $G$\+module $(M,F)$ over~$k$ with 
twisted-permutational $k$\+injective $k$\+discrete
associated quotient modules $\gr^n_FM$ to the pair formed by
the collection of twisted-permutational representations
${}_s\.\gr^n_FM=\gr_F^n\,{}_sM\in\E_{k/s,\,n}^{G\,+}$
of the group $G$ over the ring~$k/s$ and the $k/s$\+module
${}_sM\in{}''\!\.\E^+_{k/s}$.  

 It is claimed that $\pi$~is an exact-conservative functor
satisfying the conditions~(vi\+vii) for the center element~$\s$
of the exact category~$\F$.
 The condition~(viii) is also satisfied.
\end{ex}

 In this example one would \emph{like} to obtain the exact category
$\F^+_{k/s}=\F_{k/s}^{G\,+}$ in the output of the reduction procedure
(implying, in particular, the Bockstein long exact sequences
relating the Ext groups in the exact categories $\F_k^+$
and $\F_{k/s}^+$), but this is not always true
(see Subsection~\ref{reduction-filtered-rep-category} below
for further details).

\subsection{Matrix factorizations}  \label{matrix-factorizations}
 In order to construct the desired exact category $\G$, consider
first the following category~$\widetilde\H$.
 The objects of $\widetilde\H$ are the diagrams $(U,V)$ of the form
$$
 V(-1)\lrarrow U\lrarrow V\lrarrow U(1)
$$
in the category $\F$, where the morphism $V\rarrow U(1)$
is obtained from the morphism $V(-1)\rarrow U$ by applying
the twist functor~$(1)$, while the two compositions
$V(-1)\rarrow U\rarrow V$ and $U\rarrow V\rarrow U(1)$ are
equal to the maps $\s_{V(-1)}$ and~$\s_U$, respectively.
 Morphisms $(U',V')\rarrow (U'',V'')$ in the category $\H$ are
the pairs of morphisms $U'\rarrow U''$ and $V'\rarrow V''$ in $\F$
making the whole diagram $(V'(-1)\to U'\to V'\to U'(1))\rarrow
(V''(-1)\to U''\to V''\to U''(1))$ commutative.

 Furthermore, consider the following full subcategory $\H\subset
\widetilde\H$.
 By the definition, a diagram $(U,V)\in\widetilde\H$ belongs to
the category $\H$ if the functor $\pi\:\F\rarrow\E$ (which, as
we recall, takes the morphisms $\s_{V(-1)}$ and~$\s_U$ in $\F$ to
zero morphisms in~$\E$) transforms it into an exact sequence
$\pi(V(-1))\rarrow \pi(U)\rarrow \pi(V)\rarrow \pi(U(1))$ in
the exact category~$\E$.
 The functor $\Delta\:\H\rarrow\E$ assigns to a diagram $(U,V)\in\H$
the image of the morphism $\pi(U)\rarrow\pi(V)$ in $\E$ (which is
well-defined due to the exactness condition imposed on the objects
of~$\H$).

 The category $\widetilde\H$ has a natural exact category structure
in which a short sequence of diagrams is exact if it is exact
in $\F$ at every term of the diagrams.
 The full subcategory $\H\subset\widetilde\H$ is closed under
extensions (since the full subcategory of exact sequences in $\E$ is
closed under extensions in the exact category of complexes in~$\E$);
so it inherits the induced exact category structure.
 The functor $\Delta\:\H\rarrow\E$ is an exact functor between
exact categories.

 When the category $\E$ is weakly idempotent complete
(``semi-saturated'') \cite{Neem}, \cite{Bueh},
\cite[Appendix~A]{Partin}, the full subcategory $\H\subset\widetilde\H$
is also closed under the operations of passage to the cokernels of
admissible monomorphisms and the kernels of admissible epimorphisms
(since so is the full subcategory of exact sequences in $\E$ in
the exact category of complexes in~$\E$).
 Hence, in this case, a morphism in $\H$ is an admissible monomorphism
(admissible epimorphism) if and only if it is an admissible monomorphism
(admissible epimorphism) in~$\widetilde\H$.

 For any exact category $\E$, a morphism of exact sequences in $\E$ is
an admissible epimorphism if and only if it is an admissible epimorphism
of complexes in $\E$ and the induced morphisms of the objects of
cocycles of the exact sequences are admissible epimorphisms in~$\E$.
 Thus a morphism in $\H$ is an admissible epimorpism if and only if
it is an admissible epimorphism in $\widetilde\H$ and applying
the functor~$\pi$ transforms it into a morphism of exact sequences
in $\E$ such that the induced morphisms of the objects of cocycles are
admissible epimorphisms in~$\E$.
 The admissible monomorphisms in $\H$ are described similarly.

 Let $\I\subset\H$ denote the ideal of morphisms in $\H$ annihilated
by $\Delta$.
 Consider the quotient category $\H/\I$ of the category $\H$ by
this ideal of morphisms, and let $\S\subset\H/\I$ denote
the multiplicative class of morphisms which the functor 
$\Delta\:\H/\I\rarrow\E$ transforms to isomorphisms in~$\E$.
 In fact, $\S$ is a saturated class of morphisms in~$\H/\I$.

\begin{lem} \label{admissible-epi-of-diagrams}
 Assume that the conditions~(v\+vi) of Subsection~\ref{reduction-posing}
hold.
 Let $(U,V)\rarrow(K,L)$ be a morphism in\/ $\H$ such that
$\Delta(U,V)\rarrow\Delta(K,L)$ is an admissible epimorphism in~$\E$.
 Consider the object $(L(-1),L)$ in $\H$ given by the diagram formed
by the morphisms\/ $\id_{L(-1)}\:L(-1)\rarrow L(-1)$, \
$\s_{L(-1)}\:L(-1)\rarrow L$, and\/ $\id_L\:L\rarrow L$.
 Then the morphism of diagrams $(U,V)\oplus(L(-1),L)\rarrow(K,L)$ is
an admissible epimorphism in the exact category\/ $\H$ naturally
isomorphic to the original morphism $(U,V)\rarrow(K,L)$ in
the category\/~$\H/\I$.
\end{lem}

\begin{proof}
 The object $(L(-1),L)\in\H$ is annihilated by the functor $\Delta$ by
the condition~(vi), so the object $(U,V)\oplus(L(-1),L)$ is isomorphic
to $(U,V)$ in $\H/\I$.
 In order to show that the morphism $(U,V)\oplus(L(-1),L)\rarrow(K,L)$
is an admissible epimorphism in $\H$, let us first check that it is
an admissible epimorphism in~$\widetilde\H$.

 A morphism of diagrams is an admissible epimorphism in $\widetilde\H$
if and only if it is an admissible epimorphism at every term of
the diagrams.
 The morphism $V\oplus L\rarrow L$ is the projection onto a direct
summand, so it remains to check that the morphism $U\oplus L(-1)\rarrow K$
is an admissible epimorphism.
 In view of the condition~(v), it suffices to show that the morphism
$\pi(U)\oplus\pi(L(-1))\rarrow\pi(K)$ is an admissible epimorphism
in~$\E$.

 The functor~$\pi$ transforms the diagrams $(U,V)$ and $(K,L)$ into
exact sequences in the category~$\E$, which we will denote by
$\dotsb\rarrow U^{-1}\rarrow U^0\rarrow U^1\rarrow U^2\rarrow\dotsb$ and
$\dotsb\rarrow K^{-1}\rarrow K^0\rarrow K^1\rarrow K^2\rarrow\dotsb$
(so $U^{-1}=\pi(V(-1))$, $U^0=\pi(U)$, $K^{-1}=\pi(L(-1))$,
$K^0=\pi(K)$, etc.).
 Let $T=\Delta(U,V)$ and $M=\Delta(K,L)$ denote the images of
the morphisms $U^0\rarrow U^1$ and $V^0\rarrow V^1$ in~$\E$.
 We have a morphism of complexes/exact sequences $U^\bu\rarrow K^\bu$
which, by assumption, induces an admissible epimorphism of the objects
of cocycles $T\rarrow M$.
 In this setting, let us show that $U^0\oplus K^{-1}\rarrow K^0$ is
an admissible epimorphism.

 Let $R$ and $N$ denote the images of the morphisms $U^{-1}\rarrow U^0$
and $K^{-1}\rarrow K^0$.
 We have short exact sequences $0\rarrow R\rarrow U^0\rarrow T\rarrow0$
and $0\rarrow N\rarrow K^0\rarrow M\rarrow0$ in the category $\E$,
and a morphism from the former short exact sequence to the latter one.
 Furthermore, we have a short exact sequence of morphisms $(0\to 0)
\rarrow(K^{-1}\to N)\rarrow (U^0\oplus K^{-1}\to K^0)\rarrow(U^0\to M)
\rarrow (0\to 0)$ in the category~$\E$.
 The morphism $U^0\rarrow M$ is the composition of admissible
epimorphisms $U^0\rarrow T\rarrow M$, hence also
an admissible epimorphism.
 The morphism $K^{-1}\rarrow N$ is an admissible epimorphism, too.
 Since the full subcategory of admissible epimorphisms is closed under
extensions is the exact category of morphisms in $\E$, the morphism
$U^0\oplus K^{-1}\rarrow K^0$ is an admissible epimorphism.

 We have shown that the morphism of diagrams $(U,V)\oplus(L(-1),L)
\rarrow(K,L)$ is an admissible epimorphism in~$\widetilde\H$.
 It remains to check that applying the functor~$\pi$ transforms it into
a morphism of exact sequences in $\E$ which induces admissible
epimorphisms on the objects of cocycles.
 The objects of cocycles of the exact sequence $K^\bu=\pi(K,L)$ are
$M$ and $N$ (and their twists), while the objects of cocycles of
the exact sequence $\pi((U,V)\oplus(L(-1),L))$ are $T$ and $R\oplus
K^{-1}$ (and their twists), respectively.
 The morphism $T\rarrow M$ is an admissible epimorphism by assumption,
and the morphism $R\oplus K^{-1}\rarrow N$ is an admissible epimorphism,
because the morphism $K^{-1}\rarrow N$ is.
\end{proof}

\begin{lem}  \label{ore-conditions}
 Assuming the conditions~(v\+vi) of Subsection~\ref{reduction-posing},
the class of morphisms $\S$ is localizing in the category $\H/\I$
(i.~e., it satisfies the left and right Ore conditions).
\end{lem}

\begin{proof}
 The argument essentially follows that in~\cite[Subsection~4.2]{Partin}.
 It is clear from the definitions of the classes $\I$ and $\S$ that
if any two morphisms $X\birarrow Y$ in $\H/\I$ have equal
compositions with a morphism $X'\rarrow X$ or $Y\rarrow Y'$
belonging to $\S$, then these two morphisms $X\birarrow Y$ are equal
to each other in~$\H/\I$.

 Let $(S,T)\rarrow(K,L)\larrow(U,V)$ be a pair of morphisms in $\H$
such that $\Delta((U,V)\allowbreak\rarrow(K,L))$ is an admissible 
epimorphism in~$\E$.
 Consider the object $(U',V')=(U,V)\oplus(L(-1),L)\in\H$.
 By Lemma~\ref{admissible-epi-of-diagrams}, the induced morphism
$(U',V')\rarrow(K,L)$ is an admissible epimorphism in $\H$ isomorphic
to the morphism $(U,V)\rarrow(K,L)$ in the quotient category~$\H/\I$.

 Let $(P,Q)$ be the fibered product of the objects $(S,T)$ and $(U',V')$
over $(K,L)$ in the category~$\H$.
 Since $\H\subset\widetilde\H$ is a full subcategory closed under
extensions and $(U',V')\rarrow(K,L)$ is an admissible epimorphism,
this is the same thing as the fibered product in the category
$\widetilde\H$, that is, the termwise fibered product of the diagrams.

 The square diagram of diagrams formed by the morphisms $(S,T)\rarrow
(K,L) \larrow(U',V')$ and $(S,T)\larrow(P,Q)\rarrow(U',V')$ is
commutative in~$\H$.
 Dropping the direct summand $(L(-1),L)$ annihilated by
the functor $\Delta$, one can transform it into a square diagram
$(S,T)\rarrow(K,L)\larrow(U,V)$, \ $(S,T)\larrow(P,Q)\rarrow(U,V)$
in the category $\H$ which is commutative modulo~$\I$.
 
 Since the functor $\Delta$ is exact, the object $\Delta(P,Q)$ is
the fibered product of $\Delta(S,T)$ and $\Delta(U',V')=\Delta(U,V)$
over $\Delta(K,L)$.
 In particular, if the morphism $(U,V)\rarrow(K,L)$ belongs to $\S$,
then so does the morphism $(P,Q)\rarrow(S,T)$.
 This proves a half of the Ore conditions, and the dual half can be
proven in the dual way.
\end{proof}

\subsection{Exact category structure} \label{exact-structure}
 We define the category $\G$ as the localization $(\H/\I)[\S^{-1}]$.
 By Lemma~\ref{ore-conditions}, $\G$ is an additive category and
the localization $\H\rarrow\G$ is an additive functor.
 The twist functor $(1)\:\G\allowbreak\rarrow\G$ is induced by
the obvious twist functor $(U,V)\mpsto(U(1),V(1))$ on the category~$\H$.
 The functor $\gamma\:\F\rarrow\G$ assigns to an object $X\in\F$
the diagram $(X,X)$ with the identity morphism $X\rarrow X$
in the middle.
 The functor $\eps\:\G\rarrow\E$ is induced by the functor~$\Delta$.

 The functor $\eps\:\G\rarrow\E$ is faithful, because the functor
$\Delta\:\H/\I\rarrow\E$ is faithful by the definition of~$\I$.
 Consequently, the additive functor~$\eps$ reflects monic and epic
morphisms (i.~e., any morphism in $\G$ whose image under~$\eps$ is monic
or epic in $\E$ is, respectively, monic or epic in~$\G$).
 The functor~$\eps$ is also conservative (reflects isomorphisms), since
the morphisms taken to isomorphisms in $\E$ by the functor $\Delta$
have been inverted in~$\G$.

\begin{rem} \label{morphisms-up-to-isomorphism}
 By construction, morphisms in the category $\G$ are represented by
either left or right fractions of morphisms in $\H/\I$, with
an arbitrary morphism in $\H/\I$ in the numerator and a morphism
from $\S$ in the denominator.
 However, in some contexts it suffices to consider a morphism in
the category $\G$ \emph{up to isomorphism in the category of
morphisms in~$\G$}.
 From this point of view, since the morphisms from $\S$ become
isomorphisms in $\G$, one observes that any morphism in $\G$ is
\emph{isomorphic} (though not equal) to a morphism coming from
a morphism in $\H/\I$.

 Moreover, since one can choose between representing morphisms in $\G$
by left or right fractions, the same observation holds when one of
the ends of a morphism is fixed and the other one is viewed up
to isomorphism.
 E.~g., let $(U,V)$ be an object of~$\H/\I$.
 Then for any morphism $g\:(X,Y)\rarrow(U,V)$ in $\G$ there
exists an isomorphism $(X',Y')\simeq(X,Y)$ in $\G$ such that
the composition $g'\:(X',Y')\rarrow(U,V)$ comes from a morphism
$h'\:(X',Y')\rarrow(U,V)$ in the category~$\H/\I$ (which, in turn,
comes from a morphism~$h''$ in the category~$\H$).
 In this sense we will say that, being allowed to replace the object
in the source by an isomorphic one, we can represent a morphism~$g$
in the category $\G$ by a morphism~$h''$ in the category~$\H$.
\end{rem}

\begin{lem}
 In the assumption of the conditions~(v\+viii) from
Subsection~\ref{reduction-posing},
the rule according to which a short sequence in the category $\G$
is said to be exact if its image under the functor~$\eps$ is exact
in $\E$ defines an exact category structure on~$\G$.
 Moreover, a morphism is an admissible monomorphism (resp.,
admissible epimorphism) in $\G$ if and only if its image
under~$\eps$ is an admissible monomorphism (resp., admissible
epimorphism) in~$\E$.
\end{lem}

\begin{proof}
 We follow the argument in~\cite[Subsection~4.3]{Partin}.
 Consider a morphism $f$ in $\G$ such that $\eps(f)$ is
an admissible epimorphism in~$\E$.
 Then the morphism $f$~is epic in the category~$\G$.
 Represent~$f$ by a morphism $(U,V)\rarrow(K,L)$ in $\H$ (as explained
in Remark~\ref{morphisms-up-to-isomorphism}) and apply the construction
of Lemma~\ref{admissible-epi-of-diagrams} to obtain an admissible
epimorphism $(U',V')=(U,V)\oplus(L(-1),L)\rarrow(K,L)$
in the category~$\H$.

 Let $(M,N)=(\ker(U\oplus L(-1)\to K)\;V)$ denote the kernel of
the morphism $(U',V')\rarrow (K,L)$ in~$\H$.
 Then $(0,0)\rarrow (M,N)\rarrow(U',V')\rarrow(K,L)\rarrow(0,0)$ is
a short exact sequence in~$\H$.
 Since the functor $\Delta\:\H\rarrow\E$ is exact, the short sequence
$0\rarrow\Delta(M,N)\rarrow\Delta(U',V')=\Delta(U,V)\rarrow\Delta(K,L)
\rarrow0$ is exact in~$\E$.

 We also have a natural morphism $(M,N)\rarrow(U,V)$ in~$\H$.
 Its image~$g$ in $\G$ completes the morphism~$f$ to
a short sequence $0\rarrow(M,N)\rarrow(U,V)\rarrow(K,L)\rarrow0$
that is exact in $\G$ (in the sense of our definition;
i.~e., its image under~$\eps$ is exact in~$\E$).
 Let us check that the morphism~$g$ is the kernel of~$f$ in~$\G$.

 Any morphism with the target $(U,V)$ in $\G$ can be represented by
a morphism $(X,Y)\rarrow(U,V)$ in~$\H$.
 Assume that the composition $(X,Y)\rarrow(U,V)\rarrow(K,L)$ is
annihilated by~$\Delta$.
 Then the composition $X\rarrow K\rarrow L$ is annihilated by~$\pi$,
hence it follows from the conditions~(vii\+viii) that the morphism
$X\rarrow K$ factorizes through the morphism $L(-1)\rarrow K$ in~$\F$.
 In other words, this means that the composition $(X,Y)\rarrow(U,V)
\rarrow(K,L)$, viewed as a morphism in $\H$, factorizes through
the natural morphism $(L(-1),L)\rarrow(K,L)$.
 This allows to lift the morphism $(X,Y)\rarrow(U,V)$ to a morphism
$(X,Y)\rarrow(M,N)$ in~$\H$.

 The morphism~$g$ being monic in $\G$, the above lifting is
unique as a morphism in~$\G$; so $g$~is the kernel of~$f$ in~$\G$.
 Now for any morphism $(S,T)\rarrow(U,V)$ in $\G$ such that
the short sequence $0\rarrow\eps(S,T)\rarrow\eps(U,V)\rarrow\eps(K,L)
\rarrow0$ is exact in~$\E$, the induced morphism $(S,T)\rarrow(M,N)$
is transformed to an isomorphism by the functor~$\eps$.
 Hence $(S,T)\rarrow(M,N)$ is an isomorphism in $\G$ and $(S,T)\rarrow
(U,V)$ is a kernel of~$f$ in~$\G$.
 By the dual argument, the morphism~$f$ is a cokernel of $(S,T)\rarrow
(U,V)$.

 This suffices to check the axioms~Ex0--Ex1 and Ex3
from~\cite[Subsection~A.3]{Partin} for the category~$\G$;
it remains to prove~Ex2.
 Suppose that we are given a short exact sequence in $\G$; according
to the above, it can be represented, up to an isomorphism of short exact
sequences in $\G$, by a short exact sequence of the form
$(M,N)\rarrow(U',V')\rarrow(K,L)$ in~$\H$.
 Any morphism with the target $(K,L)$ in $\G$ can be represented by
a morphism $(X,Y)\rarrow(K,L)$ in~$\H$.

 As in the proof of Lemma~\ref{ore-conditions}, we consider the fibered
product $(P,Q)$ of the objects $(X,Y)$ and $(U',V')$ over $(K,L)$ 
in the category~$\H$.
 Then we have a commutative triangle $(M,N)\rarrow(P,Q)\rarrow(U',V')$
and a short exact sequence $(0,0)\rarrow(M,N)\rarrow(P,Q)\rarrow(X,Y)
\rarrow(0,0)$ in~$\H$.
 Since the functor $\Delta$ is exact, the image of any short exact
sequence in $\H$ is a short exact sequence in~$\G$.
% Dropping the ``trivial'' summand $(L(-1),L)$ in $(U',V')$, we see
%that the triangle $(M,N)\rarrow(P,Q)\rarrow(U,V)$ is commutative
%already in $\H$, while the square $(X,Y)\rarrow(K,L)\larrow(U,V)$, \
%$(X,Y)\larrow(P,Q)\rarrow(U,V)$ is commutative modulo~$\I$.
 We have constructed the desired diagram in the category~$\G$, and
the exact category axioms are verified.
\end{proof}

 It follows immediately from the above description of the exact
category structure on $\G$ that the functor $\gamma\:\F\rarrow\G$
is exact and exact-conservative (since both the functors
$\eps\:\G\rarrow\E$ and $\pi=\eps\gamma\:\F\rarrow\E$ are).
 It is also clear that the functors~$\eps$ and~$\gamma$ commute
with the twists.
 The more advanced properties of our reduction construction are
discussed in the next subsection.

\subsection{Properties of the reduction functor}
\label{matrix-reduction-properties}
 The following lemma shows that the reduction functor $\gamma\:\F
\rarrow\G$ satisfies the conditions~(i\+ii) and~($*\.$\+-$\.**$)
from Subsection~\ref{exact-surjectivity} (and even a stronger form
of the last one of these).

\begin{lem}  \label{gamma-properties}
 The exact functor $\gamma\:\F\rarrow\G$ constructed above has
the following ``exact surjectivity'' properties:
\begin{enumerate}
\renewcommand{\theenumi}{\alph{enumi}$'$}
\item for any object $X\in\F$ and any admissible epimorphism
$T\rarrow\gamma(X)$ in $\G$ there exists an admissible epimorphism
$Z\rarrow X$ in $\F$ and a morphism $\gamma(Z)\rarrow T$ in $\G$
making the triangle diagram $\gamma(Z)\rarrow T\rarrow\gamma(X)$
commutative;
\item for any object $T\in\G$ there exists an object $U\in\F$
and an admissible epimorphism $\gamma(U)\rarrow T$ in~$\G$;
\item for any objects $X$, $Y\in\F$ and any morphism $\gamma(X)
\rarrow\gamma(Y)$ in $\G$ there exists an admissible epimorphism
$X'\rarrow X$ and a morphism $X'\rarrow Y$ in $\F$ making
the triangle diagram $\gamma(X')\rarrow\gamma(X)\rarrow\gamma(Y)$
commutative in~$\G$;
\item for any object $X\in\F$ and any morphism $\gamma(X)\rarrow T$
in $\G$ there exists a morphism $X\rarrow S$ in $\F$ and
an admissible epimorphism $\gamma(S)\rarrow T$ in $\G$ such that
the triangle diagram $\gamma(X)\rarrow\gamma(S)\rarrow T$
is commutative;
\end{enumerate}
as well as the properties~(a$\.''$\+d$\,''$) dual to~(a$\.'$\+d$\,'$).
\end{lem}

\begin{proof}
 Part~(b$'$): let the object $T\in\G$ be represented by a diagram
$(U,V)=(V(-1)\to U\to V\to U(1))$ in $\H$; then there is a natural
admissible epimorphism $\gamma(U)\rarrow T$ in~$\G$.
 (Indeed, $\pi(U)\rarrow\Delta(U,V)$ is an admissible epimorphism
in~$\E$.)
 Part~(c$'$): let the morphism $\gamma(X)\rarrow\gamma(Y)$ be
represented by a fraction $(X,X)\larrow(U,V)\rarrow(Y,Y)$ of
two morphisms in $\H$, where the morphism $(U,V)\rarrow(X,X)$
belongs to~$\S$ (modulo~$\I$).
 Then there is a natural admissible epimorphism $U\rarrow X$
and a natural morphism $U\rarrow Y$ in $\F$ making the diagram
$\gamma(U)\rarrow\gamma(X)\rarrow\gamma(Y)$ commutative in
$\G$ by the definition.
 (Indeed, the morphism $\pi(U)\rarrow\Delta(U,V)\simeq
\Delta(X,X)=\pi(X)$ is an admissible epimorphism in~$\E$.)

 Part~(d$'$): one can represent the morphism $\gamma(X)\rarrow T$
in $\G$ by a morphism of diagrams $(X,X)\rarrow (U,V)$ in~$\H$.
 Then there is a morphism $X\rarrow U$ in $\F$ and an admissible
epimorphism $\gamma(U)\rarrow(U,V)$ in $\G$, while the triangle
$(X,X)\rarrow(U,U)\rarrow(U,V)$ is commutative already in~$\H$.
 Part~(a$'$) follows from~(b$'$) and~(c$'$) according to
Lemma~\ref{exact-surjectivity-implications}; to prove it directly,
represent the morphism $T\rarrow\gamma(X)$ in $\G$ by a morphism
of diagrams $(U,V)\rarrow(X,X)$ in~$\H$.
 Then $U\rarrow X$ is an admissible epimorphism in $\F$ (since
$\pi(U)\rarrow\Delta(U,V)\rarrow\pi(X)$ is a composition of
admissible epimorphisms in~$\E$) and $(U,U)\rarrow(U,V)$ is
an admissible epimorphism in $\G$, while the triangle
$(U,U)\rarrow(U,V)\rarrow(X,X)$ is commutative in~$\H$.
\end{proof}

\subsection{The first Bockstein sequence}  \label{first-bockstein}
 Now we are in the position to construct the Bockstein long
exact sequence for the Ext groups in the categories $\F$ and $\G$
promised in Subsection~\ref{reduction-posing}.
 Namely, we obtain the desired construction as a particular case
of the general construction of a Bockstein long exact sequence
from Section~\ref{bockstein-sequence-section}, or more specifically,
of its version from Subsection~\ref{bockstein-toy}.

 To compare our present notation with that in 
Section~\ref{bockstein-sequence-section}, set $\F_\t=\F_{\s\t}=\F$
and $\F_\s=\G$.
 The exact category structures on $\F_\t$, \,$\F_{\s\t}$, and
$\F_\s$ are the same as the exact category structures on $\F$
and~$\G$.
 Let $\eta_\t\:\F\rarrow\F_\t$ and $\eta_{\s\t}\:\F\rarrow\F_{\s\t}$
be the identity functor $\Id_\F$, and $\eta_\s\:\F\rarrow\F_\s$
be our reduction functor~$\gamma$.
 Finally, let the natural transformation $\s\:\Id\rarrow(1)$ on
the category $\F$ from Subsection~\ref{bockstein-toy}
be our natural transformation $\s\:\Id\rarrow(1)$ on
the category~$\F$ from Subsection~\ref{reduction-posing}.

 Let us check that the conditions~(i\+\i\i\i\i) of
Subsections~\ref{exact-surjectivity}\+-\ref{bockstein-toy}
hold in the situation at hand.
 The functor~$\gamma$ satisfies the conditions~(i$'$\+ii$'$)
according to Lemma~\ref{gamma-properties}(a$'$,c$'$), and
the dual conditions~(i$''$\+ii$''$) according to the dual
assertions of the same Lemma.
 The condition~(\i\i\i) coincides with the condition~(viii); and
the condition~(\i\i\i\i) of Subsection~\ref{bockstein-toy} for
the functor~$\gamma$ follows from the conditions~(vi\+vii) of
Subsection~\ref{reduction-posing} for the functor~$\pi$
(since the functor $\pi\:\F\rarrow\E$ decomposes as $\F\rarrow
\G\rarrow\E$ and the functor $\eps\:\G\rarrow\E$ does not
annihilate any morphisms).

 The Bockstein long exact sequence promised in
Subsection~\ref{reduction-posing} is now obtained.
 The equation for the maps~$\d^n$ from
Subsection~\ref{reduction-posing} is
the equation~(c) of Subsection~\ref{bockstein-toy}.

\subsection{Independence of the background functor}
\label{reduction-independence}
 The following arguments show that the construction of the reduction
$\G$ of an exact category $\F$ with a twist functor $(1)\:\F\rarrow\F$
by a natural transformation $\s\:\Id_\F\rarrow(1)$ ``almost'' does not
depend on the background exact-conservative functor $\pi\:\F\rarrow\E$.

 Let $\F$ be an exact category with a twist functor $(1)\:\F\rarrow\F$
and a natural transformation $\s\:\Id_\F\rarrow(1)$ commuting with
the twist.
 Suppose that $\G'$ and $\G''$ are two exact categories with the twist
functors, and $\gamma'\:\F\rarrow\G'$ and $\gamma''\:\F\rarrow\G''$
are two exact functors, commuting with the twists, for which
the boundary maps~$\d^n$ are defined between the Ext groups in
the categories $\G'$, $\G''$ and $\F$ making the long sequence of
Subsection~\ref{reduction-posing} exact in both cases.

 Suppose further that there is an exact functor
$\lambda\:\G'\rarrow\G''$ commuting with the twists,
making a commutative triangle diagram with the functors~$\gamma'$
and~$\gamma''$, and inducing a morphism between the long exact
sequences of Subsection~\ref{reduction-posing}
corresponding to the functors~$\gamma'$ and~$\gamma''$.
 Let $\overline\G'$ and $\overline\G''$ denote the minimal full
subcategories in $\G'$ and $\G''$, closed under extensions and
containing the all the objects in the image of the functors~$\gamma'$
and~$\gamma''$, respectively.
 Endow the full subcategories $\overline\G'\subset\G'$ and
$\overline\G''\subset\G''$ with the induced exact category structures.

\begin{lem}  \label{reductions-comparison}
\textup{(a)} The exact functor $\lambda\:\G'\rarrow\G''$ restricts to
an equivalence $\overline\G'\simeq\overline\G''$ between
the full exact subcategories $\overline\G'$ and $\overline\G''$
in $\G'$ and~$\G''$. \par
\textup{(b)} Assume that the functor~$\gamma'$ or~$\gamma''$
satisfies one of the conditions~(i$\.'$) or~($i\.''$) of
Subsection~\ref{exact-surjectivity}.
 Then the Ext groups between the objects of the corresponding
subcategory $\overline\G'$ or $\overline\G''$ computed in the exact
subcategory $\overline\G'$ or $\overline\G''$ concide with those
computed in the whole exact category $\G'$ or $\G''$, respectively. \par
\textup{(c)} Assume that both functors $\gamma'$ and~$\gamma''$
simultaneously satisfy one of the conditions~(i$\.'$) or~($i\.''$),
and the functor~$\gamma'$ satisfies both the conditions~($*'$)
and~($*''$) of Subsection~\ref{exact-surjectivity}.
 Then the functor~$\lambda$ is fully faithful, its image $\lambda(\G')$
is a full subcategory closed under extensions in $\G''$,
the exact category structure on $\G'$ coincides with the one
induced from $\G''$ via~$\lambda$, and the functor~$\lambda$ induces
isomorphisms $\lambda^n\:\Ext^n_{\G'}(Z,W)\simeq
\Ext^n_{\G''}(\lambda(Z),\lambda(W))$ for all objects $Z$,
$W\in\G'$ and all\/~$n\ge0$.
\end{lem}

\begin{proof}
 Notice first of all that the maps $\lambda^n\:\Ext^n_{\G'}(\gamma'(X),
\gamma'(Y))\rarrow\Ext^n_{\G''}(\gamma''(X),\allowbreak\gamma''(Y))$
are isomorphisms for all $X$, $Y\in\F$ and $n\ge0$ by 5\+lemma.

 In the situation of part~(a), the maps
$\Ext^n_{\overline\G'}(Z,W)\rarrow\Ext^n_{\G'}(Z,W)$ are isomorphisms
for all $Z$, $W\in\overline\G'$ and $n=0$,~$1$, and monomorphisms
for $n=2$, because $\overline\G'$ is a full subcategory closed under
extensions in the exact category $\G'$, with the induced
exact category structure~\cite[Subsection~A.8]{Partin}.
 The same applies to the maps of Ext groups induced by the embedding
$\overline\G''\rarrow\G''$; and consequently the maps
$\lambda^n\:\Ext^n_{\overline\G'}(\gamma'(X),\gamma'(Y))\rarrow
\Ext^n_{\overline\G''}(\gamma''(X),\gamma''(Y))$ are
also isomorphisms for all $X$, $Y\in\F$ and $n=0$,~$1$, and
monomorphisms for $n=2$.
 The desired assertion now follows by the general criterion
of~\cite[Lemma~3.2]{Partin}.

 To obtain part~(b), one applies 
Corollary~\ref{exact-surjectivity-ext-cor2}(b) in its first
set of assumptions, and to prove part~(c), in the second one.
\end{proof}

 Given an exact category $\F$ with a twist functor~$(1)$ and
a natural transformation~$\s$ commuting with the twist
and satisfying the condition~(viii), let $\E'$ and $\E''$ be two
exact categories endowed with twist functors, and let
$\pi'\:\F\rarrow\E'$ and $\pi''\:\F\rarrow\E''$ be two exact
functors, both commuting with the twists and satisfying
the conditions~(v\+vii) of Subsection~\ref{reduction-posing}.
 Let $\G'$ and $\G''$ denote the corresponding two exact categories
obtained by the reduction procedure of
Subsections~\ref{matrix-factorizations}\+-\ref{exact-structure},
and $\gamma'\:\F\rarrow\G'$, \ $\gamma''\:\F\rarrow\G''$ be
the two related exact functors.

 Suppose first that there exists an exact functor $\E'\rarrow\E''$
commuting with the twists and making a commutative triangle diagram
with the functors $\pi'$ and~$\pi''$.
 Then one easily constructs the induced exact functor
$\lambda\:\G'\rarrow\G''$, commuting with the twists, making
a commutative triangle diagram with the functors $\gamma'$
and~$\gamma''$, and satisfying the assumptions of
Lemma~\ref{reductions-comparison} (while the functors $\gamma'$
and~$\gamma''$ satisfy the conditions of
Lemma~\ref{reductions-comparison}(c) as well by
Lemma~\ref{gamma-properties}).
 All the conclusions of Lemma~\ref{reductions-comparison}(a\+c)
accordingly apply.

 More generally, denote by $\E=\E'\times\E''$ the Cartesian product
of the two exact categories $\E'$ and~$\E''$.
 The objects of $\E$ are pais $(E',E'')$, where $E'$ is an object of
$\E'$ and $E''$ is an object of $\E''$; morphisms $(E'_1,E''_1)\rarrow
(E'_2,E''_2)$ are pairs of morphisms $E'_1\rarrow E'_2$ and
$E''_1\rarrow E''_2$ in $\E'$ and $\E''$; and short exact sequences
in $\E$ are pairs of short exact sequences in the exact categories
$\E'$ and $\E''$ (cf.~\cite[Example~A.5(4)]{Partin}).

 Let $\pi=(\pi',\pi'')\:\F\rarrow\E$ denote the functor taking
an object $X\in\F$ to the pair $(\pi'(X),\pi''(X))$; then
$\pi$~is also an exact functor satisfying the conditions~(v\+vii)
for the natural transformation $\s\:\Id_\F\rarrow(1)$.
 Let $\G$ denote the reduction of the exact category $\F$
by the natural transformation~$\s$ taken on the background
of the functor~$\pi$, and let $\gamma\:\F\rarrow\G$ be
the corresponding exact functor.

 Then the natural projections $\E\rarrow\E'$, $\E''$ forming
commutative triangle diagrams with the exact-conservative
functors~$\pi'$, $\pi''$, and~$\pi$ induce exact functors
$\lambda'\:\G\rarrow\G'$ and $\lambda''\:\G\rarrow\G''$
between the reduced exact categories.
 The functors~$\lambda'$ and $\lambda''$ form commutative
triangle diagrams with the reduction functors~$\gamma'$,
$\gamma''$, and~$\gamma$.
 All the conclusions of Lemma~\ref{reductions-comparison}(a\+c)
apply to both the functors~$\lambda'$ and $\lambda''$, leading
in particular to the following corollary.

\begin{cor}
 The exact functors $\lambda'\:\G\rarrow\G'$ and
$\lambda''\:\G\rarrow\G''$ are fully faithful, and their images
are full subcategories closed under extensions in $\G'$ and~$\G''$.
 The exact category structure on $\G$ coincides with the exact
category structures induced on the full subcategories
$\lambda'(\G)\subset\G'$ and $\lambda''(\G)\subset\G''$ by
the exact category structures on $\G'$ and~$\G''$.
 The induced maps of the Ext groups $\lambda'\.{}^n\:\Ext^n_\G(Z,W)
\rarrow\Ext^n_{\G'}(\lambda'(Z),\lambda'(W))$ and $\lambda''\.{}^n\:
\Ext^n_\G(Z,W)\rarrow\Ext^n_{\G''}(\lambda''(Z),\lambda''(W))$
are isomorphisms for all the objects $Z$, $W\in\G$ and all
integers~$n\ge0$. 

 The restrictions of the functors $\lambda'$ and~$\lambda''$ to
the minimal full subcategories $\overline\G\subset\G$, \
$\overline\G'\subset\G'$, \ $\overline\G''\subset\G''$
containing all the objects in the images of the functors
$\gamma$, $\gamma'$, $\gamma''$ and closed under extensions
are equivalences of exact categories $\overline\G'\simeq
\overline\G\simeq\overline\G''$ (with the exact category
structures induced from $\G$, \,$\G'$,~\,$\G''$).
 The Ext groups between the objects of the subcategories
$\overline\G$, \,$\overline\G'$, \,$\overline\G''$ computed
in these exact subcategories coincide with the Ext groups computed
in the whole exact categories $\G$, $\,\G'$,~\,$\G''$. \qed
\end{cor}

 Based on the above ``almost uniqueness''/``almost independence''
results, we will sometimes denote the exact category obtained by
the reduction procedure developed in this section by $\G=\F/\s$.
 When a specific choice of the background exact-concervative functor
needs to be mentioned, we will write
$\G=\F/_{\!\!{}_\E}\.\s=\F/_{\!\!\pi}\.\s$.

\subsection{The second Bockstein sequence}  \label{second-bockstein}
 Let $\F$ be an exact category with two commuting exact
autoequivalences $X\mpsto X(1)$ and $X\mpsto X\{1\}$.
 Let $\s\:\Id\rarrow(1)$ and $\t\:\Id\rarrow\{1\}$ be two natural
tranformations of endofunctors on~$\F$, both commuting with both
the twist functors $(1)$ and~$\{1\}$ and acting by morphisms
$\s_X\:X\rarrow X(1)$ and $\t_X\:X\rarrow X\{1\}$ that are 
monic and epic for all the objects $X\in\F$.
 Denote by $\s\t\:\Id_\F\rarrow(1)\{1\}$ the product of these
two bigraded center elements of~$\F$.

 Let $\E_\s$, $\E_\t$, and $\E_{\s\t}$ be three exact categories
endowed with commuting exact autoequivalences $(1)$ and~$\{1\}$,
and let $\pi_\s\:\F\rarrow\E_\s$, \,$\pi_\t\:\F\rarrow\E_\t$, and
$\pi_{\s\t}\:\F\rarrow\E_{\s\t}$ be three exact-conservative
functors, commuting with the twist functors $(1)$ and~$\{1\}$
and satisfying the conditions~(vi\+vii) of
Subsection~\ref{reduction-posing} for the natural transformations~$\s$,
$\t$, and~$\s\t$, respectively.
 Consider the three reduced exact categories
$\G_\s=\F/_{\!\!\pi_\s}\.\s$, \ $\G_\t=\F/_{\!\!\pi_\t}\.\t$,
and $\G_{\s\t}=\F/_{\!\!\pi_{\s\!\t}}\.\s\t$ with the corresponding
reduction functors $\gamma_\s\:\F\rarrow\G_\s$, \,$\gamma_\t\:
\F\rarrow\G_\t$, and $\gamma_{\s\t}\:\F\rarrow\G_{\s\t}$ and
exact-conservative faithful functors $\eps_\s\:\G_\s\rarrow\E_\s$,
\,$\eps_\t\:\G_\t\rarrow\E_\t$, and
$\eps_{\s\t}\:\G_{\s\t}\rarrow\E_{\s\t}$.
 Clearly, there are induced exact autoequivalences $(1)$
and~$\{1\}$ on all the three categories $\G_\s$, $\G_\t$, and
$\G_{\s\t}$, and all the functors $\gamma$ and~$\eps$ commute
with both of the twists.

 Denote by $\widetilde\H_{\s\t}$ and $\H_{\s\t}$ the intermediate
categories employed in the construction of the category
$\G_{\s\t}=\F/_{\!\!\pi_{\s\!\t}}\.\s\t$ in
Subsections~\ref{matrix-factorizations}\+-\ref{exact-structure}.
 Then the twist functors $(1)$ and~$\{1\}$ act naturally
on the categories $\widetilde\H_{\s\t}$ and~$\H_{\s\t}$; and
the  natural transformation $\s\:\Id_\F\rarrow(1)$ induces
bigraded center elements $\s\:\Id\rarrow(1)$ in
the categories~$\widetilde\H_{\s\t}$, \,$\H_{\s\t}$, and~$\G_{\s\t}$.
 Composing the natural transformation $\s\:\Id_{\G_{\s\!\t}}\rarrow(1)$
with the functor~$\gamma_{\s\t}$, or equivalently,
the functor~$\gamma_{\s\t}$ with the natural transformation
$\s\:\Id_\F\rarrow(1)$ produces a morphism
$\s\:\gamma_{\s\t}\rarrow\gamma_{\s\t}(1)$ of functors
$\F\rarrow\G_{\s\t}$ commuting with both the twists $(1)$ and~$\{1\}$
in the sense of Subsection~\ref{conventions}.

\begin{lem}  \label{second-bockstein-satisfied}
 The conditions~(i\+-IV) of Subsection~\ref{bockstein-generalization}
are satisfied for the exact categories $\F$, \ $\F_\t=\G_\t$, \
$\F_\s=\G_\s$, and $\F_{\s\t}=\G_{\s\t}$ and the exact functors
$\gamma_\t\:\F\rarrow\G_\t$, \ $\gamma_\s\:\F\rarrow\G_\s$, and
$\gamma_{\s\t}\:\F\rarrow\G_{\s\t}$
together with the natural transformation $\s\:\gamma_{\s\t}\rarrow
\gamma_{\s\t}(1)$ commuting with the twist functors~$(1)$
on the categories $\F$, $\G_\t$, $\G_\s$, and\/~$\G_{\s\t}$.
\end{lem}

\begin{proof}
 The conditions~(i\+ii) hold for the functors~$\gamma_\t$, $\gamma_\s$,
and~$\gamma_{\s\t}$ by Lemma~\ref{gamma-properties}(a,c).
 To prove the condition~(III), recall that a morphism $f\:X\rarrow Y$
in the category $\F$ is annihilated by the functor~$\gamma_\t$ if
and only if it is divisible by the natural transformation
$\t\:\Id_\F\rarrow\{1\}$.
 Similarly, the morphism $\gamma_{\s\t}(f)$ is annihilated by
the natural transformation $\s\:\gamma_{\s\t}\rarrow\gamma_{\s\t}(1)$
if and only if the morphism $\s f=\s_Y\!\.f\:X\rarrow Y(1)$ in
the category $\F$ is annihilated by the functor~$\gamma_{\s\t}$,
that is, if and only if the morphism $\s f$ is divisible by
the natural transformation $\s\t\:\Id_\F\rarrow(1)\{1\}$.
 In the latter case, let $g\:X\rarrow Y\{-1\}$ be a morphism in~$\F$
such that $\s_Y\!\.f = \s_Y\t_{Y\{-1\}}g$.
 The morphism~$\s_Y$ being monic in the category $\F$,
one has $f=\t_{Y\{-1\}}g$.

 A morphism $f\:X\rarrow Y$ in the category $\F$ is annihilated by
the functor~$\gamma_\s$ if and only if it is divisible by
the natural transformation $\s\:\Id_\F\rarrow(1)$, so the ``only if''
assertion in the condition~(IV) follows immediately (and even
in a stronger form).
 To prove the ``if'', suppose there is an admissible epimorphism
$X'\rarrow X$ and a morphism $X'\rarrow Y(-1)$ in the category~$\F$
making the square diagram $\gamma_{\s\t}(X')\rarrow\gamma_{\s\t}(X)
\rarrow\gamma_{\s\t}(Y)$, \ $\gamma_{\s\t}(X')\rarrow
\gamma_{\s\t}(Y)(-1)\rarrow\gamma_{\s\t}(Y)$ commutative in
the category~$\G_{\s\t}$.
 This diagram being the image of the diagram
$X'\rarrow X\rarrow Y$, \ $X'\rarrow Y(-1)\rarrow Y$ in
the category $\F$ under the functor~$\gamma_{\s\t}$, we conclude
that the diagram in the category $\F$ commutes up to a morphism
divisible by~$\s\t$.
 Hence one can make the square diagram commutative in the category $\F$
by adding a summand divisible by~$\t$ to the morphism $X'\rarrow Y(-1)$.
 So the composition $X'\rarrow X\rarrow Y$ is divisible by~$\s$
in the category $\F$, and consequently annihilated by
the functor~$\gamma_\s$.
 The morphism $X'\rarrow X$ being an admissible epimorphism,
it follows that $\gamma_\s(f)=0$.
\end{proof}

 Therefore, the construction of
Section~\ref{bockstein-sequence-section} applies and we
obtain a natural Bockstein long exact {sequence \hfuzz=1.6pt
\begin{alignat*}{3}
 0&\rarrow\Hom_{\G_\t}(\gamma_\t(X),\gamma_\t(Y)(-1))&&\rarrow
 \Hom_{\G_{\s\!\.\t}}(\gamma_{\s\t}(X),\gamma_{\s\t}(Y))&&\rarrow
 \Hom_{\G_\s}(\gamma_\s(X),\gamma_\s(Y)) \\
 &\rarrow \Ext^1_{\G_\t}(\gamma_\t(X),\gamma_\t(Y)(-1))&&\rarrow
 \Ext^1_{\G_{\s\!\.\t}}(\gamma_{\s\t}(X),\gamma_{\s\t}(Y))&&\rarrow
 \Ext^1_{\G_{\s}}(\gamma_\s(X),\gamma_\s(Y)) \\
 &\rarrow\Ext^2_{\G_\t}(\gamma_\t(X),\gamma_\t(Y)(-1))&&\rarrow
 \Ext^2_{\G_{\s\!\.\t}}(\gamma_{\s\t}(X),\gamma_{\s\t}(Y))&&\rarrow\dotsb
\end{alignat*}
for any} two objects $X$ and $Y$ in the category~$\F$.
 The differentials in this long exact sequence have
the properties~(a\+c) of Subsection~\ref{bockstein-generalization}.

\subsection{The third Bockstein sequence}  \label{third-bockstein}
 The above construction of a ``finite-finite-finite'' Bockstein
long exact sequence for three reductions of a given exact
category $\F$ still leaves somewhat more to be desired.
 One would like to have such a sequence defined and exact for any
two objects of the exact category $\G_{\s\t}$, and functorial with
respect to all the morphisms in the category~$\G_{\s\t}$.

 Keeping the setting and assumptions of
Subsection~\ref{second-bockstein} in place, suppose additionally that
we are given two exact functors $\ups_\s\:\E_{\s\t}\rarrow\E_\s$ and
$\ups_\t\:\E_{\s\t}\rarrow\E_\t$, commuting with the twists $(1)$
and~$\{1\}$ and forming commutative triangle diagrams with
the background functors~$\pi_\s$, $\pi_\t$, and~$\pi_{\s\t}$.
 In this situation, one would expect existence of exact functors
$\eta_\s\:\G_{\s\t}\rarrow\G_\s$ and $\eta_\t\:\G_{\s\t}\rarrow\G_\t$
forming a commutative diagram will all the functors above.
 However, there does \emph{not} seem to be a natural way to construct
a matrix factorization of the natural transformation $\s$ or~$\t$
on $\F$ starting from a matrix factorization of
the central element~$\s\t$.

 Therefore, let us simply \emph{assume} that we are given exact
functors $\eta_\s\:\G_{\s\t}\rarrow\G_\s$ and $\eta_\t\:\G_{\s\t}
\rarrow\G_\t$, commuting with the twist functors
$(1)$ and~$\{1\}$ and making a commutative diagram of seven
categories and ten functors with the above functors~$\gamma$,
$\eps$, and~$\ups$.
 Assume further that there is a bigraded center element
$\s\:\Id_{\E_{\s\!\t}}\rarrow(1)$ in the category $\E_{\s\t}$
agreeing with the natural transformation~$\s$ on~$\F$, and that
a morphism in~$\E_{\s\t}$ is annihilated by the functor~$\ups_\t$
if and only if it is annihilated by~$\s$.

\begin{lem}
 The conditions~(i\+iv) of Subsection~\ref{bockstein-posing} are
satisfied for the exact functors $\eta_\s\:\G_{\s\t}\rarrow\G_\s$
and $\eta_\t\:\G_{\s\t}\rarrow\G_\t$ together with
the bigraded center element $\s\:\Id_{\G_{\s\!\t}}\rarrow(1)$
in the category\/~$\G_{\s\t}$.
\end{lem}

\begin{proof}
 The functor~$\gamma_{\s\t}$ satisfies the condition~($*$)
by Lemma~\ref{gamma-properties}(b), and the functors $\gamma_\s$
and~$\gamma_\t$ satisfy the conditions (i\+ii) and~($*$\+-$**$)
of Subsection~\ref{exact-surjectivity} by
Lemma~\ref{gamma-properties}(a\+d).
 The latter two functors also reflect admissible epimorphisms and
admissible monomorphisms, as explained in the end of
Subsection~\ref{exact-structure}. 
 According to Lemma~\ref{exact-surjectivity-compositions-lemma}(b,d),
it follows that the functors $\eta_\s$ and~$\eta_\t$ satisfy
the conditions~(i\+ii).

 To check the condition~(iii) of Subsection~\ref{bockstein-posing},
consider a morphism $g\:X\rarrow Y$ in the category~$\G_{\s\t}$.
 If it is annihilated by the functor~$\eta_\t$, then
$0=\eps_\t\eta_\t(g)=\ups_\t\eps_{\s\t}(g)$ implies, according to
our condition on the functor~$\ups_\t$, the equations
$0=\s_{\eps_{\s\!\t}(Y)}\eps_{\s\t}(g)=\eps_{\s\t}(\s_Yg)$
in $\E_{\s\t}$ and $\s_Yg=0$ in $\G_{\s\t}$, since
the functor~$\eps_{\s\t}$ is faithful. 
 Conversely, the equation $\s_Yg=0$ in $\G_{\s\t}$ implies
$\s_{\eps_{\s\!\t}(Y)}\eps_{\s\t}(g)=0$ in $\E_{\s\t}$, hence
$0=\ups_\t\eps_{\s\t}(g)=\eps_\t\eta_\t(g)$ and $\eta_\t(g)=0$,
since the functor $\eps_\t$ is faithful, too.

 To prove the remaining condition~(iv), represent a morphism~$g$
in the category $\G_{\s\t}$ by a morphism of diagrams
$(K,L)\rarrow(P,Q)$ in the category $\H=\H_{\s\t}$ (in the sense
of Remark~\ref{morphisms-up-to-isomorphism}).
 Then the composition $(K,K)\rarrow(K,L)\rarrow(P,Q)\rarrow(Q,Q)$
comes from a morphism $f\:K\rarrow Q$ in the category $\F$
via the functor~$\gamma_{\s\t}$.
 Hence the equation $\eta_\s(g)=0$ impies the equation
$\gamma_\s(f)=0$ in the category~$\G_\s$.
 According to the condition~(vii) of
Subsection~\ref{reduction-posing}, it follows that the morphism~$f$
is divisible by the natural transformation~$\s$ in the category~$\F$.

 We have shown that for any morphism $g\:X\rarrow Y$ in
the category $\G_{\s\t}$ annihilated by the functor~$\eta_\s\:\G_{\s\t}
\rarrow\G_\s$ there exist an admissible epimorphism $X'\rarrow X$
and an admissible monomophism $Y\rarrow Y'$ in the category $\G_{\s\t}$
such that the composition $X'\rarrow X\rarrow Y\rarrow Y'$
is divisible by the natural transformation~$\s$ in~$\G_{\s\t}$.
 Since the conditions~(ii\+iii) of Subsection~\ref{bockstein-posing}
are proven already, it follows by the way of
Lemma~\ref{bockstein-s-zero}(d) applied to the composition of
morphisms $X'\rarrow X\rarrow Y$ (and the functors $\eta_\t\:\G_{\s\t}
\rarrow\G_\t$, \ $\eta_\s\:\G_{\s\t}\rarrow\G_\s$, and
$\eta_{\s\t} = \Id_{\G_{\s\!\t}}$) that there exists an admissible
epimorphism $X''\rarrow X$ in the category $\G_{\s\t}$ for which
the composition $X''\rarrow X'\rarrow X\rarrow Y$ is divisible by~$\s$.
 This proves the condition~(iv).
\end{proof}

 Consequently, the construction of
Section~\ref{bockstein-sequence-section} is applicable in our
assumptions and we obtain a natural long exact sequence
\begin{alignat*}{3}
 0&\lrarrow\Hom_{\G_\t}(\eta_\t(X),\eta_\t(Y)(-1))&&\lrarrow
 \Hom_{\G_{\s\!\.\t}}(X,Y)&&\lrarrow
 \Hom_{\G_\s}(\eta_\s(X),\eta_\s(Y)) \\
 &\lrarrow \Ext^1_{\G_\t}(\eta_\t(X),\eta_\t(Y)(-1))&&\lrarrow
 \Ext^1_{\G_{\s\!\.\t}}(X,Y)&&\lrarrow
 \Ext^1_{\G_{\s}}(\eta_\s(X),\eta_\s(Y)) \\
 &\lrarrow\Ext^2_{\G_\t}(\eta_\t(X),\eta_\t(Y)(-1))&&\lrarrow
 \Ext^2_{\G_{\s\!\.\t}}(X,Y)&&\lrarrow\dotsb
\end{alignat*}
for any two objects $X$ and $Y$ in the category~$\G_{\s\t}$.
 The differentials in this long exact sequence have
the properties~(a\+c) of Subsection~\ref{bockstein-posing}.

\Section{Reduction of Coefficients in Artin--Tate Motives}

\subsection{The main hypothesis}  \label{main-hypothesis}
 Let $\A$ be an exact category, \,$\E_i$, \,$i\in\Z$, be a sequence of
additive categories, and $\phi_i\:\E_i\rarrow\A$ be additive functors.
 We will view the categories $\E_i$ as exact categories with split
exact category structures.

 Consider the category $\F$ whose objects are the triples $(M,Q,q)$,
where $M=(M,F)$ is a finitely filtered object of the exact
category~$\A$, \ $Q=(Q_i)$ is a finitely supported object of
the Cartesian product of additive categories $\prod_{i\in\Z}\E_i$,
and $q$~is a collection of isomorphisms $q_i\:\gr_F^iM\simeq\phi_i(Q_i)$
of objects in the category~$\A$.
 A morphism $f\:(M',Q',q')\rarrow(M'',Q'',q'')$ in the category $\F$ is
a pair $f=(g,h)$, where $g\:(M',F')\rarrow(M'',F'')$ is a morphism of
filtered objects in $\A$ and $h\:Q'\rarrow Q''$ is a morphism of objects
in $\prod_{i\in\Z}\E_i$ such that the induced morphism $\gr_F^i(g):
\gr_{F'}^iM'\rarrow\gr_{F''}^iM''$ forms a commutative square diagram
with the morphism $h_i\:Q'_i\rarrow Q''_i$ and the isomorphisms
$q'_i$, $q''_i$ for all $i\in\Z$.

 The category $\F$ is endowed with the exact category structure in
which a short sequence with zero composition is exact if the related
short sequence of the objects $Q_i$ is split exact in $\E_i$ for each
$i\in\Z$.
 Then the related short sequence of the filtered objects $(M,F)$ is also
exact, because the class of exact sequences is closed under extensions
in the category of complexes in~$\A$.
 We refer to~\cite[Section~3 and
Examples~A.5(4\+5)]{Partin} for some further details (cf.\ 
Example~\ref{filtered-exact-associated-graded-reduction-example}(b)
above).

 There are natural embedding functors $\E_i\rarrow\F$ identifying
the category $\E_i$ with the full exact subcategory of $\F$ consisting
of all the objects $(M,Q,q)$ such that $Q_j=0$ for all $i\ne j$.
 For the sake of simplicity of the terminology and notation, assume
further that there are equivalences of categories $(1)\:\E_i\rarrow
\E_{i+1}$ and a twist functor (exact autoequivalence) $(1)\:\A\rarrow\A$
forming commutative diagrams with the functors~$\phi_i$.
 Then there is a naturally induced twist functor $(1)\:\F\rarrow\F$
taking an object $(M,Q,q)$ to the object $(M(1),Q(1),q(1))$, where
$F^{i+1}(M(1))=(F^iM)(1)$, \ $Q(1)_{i+1}=Q_i(1)$, and
$q(1)_{i+1}=q_i(1)$.

 Denote by $\psi\:\F\rarrow\A$ the forgetful exact functor taking
an object $(M,Q,q)=(M,F,Q,q)\in\F$ to the object $M\in\A$.
 Let $\J$ be a full subcategory of $\E_0$ such that any object of
$\E_0$ is a finite direct sum of objects from~$\J$.

\begin{prop}  \label{diagonal-quadratic-prop}
\textup{(a)} One has\/ $\Ext^n_\F(X,Y)=0$ for any two objects
$X\in\E_i$ and $Y\in\E_j\.\subset\.\F$ and any $n>j-i$; \par
\textup{(b)} the maps\/ $\Ext^n_\F(X,Y)\rarrow\Ext^n_\A
(\phi_i(X),\phi_j(X))$ are isomorphisms for all\/ $i<j$, \,$n=0$
or~$1$ and monomorphisms for all\/ $i,j\in\Z$, \,$n=2$; \par
\textup{(c)} the big graded ring of diagonal cohomology\/
$\Ext^n_\F(X,Y(n))_{Y,X\in\J;\,n\ge0}$ is quadratic
(see~\cite[Subsections~A.1 and~6.1]{Partin} for the definitions).
\end{prop}

\begin{proof}
 The cases $n=0$ and~$1$ in part~(a) are obvious.
 Applying the result of~\cite[Theorem~6.1]{Partin} to the derived
category $\D=\D^b(\F)$ and its full subcategories
$\E_i\subset\F\subset\D$, one obtains the rest of part~(a)
together with part~(c).
 Applying the result of~\cite[Theorem~3.1(2)]{Partin} to
the sequence of exact functors $\phi_i\:\E_i\rarrow\A$ provides
part~(b).
\end{proof}

 We will say that the exact category $\A$ together with the additive
categories $\E_i$ and additive functors $\phi_i\:\E_i\rarrow\A$
\emph{satisfy the main hypothesis} if the maps
$$
 \psi^n\:\Ext^n_\F(X,Y)\lrarrow\Ext^n_\A(\phi_i(X),\phi_j(Y))
$$
are isomorphisms for all the objects $X\in\E_i$ and $Y\in\E_j\.
\subset\.\F$ and all integers $n\le j-i$.
 In particular, this condition for $n=0$, \,$i=j$ means that
the functors~$\phi_i$ are fully faithful.
 With this fact in mind, we will often consider the categories $\E_i$
as full additive subcategories in the exact category $\A$ and
the functors $\E_i\rarrow\A$ as identity embeddings, dropping
the notation~$\phi_i$.
 Abusing terminology, we will simply speak of the exact category $\F$
or, sometimes, the exact functor $\psi\:\F\rarrow\A$ as satisfying
(or not satisfying) the main hypothesis.

 It follows from Proposition~\ref{diagonal-quadratic-prop}(c) that
the main hypothesis implies quadraticity of the big graded ring
$\Ext_\A^n(X,Y(n))_{Y,X\in\J;\,n\ge0}$.
 When the twist $(1)\:\A\rarrow\A$ is the identity functor, the main
result of the paper~\cite{Partin} allows one to say more.

\begin{thm}  \label{koszulity-theorem}
 Assume that all the functors $\phi_i\:\E_i\rarrow\A$ are fully
faithful embeddings of one and the same full additive subcategory
$\E_0\subset\A$.
 Then the main hypothesis holds if and only if the big graded ring\/
$\Ext_\A^n(X,Y(n))_{Y,X\in\J;\,n\ge0}$ is Koszul
(see~\cite[Section~7]{Partin} for the definition and discussion).
\end{thm}

\begin{proof}
 This is essentially the assertion of~\cite[Theorem~9.1]{Partin}.
\end{proof}

 In the rest of this section we discuss a specific class of examples
of exact categories $\A$ with additive functors $\phi_i\:\E_i\rarrow
\A$ for which we would like to be able to show that an appropriate
Koszulity assumption guarantees validity of the main hypothesis,
even though the twist functor $(1)\:\A\rarrow\A$ is not isomorphic
to the identity, but only becomes so after the passage to
the reduction of the exact category $\A$ by a certain element of
its center.

\subsection{Reduction of representation categories}
\label{reduction-rep-categories}
 Let $k$~be a complete Noetherian commutative local ring with
the maximal ideal $l\subset k$.
 Pick an injective hull $I$ of the $k$\+module $k/l$ in the abelian
category of $k$\+modules.
 A $k$\+module $M$ is said to be \emph{discrete} if for every element
$m\in M$ there exists an integer $N\ge1$ such that $l^Nm=0$ in~$M$.
 In particular, the $k$\+module $I$ is discrete.
 A discrete $k$\+module $M$ is said to be \emph{of finite rank} if
its submodule of elements annihilated by~$l$ is a finite-dimensional
$k$\+vector space.

 The category of discrete $k$\+modules is a locally Noetherian
Grothendieck abelian category with a single isomorphism class of
indecomposable injective modules formed by the $k$\+module~$I$;
every injective object in the category of discrete $k$\+modules
is a direct sum of copies of the module~$I$.
 We will denote by $\A_k^{\{e\}\,+}$ the additive category of injective
discrete $k$\+modules and by $\A_k^{\{e\}}\subset\A_k^{\{e\}\,+}$ its
full subcategory of injective discrete $k$\+modules of finite rank.
 The subcategory $\A_k^{\{e\}}$ consists of all the objects in
$\A_k^{\{e\}\,+}$ isomorphic to finite direct sums of the object~$I$.

 Furthermore, let $G$ be a profinite group.
 Let us denote by $\A_k^G{}^+$ the category of injective discrete
$k$\+modules endowed with a discrete action of the group~$G$,
and by $\A_k^G$ the full subcategory of $\A_k^G{}^+$ formed
by injective discrete $k$\+modules of finite rank endowed
with a discrete $G$\+action.
 The categories $\A_k^G$ and $\A_k^G{}^+$ are endowed with exact
category structures in which a short sequence is exact if it is
(split) exact as a short sequence of discrete $k$\+modules.

 Let $s\in k$ be a noninvertible, nonzero-dividing element and
$k/s=k/(s)$ be the quotient ring by the principal ideal generated
by~$s$.
 Denote by $\pi\:\A_k^G{}^+\rarrow\A_{k/s}^{\{e\}\,+}$
the exact functor assigning to a discrete $G$\+module
$M\in\A_k^G{}^+$ over~$k$ the $k/s$\+module $\pi(M)={}_sM\subset M$
of $s$\+torsion elements in~$M$.
 So the functor $\pi$ is the diagonal composition in the commutative
square diagram of the reduction and forgetful functors
$\A_k^G{}^+\rarrow\A_{k/s}^{G\,+}\rarrow\A_{k/s}^{\{e\}\,+}$ and
$\A_k^G{}^+\rarrow\A_k^{\{e\}\,+}\rarrow\A_{k/s}^{\{e\}\,+}$.

 An alternative description of the category $\A_k^G{}^+$ and
the functor~$\pi$ is provided by the theory of $k$\+contramodule
coalgebras and comodules developed in
the paper~\cite[Sections~1 and~3]{Pweak}.
 The ring $k$ can be considered as (a very particular case of)
a pro-Artinian topological local ring~\cite[Section~1 and
Appendix~B]{Pweak}.
 Continuous functions $G\rarrow k$ in the $l$\+adic topology of $k$
form a $k$\+free $k$\+contramodule coalgebra $k(G)$, and the category
$\A_k^G{}^+$ is that of $k$\+cofree $k$\+comodule $k(G)$\+comodules, or,
equivalently, $k$\+free $k$\+contramodule $k(G)$\+comodules.

 The choice of an injective $k$\+module $I$ as above fixes
an equivalence between these two representations of objects of
the category $\A_k^G{}^+$ given by the usual functors
$P=\Psi_I(M)=\Hom_k(I,M)$ and $M=\Phi_I(P)=I\ot_kP$.
 While on the level of $k$- and $G$\+discrete $G$\+modules over~$k$
the functor $\pi$ assigns to a module $M$ its maximal submodule
annihilated by~$s$, on the level of $k$\+free $k$\+contramodule
$k(G)$\+comodules it assigns to a module $P$ its quotient module $P/sP$.

\begin{prop}  \label{unfiltered-reduction-prop}
 Let $\A_k^G{}^+/s$ denote the reduction of the exact category
$\A_k^G{}^+$ by the natural transformation $s\:\Id\rarrow\Id$ taken
on the background of the exact-conservative functor
$\pi\:\A_k^G{}^+\rarrow\A_{k/s}^{\{e\}\,+}$.
 Then the natural exact functor $\A_k^G{}^+/s\rarrow\A_{k/s}^{G\,+}$ is
an equivalence of exact categories.
\end{prop}

\begin{proof}
 The natural exact functor assigns to a diagram $V\rarrow U\rarrow
V\rarrow U$ the image of the map ${}_sU\rarrow{}_sV$ in the category
$\A_{k/s}^{G\,+}$.
 To construct the inverse functor, consider an object
$M\in\A_{k/s}^{G\,+}$.
 Viewed as an object of the abelian category of discrete $k$\+modules
with a discrete $G$\+action, it can be embedded into a $k$- and
$k(G)$\+injective $k$-discrete $k(G)$-comodule $V$.
 The element $s$ being a nonzero-divizor, the injective dimension of $M$
as a discrete $k$\+module is equal to~$1$ (if $M\ne0$), hence
the quotient module $U=V/M$ is also $k$\+injective.
 The multiplication map $s\:V\rarrow V$ annihilates $M$, so it
factorizes naturally through the surjection $V\rarrow U$.
 The composition $U\rarrow V\rarrow U$ is also equal to
the multiplication with~$s$, as one can see by composing it with
the same surjection $V\rarrow U$.
 We have constructed the desired matrix factorization in $\A_k^G{}^+$;
since it consists of surjective maps, the passage to the submodules
of elements annihilated by~$s$ transforms it into an exact sequence
of $k/s$\+injective objects in the abelian category of discrete
$k/s$\+modules with a discrete action of~$G$ \cite[Lemma~1.5]{PolVain}.
 The $k/s$\+modules of cocycles in this exact sequence are injective,
since such is the one of them that is isomorphic to $M$ by construction.
\end{proof}

 Let $H\subset G$ be a fixed closed normal subgroup in the profinite
group~$G$.
 Denote by $\E_{k,\,0}^{G/H}$ the category of $k$\+injective
permutational $G/H$\+modules of finite rank, i.~e., the full subcategory
of $\A_k^G$ consisting of finite direct sums of modules induced from
trivial representations of open subgroups in $G$ containing $H$ in
injective discrete $k$\+modules of finite rank.
 Similarly, let $\E_{k,\,0}^{G/H\,+}$ denote the category of arbitrary
$k$\+injective permutational discrete $G/H$\+modules, i.~e., the full
subcategory in $\A_k^G{}^+$ whose objects are the (infinite) direct sums
of objects from $\E_{k,\,0}^{G/H}$.
 The categories $\E_{k,\,0}^{G/H}$ and $\E_{k,\,0}^{G/H\,+}$ are endowed
with split exact category structures.

 Consider the reduction functor $\pi\:\E_{k,\,0}^{G/H}\rarrow
\E_{k/s,\,0}^{G/H}$ taking a $k$\+injective permutational $G/H$\+module
$M$ to the $k/s$\+injective permutational $G/H$\+module ${}_sM$,
and denote by $f\mpsto\bar f$ its action on morphisms
in the two categories; and similarly for the reduction functor
$\pi\:\E_{k,\,0}^{G/H\,+}\rarrow \E_{k/s,\,0}^{G/H\,+}$ and its action
on morphisms.

\begin{lem} \label{reduction-conservative-lem}
\textup{(a)}
 Let $M$ and $N$ be $k$\+injective permutational $G/H$\+modules of
finite rank over~$k$, i.~e., objects of the category $\E_{k,\,0}^{G/H}$.
 Then every morphism $\bar f\:{}_sM\rarrow{}_sN$ in the category
$\E_{k/s,\,0}^{G/H}$ can be lifted to a morphism $f\:M\rarrow N$.
 Moreover, a morphism $f$ is an admissible monomorphism or admissible
epimorphism if and only if the morphism $\bar f$ is, and a sequence of
permutational $G/H$\+modules with zero composition is exact in
$\E_{k,\,0}^{G/H}$ if and only if its reduction modulo~$s$ is. \par
\textup{(b)}
 Let $M$ and $N$ be arbitrary $k$\+injective discrete permutational
$G/H$\+modules over~$k$, i.~e., objects of the category
$\E_{k,\,0}^{G/H\,+}$.
 Then every morphism $\bar f\:{}_sM\rarrow{}_sN$ in the category
$\E_{k/s,\,0}^{G/H\,+}$ can be lifted to a morphism
$f\:M\rarrow N$.
 Moreover, a morphism $f$ is an admissible monomorphism or admissible
epimorphism if and only if the morphism $\bar f$ is, and a sequence of
permutational $G/H$\+modules with zero composition is exact in
$\E_{k,\,0}^{G/H\,+}$ if and only if its reduction modulo~$s$ is. \par
\end{lem}

\begin{proof}
 Part~(a): the $G$\+module $\Hom_k(M,N)$ is a permutational $k$\+free
$G/H$\+module of finite rank, while the $G$\+module
$\Hom_{k/s}({}_sM,{}_sN)$ is its reduction modulo~$s$.
 Clearly, the reduction map $\Hom_k(M,N)\rarrow\Hom_{k/s}({}_sM,{}_sN)$
identifies the submodule of $G$\+invariants in $\Hom_{k/s}({}_sM,{}_sN)$
with the reduction of the submodule of $G$\+invariants in
$\Hom_k(M,N)$, proving the first assertion.
 Furthermore, by a version of Nakayama's lemma a morphism in the category
$\E_{k,\,0}^{G/H}$ is an isomorphism if and only if so is its image in
the category $\E_{k/s,\,0}^{G/H}$.
 The remaining assertions now follow from the fact any $k/s$\+injective
permutational $G/H$\+module of finite rank can be lifted to a similar
$k$\+injective permutational $G/H$\+module.

 Part~(b): in this case, the $G$\+module $\Hom_k(M,N)$ is an infinite
product of infinite direct sums (in the $k$\+contramodule category) of
permutational $k$\+free $G/H$\+modules of finite rank, while
the $G$\+module $\Hom_{k/s}({}_sM,{}_sN)$ is the reduction of this
infinite product of infinite direct sums modulo~$s$
\cite[Subsections~1.2\+-1.3 and~1.5]{Pweak}.
 Otherwise, the argument works in the same way.
 One only has to check that the passage to $G$\+invariants commutes with
contramodule infinite direct sums of permutational representations and
the reductions of such contramodule representations by the ideals in~$k$.
 Indeed, the direct sum of a family of free $k$\+contramodules $P_i$
is the set of all families of elements $p_i\in P_i$ such that for any
$n>0$ all but a finite number of~$p_i$ belong to $l^nP_i$.
 The key observation is that for a permutational $k$\+free
$G$\+module $P$ (of finite rank over~$k$), the $k$\+submodules
$(l^nP)^G$ and $l^n(P^G)$ coincide in~$P$.
\end{proof}

\begin{cor}  \label{permutational-reduction-cor}
\textup{(a)} The reduction functor $\pi\:\E_{k,\,0}^{G/H}\rarrow
\E_{k/s,\,0}^{G/H}$ satisfies the conditions (v-viii) of
Subsection~\ref{reduction-posing}.
 One obtains the category $\E_{k/s,\,0}^{G/H}$ as the outcome of
the reduction procedure. \par
\textup{(b)} The reduction functor $\pi\:\E_{k,\,0}^{G/H\,+}\rarrow
\E_{k/s,\,0}^{G/H\,+}$ satisfies the conditions (v-viii) of
Subsection~\ref{reduction-posing}.
 One obtains the category $\E_{k/s,\,0}^{G/H\,+}$ as the outcome of
the reduction procedure.
\end{cor}

\begin{proof}
 By Lemma~\ref{reduction-conservative-lem}, the reduction functors~$\pi$
are exact-conservative.
 The remaining verifications are easy.
\end{proof}

\subsection{Large and small filtered representation categories}
\label{large-and-small-categories}
 Let $\bar k$ be a discrete Artinian local ring and $G$ be
a profinite group.
 As above, we denote by $\A_{\bar k}^G{}^+$ the exact category of
$\bar k$\+injective discrete $G$\+modules over~$\bar k$, and
by $\A_{\bar k}^G$ its full exact subcategory formed by injective
$\bar k$\+modules of finite rank endowed with a discrete
action of~$G$.

\begin{lem} \label{large-and-small-unfiltered-lem}
\textup{(a)} The embedding functor $\A_{\bar k}^G\rarrow
\A_{\bar k}^G{}^+$ satisfies the condition (i$\.'$) of
Subsection~\ref{exact-surjectivity}; in particular, it induces
isomorphisms on all the Ext groups. \par
\textup{(b)} The functor Ext in the exact category $\A_{\bar k}^G{}^+$
transforms infinite direct sums in its first argument into infinite
products.
 When its first argument belongs to $\A_{\bar k}^G$, the functor Ext
in the exact category $\A_{\bar k}^G{}^+$ transforms infinite direct
sums in the second argument into infinite direct sums.
\end{lem}

\begin{proof}
 Choosing an injective $\bar k$\+module $I$ as in
Subsection~\ref{reduction-rep-categories} allows to identify
$\A_{\bar k}^G{}^+$ with the category of $\bar k$\+projective
discrete $G$\+modules over~$\bar k$ and $\A_{\bar k}^G$ with its
exact subcategory of projective $\bar k$\+modules of finite rank
endowed with a discrete $G$\+action.
 In this setting, the assertions of Lemma become true for any
discrete commutative ring~$\bar k$.
 To prove part~(a), notice that any element in a discrete
$G$\+module over~$\bar k$ generates a submodule isomorphic to
a quotient module of a (permutational) $\bar k$\+projective discrete
$G$\+module of finite rank over~$\bar k$.

 The first assertion in part~(b) holds for any exact category with
enough injective objects, and the second one follows from
the similar property of the groups Hom and the first assertion of
part~(a) (alternatively, it can be deduced from the fact that
infinite direct sums are exact and preserve injective objects in
$\A_{\bar k}^G{}^+$).
 Another alternative way of arguing is to compare the groups Ext in
the exact categories $\A_{\bar k}^G{}^+$ and $\A_{\bar k}^G$ with
the similar groups in the abelian categories of arbitrary and
$\bar k$\+finitely generated discrete $G$\+modules over~$\bar k$
(assuming $\bar k$ Noetherian).
\end{proof}

 Let $H\subset G$ be a closed normal subgroup and $\bar c\:G
\rarrow \bar k^*$ be a discrete multiplicative character
of the profinite group~$G$.

 As in Subsection~\ref{reduction-rep-categories}, we denote by
$\E_{\bar k,\,0}^{G/H}\subset\A_{\bar k}^G$ the full subcategory of
$\bar k$\+injective permutational $G/H$\+modules of finite rank
and by $\E_{\bar k,\,0}^{G/H\,+}\subset\A_{\bar k}^G{}^+$ the full
subcategory of $\bar k$\+injective permutational discrete
$G/H$\+modules of possibly infinite rank.
 For any $i\in\Z$, denote by $\E_{\bar k,\,i}^{G/H}\subset\A_{\bar k}^G$
and $\E_{\bar k,\,i}^{G/H\,+}\subset\A_{\bar k}^G{}^+$ the full
subcategories of objects obtained by twisting objects of the full
subcategories $\E_{\bar k,\,0}^{G/H}$ and $\E_{\bar k,\,0}^{G/H\,+}$,
respectively, with the character~$\bar c^i$
(the $i$\+th power of the character~$\bar c$).

 Let $\E_{\bar k}^G$ denote the category of finitely supported graded
objects $Q=(Q_i)$ in the category $\A_{\bar k}^G$ in which
the object $Q_i$ belongs to the subcategory $\E_{\bar k,\,i}^{G/H}$.
 Similarly, $\E_{\bar k}^G{}^+$ denotes the category of finitely
supported graded objects $Q=(Q_i)$ in the category $\A_{\bar k}^G{}^+$
in which the object $Q_i$ belongs to the subcategory
$\E_{\bar k,\,i}^{G/H\,+}$.
 The categories $\E_{\bar k}^G$ and $\E_{\bar k}^G{}^+$ are endowed with
split exact category structures.

 As in Subsection~\ref{main-hypothesis}, we consider the exact
category $\F_{\bar k}^G$ of finitely filtered objects of the category
$\A_{\bar k}^G$ with the associated quotient modules belonging
to $\E_{\bar k}^G$.
 The twist functor~$(1)$ is the twist with the character~$\bar c$.
 Similarly, $\F_{\bar k}^G{}^+$ is the exact category of finitely
filtered objects of the category $\A_{\bar k}^G{}^+$ with
the associated quotient modules belonging to $\E_{\bar k}^G{}^+$.
 Finally, let $\F_{\bar k,\,[j',j'']}^{G\,+}\subset\F_{\bar k}^G{}^+$
denote the full exact subcategory formed by all the triples $(M,Q,q)$
where the graded object $Q=(Q_i)$ is supported in the segment
of degrees $j'\le i\le j''$.

\begin{prop}  \label{large-and-small-filtered-prop}
\textup{(a)} Any object of the category $\F_{\bar k}^G{}^+$ is
the inductive limit of a directed diagram of objects from
$\F_{\bar k}^G$ and admissible monomorphisms between them.
 In particular, the embedding functor $\F_{\bar k}^G\rarrow
\F_{\bar k}^G{}^+$ satisfies the condition (i$\.'$) of
Subsection~\ref{exact-surjectivity}, so it induces
isomorphisms on all the Ext groups. \par
\textup{(b)} The functor Ext in the exact category
$\F_{\bar k,\,[j',j'']}^{G\,+}$ transforms infinite direct sums in
its first argument into infinite products.
 When its first argument belongs to $\F_{\bar k}^G$, the functor Ext
in the exact category $\F_{\bar k,\,[j',j'']}^{G\,+}$ transforms infinite
direct sums in the second argument into infinite direct sums.
\end{prop}

\begin{proof}
 Part~(a): let $(M,Q,q)$ be an object of the category
$\F_{\bar k,\,[j',j'']}^{G\,+}\subset\F_{\bar k}^G{}^+$.
 Let us fix decompositions of the objects $Q_i$ into direct sums
of twists of $\bar k$\+injective permutational representations
of $G/H$ of finite rank over~$\bar k$, and show that for any
finitely generated $G$\+submodule $K\subset M$ there exists
a finitely generated $G$\+submodule $K\subset N\subset M$ such
that the associated graded modules $\gr_F^iN\subset Q_i$ are 
direct sums of finite subcollections of direct summands in our
fixed direct sum decompositions.
 Proceeding by induction in $j''-j'$, consider the object
$M_1=M/F^{j''}M\in\F_{\bar k,\,[j',j''-1]}^{G\,+}$ and the submodule
$K_1=K/F^{j''}M\cap K\subset M_1$; let $K_1\subset N_1\subset M_1$
be the related intermediate submodule satisfying the above condition.
 We have obtained an extension class in
$\Ext^1_{\A_{\bar k}^{G\.+}}(N_1,Q_{j''})$, and now it suffices to show
that it is induced from an extension class with a direct sum of
a finite subcollection of the direct summands in the second argument.
 This is a particular case of the second assertion of
Lemma~\ref{large-and-small-unfiltered-lem}(b).

 To prove the first assertion of part~(b), it suffices to show that
there are enough injective objects in the exact category
$\F_{\bar k,\,[j',j'']}^{G\,+}$.
 To construct such injective objects, we proceed again by induction
in the length of the filtration $j''-j'$.
 Suppose $J_1$ is an injective object in
$\F_{\bar k,\,[j'+1,j'']}^{G\,+}$.
 For every $\bar c^{j'}$\+twisted $\bar k$\+injective permutational
representation $R$ of $G/H$ of finite rank over~$\bar k$ assign
a nonempty multiplicity set to every element of the extension
group $\Ext^1_{\A_{\bar k}^{G\.+}}(R,\psi(J_1))$, where
$\psi\:\F_{\bar k}^G{}^+\rarrow\A_{\bar k}^G{}^+$ is the forgetful
functor, and denote by $Q_R$ the direct sum of copies of $R$ over
the disjoint union of these sets.
 By the first assertion of Lemma~\ref{large-and-small-unfiltered-lem}(b),
there is a naturally induced element in
$\Ext^1_{\A_{\bar k}^{G\.+}}(\bigoplus_R Q_R,\psi(J_1))$, where
the direct sum is taken over all twisted permutational representations
$R$ of finite rank (or just twisted representations induced from
the standard trivial modules $I$ over open subgroups in $G$
containing~$G/H$).
 This extension class corresponds to an injective object
$J\in\F_{\bar k,\,[j',j'']}^{G\,+}$, and choosing the multiplicity sets
big enough one can construct an admissible monomorphism from
any object of $\F_{\bar k,\,[j',j'']}^{G\,+}$ into an injective object
constructed in this way.
 The second assertion of part~(b) follows from the second assertion
of part~(a).
\end{proof}

\begin{rem}
 Applied to $\bar k$\+projective instead of $\bar k$\+injective
$G$\+modules over~$\bar k$ and permutational representations of 
the related kind, the above arguments hold for any discrete
Noetherian ring~$\bar k$.
 However, they fail over a complete Noetherian local ring~$k$
(as in Subsection~\ref{reduction-rep-categories}), because injective
$k$\+modules of finite rank are not finitely generated, while
infinite direct sums in the category of projective
$k$\+contramodules are not preserved by the forgetful functor
to $k$\+modules (or just abelian groups).
\end{rem}

\begin{cor}
 The forgetful functor $\F_{\bar k}^G{}^+\rarrow\A_{\bar k}^G{}^+$
satisfies the main hypothesis of Subsection~\ref{main-hypothesis}
if and only if the functor $\F_{\bar k}^G\rarrow\A_{\bar k}^G{}^+$ does,
and if and only if the functor $\F_{\bar k}^G\rarrow\A_{\bar k}^G$ does.
\end{cor}

\begin{proof}
 Follows from Lemma~\ref{large-and-small-unfiltered-lem}(a\+b) and
Proposition~\ref{large-and-small-filtered-prop}(a\+b).
\end{proof}

 Now let us suppose that $\bar k$ is a field and the character
$\bar c\:G\rarrow \bar k^*$ annihilates the closed subgroup
$H\subset G$.
 Let $\J\subset\E_{\bar k,\,0}^G$ denote the full subcategory of
representations induced from trivial representations~$\bar k$ of
open subgroups in $G$ containing~$H$.

\begin{lem}  \label{root-of-1-koszulity-lemma}
 The exact category $\F_{\bar k}^G$ satisfies the main hypothesis if and
only if the big graded ring\/
$\Ext^n_{\A_{\bar k}^G}(X,Y(n))_{Y,X\in\J,\,n\ge0}$ is Koszul.
 In this case, the big graded ring\/ $\Ext^n_{\F_{\bar k}^G}
(X,Y)_{Y\in\E_{\bar k,\.j}^G,\.X\in\E_{\bar k,\.i}^G;\,i,j\in\Z;\,n\ge0}$
is generated by its\/ $\Ext^0$ and the diagonal ($n=j-i$)
extension classes.
 The latter property also holds for the extension classes between
objects of the full subcategories $\E_{\bar k,\,i}^{G\,+}$ and
$\E_{\bar k,\,j}^{G\,+}$ in the exact category $\F_{\bar k}^{G\,+}$.
\end{lem}

\begin{proof}
 The image of the character~$\bar c$, being a compact subgroup
of a discrete group, is therefore a finite subgroup of
the multiplicative group of the field $\bar k$, whose order is
consequently prime to the characteristic.
 Therefore, the representation $\bar k(1)$ of the group $G$
corresponding to the character~$\bar c$ is a direct summand of
a permutational one.
 Since we have assumed that $\bar c$~annihilates $H$, this is
(can be chosen to be) also a permutational representation of~$G/H$.
 Twisting with the character~$\bar c$ transforms permutational
representations of $G/H$ over~$\bar k$ into direct summands of
permutational representations of~$G/H$.

 So we see that the category $\F_{\bar k}^G$ only differs from
the similar category constructed using the trivial character in
place of~$\bar c$ by adjoining some direct summands and removing
some others (cf.~\cite[Section~5]{Partin}).
 Now the first assertion follows from
Theorem~\ref{koszulity-theorem}, and the second and third ones
are also clear.
\end{proof}

\subsection{Reduction of filtered representation category}
\label{reduction-filtered-rep-category}
 Let $G$ be a profinite group, $H\subset G$ be a closed normal
subgroup, $k$~be a complete Noetherian local ring with the maximal
ideal~$l$, and $c\:G\rarrow k^*$ be a continuous multiplicative
character of the profinite group $G$ in the $l$\+adic topology of
the ring~$k$.

 As in Subsections~\ref{reduction-rep-categories}\+-%
\ref{large-and-small-categories} (see also
Example~\ref{filtered-permutational-coefficient-reduction-example}
in Subsection~\ref{reduction-examples}), we denote by $\A_k^G{}^+$
the category of injective discrete $k$\+modules endowed with
a discrete action of the group $G$, and by $\E_{k,\,i}^{G/H\,+}\subset
\A_k^G{}^+$ its full subcategory of $c^i$\+twisted $k$- and
$G$\+discrete $k$\+injective permutational $G/H$\+modules
(of possibly infinite rank over~$k$).
 Let $\E_k^G{}^+$ denote the category of finitely supported graded
objects $Q=(Q_i)$ in the category $\A_k^G{}^+$ in which the object
$Q_i$ belongs to the subcategory $\E_{k,\,i}^{G/H\,+}$.

 As in Subsection~\ref{main-hypothesis}, we consider the exact
category $\F_k^G{}^+$ of finitely filtered objects $(M,F)$ of
the exact category $\A_k^G{}^+$ with the associated graded modules
$Q=(Q_i)$, $\,q_i\:\gr_F^iM\simeq Q_i$ belonging to $\E_k^G{}^+$.
 The twist functor $(1)\:\A_k^G{}^+\rarrow\A_k^G{}^+$ is the twist
with the character~$c$.
 We denote by $\gr_F\:\F_k^G{}^+\rarrow\E_k^G{}^+$ the exact functor
assigning the graded object $Q$ to a filtered object $(M,F)$ and
by $\psi\:\F_k^G{}^+\rarrow\A_k^G{}^+$ the forgetful functor taking
$(M,F)$ to~$M$.

 Let $s\in k$ be a nonzero-dividing, noninvertible element.
 The quotient ring $k/s$ is endowed with a continuous multiplicative
character $c/s\:G\rarrow(k/s)^*$, so the above exact categories
have their natural versions with coefficients in~$k/s$.

 Denote by $\pi'\:\F_k^G{}^+\rarrow\E_{k/s}^{G\,+}$ the composition
of the functor $\gr_F$ with the reduction functor
$\E_k^G{}^+\rarrow\E_{k/s}^{G\,+}$ taking $(Q_i)$ to $({}_sQ_i)$.
 Furthermore, let $\A_k^{\{e\}\,+}$ denote the additive category of
injective discrete $k$\+modules.
 Denote by $\pi''\:\F_k^G{}^+\rarrow\A_{k/s}^{\{e\}\,+}$
the composition of the forgetful functor~$\psi$ with the forgetful
functor $\A_k^G{}^+\rarrow\A_k^{\{e\}\,+}$ and the reduction functor
$\A_k^{\{e\}\,+}\rarrow\A_{k/s}^{\{e\}\,+}$ taking $M$ to ${}_sM$.

 Both $\E_{k/s}^{G\,+}$ and $\A_{k/s}^{\{e\}\,+}$ are naturally endowed
with split exact category structures.
 Consider the Cartesian product $\E_{k/s}^{G\,+}\times
\A_{k/s}^{\{e\}\,+}$, and let $\pi=(\pi',\pi'')\:\F_k^G{}^+\rarrow
\E_{k/s}^{G\,+}\times\A_{k/s}^{\{e\}\,+}$ denote the related diagonal
exact functor.

\begin{lem}
 The exact functor $\pi=(\pi',\pi'')\:\F_k^G{}^+\rarrow
\E_{k/s}^{G\,+}\times\A_{k/s}^{\{e\}\,+}$ satisfies
the conditions~(v\+vii) from Subsection~\ref{reduction-posing}
on a background exact functor for the construction of reduction
of the exact category $\F=\F_k^G{}^+$ with respect to the natural
transformation $s\:\Id_\F\rarrow\Id_\F$.
 The condition~(viii) is also satisfied.
\end{lem}

\begin{proof}
 The condition~(viii) on the natural transformation~$s$ holds
because $s\in k$ is a nonzero-dividing element.
 The condition~(vi) is obvious from the construction.
 The condition~(v) holds due to the presence of the component~$\pi'$
in the functor~$\pi$ and by Lemma~\ref{reduction-conservative-lem}(b),
while the condition~(vii) is true due to the presence of
the component~$\pi''$ in~$\pi$.
\end{proof}

 Applying the construction of
Section~\ref{matrix-factor-construct-section}, we obtain the reduced
exact category $\F_k^G{}^+/s$.
 There is a natural exact comparison functor $\kap\:\F_k^G{}^+/s
\rarrow\F_{k/s}^{G\,+}$ assigning to a matrix factorization diagram
$V\rarrow U\rarrow V\rarrow U$ for the center element~$s$ in
the exact category $\F_k^G{}^+$ the image of the morphism
${}_sU\rarrow {}_s V$ in the sequence ${}_sV\rarrow{}_sU\rarrow
{}_sV\rarrow{}_sU$ in the exact category $\F_{k/s}^{G\,+}$ (the sequence
being exact, because so is its associated graded sequence by
the filtration~$F$).

 Furthermore, every object of the exact category $\F_k^G{}^+/s$ comes
endowed with a natural finite filtration $F$ given by the natural
filtration on the matrix factorization diagrams.
 The filtrations $F$ on the objects of $\F_k^G{}^+/s$ are preserved
by all morphisms in this category.
 The objects of $\F_k^G{}^+/s$ concentrated in a single filtration
degree~$i$ form a full exact subcategory $\E_{k,\,i}^{G\,+}/s\simeq
\E_{k/s,\,i}^{G\,+}$ (see
Corollary~\ref{permutational-reduction-cor}(b)), so the categories
$\E_{k/s,\,i}^{G\,+}$ are embedded into $\F_k^G{}^+/s$.
 The triangle diagrams of exact functors $\E_{k/s,\,i}^{G\,+}\rarrow
\F_k^G{}^+/s\rarrow\F_{k/s}^{G\,+}$ are commutative.

\begin{prop}  \label{comparison-fully-faithful}
 The comparison functor $\kap\:\F_k^G{}^+/s\rarrow\F_{k/s}^{G\,+}$
is fully faithful.
\end{prop}

\begin{proof}
 Essentially, the claim is that the induced morphisms of the Ext groups
$$
 \kap^n\:\Ext^n_{\F_k^G{}^+/s}(X,Y)\rarrow
 \Ext^n_{\F_{k/s}^{G\.+}}(X,Y)
$$
are isomorphisms for $n=0$ and monomorphisms for $n=1$.
 In view of the 5\+lemma, it suffices to show this for objects
$X\in\E_{k/s,\,i}^{G\,+}$ and $Y\in\E_{k/s,\,j}^{G\,+}$
(cf.~\cite[Lemma~3.2]{Partin}).

 Taking into account Proposition~\ref{unfiltered-reduction-prop},
the forgetful functor $\psi\:\F_k^G{}^+\rarrow\A_k^G{}^+$ induces
a morphism between the long exact sequences
$$
 \Ext^n_{\F_k^G{}^+}(X,Y)\lrarrow\Ext^n_{\F_k^G{}^+}(X,Y) \lrarrow
 \Ext^n_{\F_k^G{}^+/s}({}_sX,{}_sY)\lrarrow\Ext^{n+1}_{\F_k^G{}^+}(X,Y)
$$
and
$$
 \Ext^n_{\A_k^G{}^+}(X,Y)\lrarrow\Ext^n_{\A_k^G{}^+}(X,Y) \lrarrow
 \Ext^n_{\A_{k/s}^{G\.+}}({}_sX,{}_sY)\lrarrow\Ext^{n+1}_{\A_k^G{}^+}(X,Y)
$$
from Subsection~\ref{reduction-posing}
for any two objects $X\in\E_{k,\,i}^{G\,+}$ and $Y\in\E_{k,\,j}^{G\,+}$.
 According to Proposition~\ref{diagonal-quadratic-prop}(a),
one has $\Ext^n_{\F_k^G{}^+}(X,Y)=\nobreak0$ and consequently
$\Ext^n_{\F_k^G{}^+/s}(X,Y)=0$ for $n>j-i$.
 By Proposition~\ref{diagonal-quadratic-prop}(b), the maps
$\psi^n\:\Ext^n_{\F_k^G{}^+}(X,Y)\rarrow\Ext^n_{\A_k^G{}^+}(X,Y)$
are isomorphisms for $i<j$ and $n=\nobreak0$ or~$1$ and
monomorphisms for $n=2$.
 Applying the 5\+lemma, we conclude that the map
$\Ext^n_{\F_k^G{}^+/s}({}_sX,{}_sY)\rarrow
\Ext^n_{\A_{k/s}^{G\.+}}({}_sX,{}_sY)$ is an isomorphism for $n=0$
and a monomorphism for $n=1$ whenever $i<j$.
 It remains to compare these computations with the description
of $\Ext^n_{\F_{k/s}^{G\.+}}({}_sX,{}_sY)$ provided by the same
Proposition, keeping in mind the commutative
triangle of exact functors $\F_k^G{}^+/s\rarrow\F_{k/s}^{G\.+}
\rarrow\A_{k/s}^{G\.+}$.
\end{proof}

\begin{ex}
 The following example shows that the comparison functors~$\kap$
are not equivalences of categories in general.
 Let $l$~be a prime number and $k=\Z_l$ be the ring of $l$\+adic
integers.
 Let $s$ and $t$ be two powers of~$l$.
 Then for any $X\in\E_{k,\,i}^{G\,+}$ and $Y\in\E_{k,\,j}^{G\,+}$
there is the long exact sequence of Subsection~\ref{second-bockstein}
\begin{alignat*}{3}
 0&\lrarrow\Hom_{\F_k^G{}^+/t}({}_tX,{}_tY)&&\lrarrow
 \Hom_{\F_k^G{}^+/st}({}_{st}X,{}_{st}Y)&&\lrarrow
 \Hom_{\F_k^G{}^+/s}({}_sX,{}_sY) \\
 &\lrarrow\Ext^1_{\F_k^G{}^+/t}({}_tX,{}_tY)&&\lrarrow
 \Ext^1_{\F_k^G{}^+/st}({}_{st}X,{}_{st}Y)&&\lrarrow
 \Ext^1_{\F_k^G{}^+/s}({}_sX,{}_sY) \\
 &\lrarrow\Ext^2_{\F_k^G{}^+/t}({}_tX,{}_tY)&&\lrarrow
 \Ext^2_{\F_k^G{}^+/st}({}_{st}X,{}_{st}Y)&&\lrarrow\dotsb
\end{alignat*} 
 Set $X=\Q_l/\Z_l\in\E_{k,\,0}^{G\,+}$ and $Y=(\Q_l/\Z_l)(1)
\in\E_{k,\,1}^{G\,+}$; according to the above, one then has
$\Ext^2_{\F_k^G{}^+/t}({}_tX,{}_tY)=0$, so the map
$\Ext^1_{\F_k^G{}^+/st}({}_{st}X,{}_{st}Y)\rarrow
\Ext^1_{\F_k^G{}^+/s}({}_sX,{}_sY)$ is surjective.
 Assuming the equivalences $\F_k^G{}^+/s\simeq\F_{k/s}^{G\,+}$ and
$\F_k^G{}^+/st\simeq\F_{k/st}^{G\,+}$, this would simply mean
that the natural cohomology map $H^1(G,\Z/st(1))\rarrow
H^1(G,\Z/s(1))$ is surjective for the profinite group $G$ with
the character $c\:G\rarrow\Z_l^*$ (irrespectively of the closed
subgroup~$H$).
 However, this is not always true.
 When $G$ is the absolute Galois group of a field and $c$~is
its cyclotomic character, the surjectivity of such maps follows
from Hilbert's Theorem~90.
\end{ex}

 As it was mentioned in Subsection~\ref{third-bockstein}, there is
apparently no straightforward way to construct reduction functors
$\eta\:\F_k^G{}^+/st\rarrow\F_k^G{}^+/s$ on the matrix factorization
diagrams level.
 However, there \emph{are} natural functors $\rho\:\F_{k/st}^{G\,+}
\rarrow \F_{k/s}^{G\,+}$ assigning to an injective discrete
$k/st$\+module $N$ with a filtration $F$ and a discrete action of
the group $G$ the injective discrete $k/s$\+module ${}_sN$ of
elements annihilated by~$s$ in $N$ with the induced filtration and
the group action.

\begin{prop}
 The functor $\rho\:\F_{k/st}^{G\,+}\rarrow\F_{k/s}^{G\,+}$ takes
the full subcategory $\F_k^G{}^+/st\subset\F_{k/st}^{G\,+}$ into
the full subcategory $\F_k^G{}^+/s\subset\F_{k/s}^{G\,+}$.
\end{prop}

\begin{proof}
 An object $N=(N,F)\in\F_{k/st}^{G\,+}$ belongs to $\F_k^G{}^+/st$
if and only if there exists an object $M=(M,F)\in\F_k^G{}^+$
such that the injective $k/st$\+module $N$ can be embedded into
the injective $k$\+module $M$ in a way strictly compatible with
the filtrations, compatible with the actions of $k$ and $G$,
and such that the induced embeddings
$\gr_F^iN\rarrow{}_{st}\.\gr_F^iM$ are admissible monomorphisms
in $\E_{k/st,\,i}^{G\,+}$.
 (Cf.\ the proof of Proposition~\ref{unfiltered-reduction-prop}.)
 Indeed, the ``only if'' part is obvious from the construction, and
to prove the ``if'', one recalls that by
Lemma~\ref{reduction-conservative-lem}(b) any admissible monomorphism
in $\E_{k/st,\,i}^{G\,+}$ can be lifted to an admissible monomorphism
in $\E_{k,\,i}^{G\,+}$.

 Whenever such an embedding $N\rarrow M$ exists, its composition
${}_sN\rarrow N\rarrow M$ with the identity embedding ${}_sN\rarrow N$
provides an embedding ${}_sN\rarrow M$ showing that the object
$({}_sN,F)$ belongs to $\F_k^G{}^+/s$.
\end{proof}

\begin{thm}
 Suppose that the maps
$$
 \Ext^n_{\A_k^G{}^+}(X,Y)\lrarrow\Ext^n_{\A_{k/s}^{G\.+}}({}_sX,{}_sY)
$$
are surjective for all $X\in\E_{k,\,i}^{G\,+}$ and
$Y\in\E_{k,\,j}^{G\,+}$ with $j-i=n$ and $n\ge1$.
 Assume that the exact category $\F_k^G{}^+$ satisfies the main
hypothesis of Subsection~\ref{main-hypothesis}.
 Then the comparison functor $\kap\:\F_k^G{}^+/s\rarrow
\F_{k/s}^{G\,+}$ is an equivalence of exact categories, and
the main hypothesis for the exact category $\F_{k/s}^{G\,+}$ is
also satisfied.
\end{thm}

\begin{proof}
 From the morphism of long exact sequences considered in the proof
of Proposition~\ref{comparison-fully-faithful} and the main
hypothesis for the exact category $\F_k^G{}^+$ one can see by
the way of the 5\+lemma that the map
$\Ext^n_{\F_k^G{}^+/s}({}_sX,{}_sY)\rarrow
\Ext^n_{\A_{k/s}^{G\.+}}({}_sX,{}_sY)$ is an isomorphism for $n<j-i$
and a monomorphism for $n=j-i$.
 Furthermore, since the map $\Ext^n_{\F_k^G{}^+}(X,Y)\rarrow
\Ext^n_{\A_k^G{}^+}(X,Y)$ is surjective by the main hypothesis,
surjectivity of the map $\Ext^n_{\A_k^G{}^+}(X,Y)\rarrow
\Ext^n_{\A_{k/s}^{G\.+}}({}_sX,{}_sY)$ implies surjectivity of
the map $\Ext^n_{\F_k^G{}^+/s}({}_sX,{}_sY)\rarrow
\Ext^n_{\A_{k/s}^{G\.+}}({}_sX,{}_sY)$ for $n=j-i$.

 By Proposition~\ref{diagonal-quadratic-prop}(b), the map
$\Ext^n_{\F_{k/s}^{G\.+}}({}_sX,{}_sY)\rarrow
\Ext^n_{\A_{k/s}^{G\.+}}({}_sX,{}_sY)$ is an isomorphism for
$n\le 1$, \,$n\le j-i$ and a monomorphism for $n=2$.
 From commutativity of the triangle diagram
$\Ext^n_{\F_k^G{}^+/s}({}_sX,{}_sY)\rarrow
\Ext^n_{\F_{k/s}^{G\.+}}({}_sX,{}_sY)\allowbreak\rarrow
\Ext^n_{\A_{k/s}^{G\.+}}({}_sX,{}_sY)$ we conclude
that the map $\Ext^n_{\F_k^G{}^+/s}({}_sX,{}_sY)\rarrow
\Ext^n_{\F_{k/s}^{G\.+}}({}_sX,{}_sY)$ is an isomorphism for $n\le1$
and a monomorphism for $n=2$ (and all $i$, $j\in\Z$).
 According to~\cite[Lemma~3.2]{Partin}, it follows that
the comparison functor $\kap\:\F_k^G{}^+/s\rarrow\F_{k/s}^{G\,+}$
is an equivalence of exact categories.
 Now we have proven that the map $\Ext^n_{\F_{k/s}^{G\.+}}({}_sX,{}_sY)
\rarrow\Ext^n_{\A_{k/s}^{G\.+}}({}_sX,{}_sY)$ is an isomorphism
for $n\le j-i$, that is the main hypothesis holds for
the exact category $\F_{k/s}^{G\,+}$.
\end{proof}

 In order to formulate our main conjectures, let us pass to
the following particular case of our general setting.
 Suppose that $k$ is a complete discrete valuation ring with
a uniformizing element $l\in k$.
 Let $c\:G\rarrow k^*$ be a continuous multiplicative character of
a profinite group~$G$ and $H\subset G$ be a closed normal subgroup
annihilated by the reduced character $c/l\:G\rarrow (k/l)^*$.

\begin{conj}  \label{comparison-conjecture}
 Suppose that the natural maps
$$
\Ext_{\A_{k/l^{r+1}}^{G\.+}}^n({}_{l^{r+1}}X,\.{}_{l^{r+1}}Y(n))\lrarrow
\Ext_{\A_{k/l^r}^{G\.+}}^n({}_{l^r}\!\.X,\.{}_{l^r}\!\.Y(n))
$$
are surjective for all\/ $r>0$ and $X$, $Y\in\E_{k,\,0}^{G/H\,+}$.
 Assume that the exact category $\F_{k/l}^{G\,+}$ satisfies
the main hypothesis.
 Then the comparison functors
$$
 \kap\:\F_k^G{}^+/l^r\lrarrow\F_{k/l^r}^{G\,+}
$$ 
are equivalences of exact categories for all $r>0$.
\end{conj}

\begin{conj}  \label{main-hypothesis-conjecture}
 Suppose that the natural maps
$$
\Ext_{\A_{k/l^{r+1}}^G}^n({}_{l^{r+1}}X,\.{}_{l^{r+1}}Y(n))\lrarrow
\Ext_{\A_{k/l^r}^G}^n({}_{l^r}\!\.X,\.{}_{l^r}\!\.Y(n))
$$
are surjective for all\/ $r>0$ and $X$, $Y\in\E_{k,\,0}^{G/H}$.
 Then the exact categories $\F_{k/l^r}^G$ satisfy the main hypothesis
whenever so does the exact category $\F_{k/l}^G$.
\end{conj}

 It is not difficult to deduce
Conjecture~\ref{main-hypothesis-conjecture} from
Conjecture~\ref{comparison-conjecture}.
 Their surjectivity assumptions are equivalent by
Lemma~\ref{large-and-small-unfiltered-lem}(a\+b), while
the Ext groups in the categories $\F_{k/l^r}^G$ and $\F_{k/l^r}^{G\,+}$
agree by Proposition~\ref{large-and-small-filtered-prop}(a).
 Assuming the comparison functors~$\kap$ to be equivalences,
the forgetful functor $\psi\:\F_k^G{}^+\rarrow\A_k^G{}^+$
induces a morphism between the long exact sequences
\begin{multline*}
 \Ext^n_{\F_{k/l^{r'}}^G}({}_{l^{r'}}\!\.X,\.{}_{l^{r'}}\!\.Y)
 \lrarrow\Ext^n_{\F_{k/l^r}^G}({}_{l^r}\!\.X,\.{}_{l^r}\!\.Y)\\
 \lrarrow\Ext^n_{\F_{k/l^{r''}}^G}({}_{l^{r''}}\!\.X,\.{}_{l^{r''}}\!\.Y)
 \lrarrow\Ext^{n+1}_{\F_{k/l^{r'}}^G}({}_{l^{r'}}\!\.X,\.{}_{l^{r'}}\!\.Y)
\end{multline*}
and
\begin{multline*}
 \Ext^n_{\A_{k/l^{r'}}^G}({}_{l^{r'}}\!\.X,\.{}_{l^{r'}}\!\.Y)
 \lrarrow\Ext^n_{\A_{k/l^r}^G}({}_{l^r}\!\.X,\.{}_{l^r}\!\.Y)\\
 \lrarrow\Ext^n_{\A_{k/l^{r''}}^G}({}_{l^{r''}}\!\.X,\.{}_{l^{r''}}\!\.Y)
 \lrarrow\Ext^{n+1}_{\A_{k/l^{r'}}^G}({}_{l^{r'}}\!\.X,\.{}_{l^{r'}}\!\.Y)
\end{multline*}
of Subsection~\ref{second-bockstein} or~\ref{third-bockstein} for
any positive integers $r'+r''=r$ and any objects $X\in\E_{k,\,i}^G$
and $Y\in\E_{k,\,j}^G$.
 Applying the 5\+lemma, the main hypothesis for $\F_{k/l^r}^G$
follows from that for $\F_{k/l}^G$ by trivial induction in~$r$.

\medskip
 In the situation relevant for Artin--Tate
motives~\cite[Subsections~9.2\+-9.5]{Partin}, the coefficient
ring $k=\Z_l$ is the ring of $l$\+adic integers, the profinite
group $G=G_K$ is the absolute Galois group of a field $K$ of
characteristic different from~$l$, the character~$c$ is
the cyclotomic one, and the closed normal subgroup $H\subset G$
corresponds to a Galois extension $M/K$ with the field $M$
containing a primitive $l$\+root of unity.
 The surjectivity assumption of the Conjectures is then
a reformulation of the Milnor--Bloch--Kato conjecture
(proven by Rost, Voevodsky, et.~al.~\cite{Voev}), while
the main hypothesis of Subsection~\ref{main-hypothesis} is
the silly filtration conjecture~\cite[Conjecture~9.2]{Partin}.

 The main hypothesis for the exact category $\F_{k/l}^G$ can be 
interpreted as a Koszulity condition by
Lemma~\ref{root-of-1-koszulity-lemma}.
 The described setting includes the case of $H=\{e\}$ that appears
in connection with Artin--Tate motivic sheaves~\cite{Pmotsh}.

\end{document}